\documentclass[12pt]{amsart}
\usepackage{graphicx}
\usepackage{amsmath}
\usepackage{amssymb}
\usepackage{amsthm}
\usepackage{tikz}
\usepackage{tikz-cd}
\usepackage{comment}
\usepackage[normalem]{ulem}
\usepackage{hyperref}
\usepackage[margin=3cm]{geometry}

\newtheorem{theorem}{Theorem}[section]
\newtheorem{lemma}[theorem]{Lemma}

\newtheorem{corollary}[theorem]{Corollary}
\newtheorem{proposition}[theorem]{Proposition}
\newtheorem{observation}[theorem]{Observation}

\theoremstyle{definition}
\newtheorem{definition}[theorem]{Definition}
\newtheorem{question}[theorem]{Question}
\newtheorem{example}[theorem]{Example}
\newtheorem{remark}[theorem]{Remark}
\newtheorem{problem}[theorem]{Problem}
\newtheorem{notation}[theorem]{Notation}
\numberwithin{equation}{section}
\numberwithin{equation}{section}

\newcommand{\ord}{{\rm ord}}
\newcommand{\sord}{{\rm sord}}
\newcommand{\im}{{\rm im}}
\newcommand{\htt}{{\rm ht}}

\newcommand{\iLim}{\varprojlim}

\newcommand{\treet}{\mathcal{T}_{\mathcal{M}\it 3}}
\newcommand{\treecon}{\mathcal{T}_\mathcal{C} }
\newcommand{\treend}{\mathcal{T}_\mathcal{CE} }
\DeclareMathOperator{\diam}{diam}

\begin{document}

\title[Projective Fra\"{\i}ss\'{e} limits of trees]{Projective Fra\"{\i}ss\'{e} limits of trees with confluent epimorphisms}

\author[W. J. Charatonik]{W\l odzimierz J. Charatonik}

\address{W. J. Charatonik and R. P. Roe, Department of Mathematics and Statistics\\
         Missouri University of Science and Technology\\
         400 W 12th St\\
         Rolla MO 65409-0020}
         \email{rroe@mst.edu}

\author[A. Kwiatkowska]{Aleksandra Kwiatkowska}
\address{A. Kwiatkowska, Institut f\"{u}r Mathematische Logik und Grundlagenforschung, Universit\"{a}t  M\"{u}nster,  
Einsteinstrasse 62,
48149  M\"{u}nster,
Germany {\bf{and}} 
Instytut Matematyczny, Uniwersytet Wroc{\l}awski,  pl. Grunwaldzki 2/4, 50-384 Wroc{\l}aw, Poland, }
\email{kwiatkoa@uni-muenster.de}

\author[R. P. Roe]{Robert P. Roe}

\author[S. Yang]{Shujie Yang}
\address{S. Yang, Institut f\"{u}r Mathematische Logik und Grundlagenforschung, Universit\"{a}t  M\"{u}nster,  
Einsteinstrasse 62,
48149  M\"{u}nster,
Germany }
\email{syang2@uni-muenster.de}

\thanks{A. K. and S. Y. were funded by the Deutsche Forschungsgemeinschaft (DFG, German Research Foundation) under Germany’s Excellence Strategy EXC 2044–390685587, Mathematics Münster: Dynamics–Geometry–Structure and by CRC 1442 Geometry: Deformations and Rigidity. Additionally, S. Y. was funded by China Scholarship Council(CSC) 202204910109.} 

\date{\today}


\begin{abstract}
We continue the study of projective Fra\"{\i}ss\'e limits developed by Irwin-Solecki and  Panagiotopoulos-Solecki by investigating families of epimorphisms between finite trees and finite rooted trees. Ideas of monotone, confluent, and light mappings from continuum theory as well as several properties of continua are modified so as to apply them to topological graphs. As the topological realizations of the projective Fra\"{\i}ss\'e limits we obtain the dendrite $D_3$, the Mohler-Nikiel universal dendroid, as well as new, interesting continua for which we do not yet have topological characterizations. 
\end{abstract}

\maketitle

\section{Introduction and Definitions}

\subsection{Introduction}

In \cite{Pseudo}, T. Irwin and S. Solecki introduced the idea of a projective Fra\"{\i}ss\'e limit as a dualization of the injective Fra\"{\i}ss\'e limit from model theory. In that paper they constructed the pseudo-arc as the topological realization of a projective Fra\"{\i}ss\'e limit of a certain family of finite graphs and epimorphisms between members of the class.  Subsequently, D. Barto\v sov\' a and A. Kwiatkowska, \cite{B-K}, W. Kubiś and A. Kwiatkowska, \cite {kubis}  and A. Panagiotopoulos and S. Solecki, \cite {Menger}, extended these ideas to study, repectively, the Lelek fan,   the Poulsen simplex, and the Menger curve as projective Fra\"{\i}ss\'e limits. Further recent results on the applications of the Fra\"{\i}ss\'e theory to continuum theory are in \cite{Bar-Kub}, \cite{B-C}, \cite{Co-Kw}, \cite{SI}, \cite{LW}.

A continuum is a compact connected (metric) space. In continuum theory, a continuous map $f\colon X\to Y$ between continua is monotone if preimages of subcontinua are subcontinua. 
It is confluent if for each subcontinuum $Q$ of $Y$ each component of $f^{-1}(Q)$ is mapped onto $Q$
by $f$. It is light if the preimage of each point is discrete. We adapt those ideas of monotone, confluent, and light mappings from continuum theory as well as several properties of continua (such as hereditary indecomposability, dendroids, dendrites) so as to apply them to topological graphs.

In the article we undertake a systematic study of graphs, trees, and rooted trees with monotone, light, and confluent epimorphisms. We develop tools for studying projective Fra\"{\i}ss\'e limits of families of such epimorphisms and continua obtained as their topological realizations.

In Section 3, we first notice that all finite trees with monotone maps do not amalgamate (Example \ref{no-4od}). However, when we put restrictions on the order of vertices we obtain the well-known Ważewski dendrite $D_3$. The main results of  Section 3 are Theorems \ref{amalgamation-monotone} and \ref{D3}.
\begin{theorem}
The topological realization of the projective Fra\"{\i}ss\'e limit of finite trees having all vertices of order less than or equal to 3 with monotone epimorphisms is the Ważewski dendrite $D_3$
\end{theorem}

The main result of Section 4 is:
\begin{theorem}
Finite connected graphs with confluent epimorphisms form a projective Fra\"{\i}ss\'e family. 
\end{theorem}
The topological realization of the obtained projective Fra\"{\i}ss\'e limit as far as we know has not been investigated in continuum theory, before this work. A detailed study of that continuum is in the continuation of this article \cite{CKR}.
In contrast, finite connected graphs with monotone epimorphisms also form a projective Fra\"{\i}ss\'e family, and the topological realization of the limit is the well known Menger universal curve \cite {Menger}.

In Sections 5-10, we investigate in detail rooted trees with confluent epimorphism. Inverse limits of continua that are trees with confluent bonding maps were studied, for example, in \cite{CCP}, \cite{OP}, \cite{CP}. Our investigations are most importantly motivated by the universal Molher-Nikiel dendroid  \cite{Dendroid}   constructed as the inverse limit of continua that are rooted trees with confluent bonding maps, see Section~\ref{M-NSec}. 
The Mohler-Nikiel dendroid is universal for the class of smooth dendroids. It has a closed
set of endpoints and is such that each point is either an endpoint or a ramification point.  It is proved in \cite{CCP} that a dendroid is an absolute retract for hereditarily unicoherent continua if and only if it can be embedded as a retract into the Mohler-Nikiel universal dendroid.

Rooted trees with all monotone epimorphisms, unlike connected graphs with all monotone epimorphisms, do not amalgamate (see Example \ref{not-order-pres}). Any confluent epimorphism between rooted trees can be written as a composition of monotone and light confluent epimorphisms (see Theorem \ref{m-l-factorization} and Remark \ref{m-l-factorization-rem}). Furthermore, any light confluent epimorphism can be further written as a composition of elementary light confluent epimorphisms (see Definition \ref{def:elc} and Corollary \ref{elemenlightc}). Light confluent epimorphisms between rooted trees do amalgamate, but inverse limits in this class  
will have trivial topological realizations (the transitive closure of the edge relation will consist of a single equivalence class).  It turns out that by taking all light confluent epimorphisms, but only some monotone ones, we obtain two interesting projective Fra\"{\i}ss\'e families. One of them has the Mohler-Nikiel continuum  as the topological realization, and the other one, which we denote by $|\mathbb{T}_{\mathcal{ C}}|$, we believe is different from any known continua.

The majority of the exotic continua that continuum theorists have discovered result from a researcher constructing a continuum to answer a specific question. For example the famous pseudo-arc was constructed by Knaster, \cite{Knaster}, to show that there does exist a continuum with the property that every non-degenerate subcontinuum is indecomposable.  Of course, it turned out that this continuum has many other interesting properties and it is still the subject of continued research. As another example, the Mohler-Nikiel universal smooth dendroid discussed in this article was constructed to show the existence of a dendroid having a closed set of endpoints and all other points being ramification points, answering a question of Krasinkiewicz \cite{Dendroid}. On the other hand, some continua occur naturally, for example, the Mandelbrot set, which originates from a dynamical system and has also generated much research. The continuum $|\mathbb{T}_{\mathcal{ C}}|$ of this article is in this latter group of continua, it arises naturally out of the study of projective Fraïssé families.

An epimorphism $f\colon S\to T$ is simple confluent if the rooted tree $S$ is obtained from the rooted tree $T$ by repeating the following 2 operations:  take an edge and split it into two edges (simple-monotone epimorphism) or take a vertex and a component above it and double it (elementary light confluent epimorphism). An epimorphism $f\colon S\to T$ is simple*-confluent if $S$ is obtained from $T$ by repeating the following 3 operations:  the two operations as above, and additionally we can pick a vertex and attach a new edge to it (see Definitions \ref{def:sc} and \ref{def:sstarc}).

We summarize the main results of Sections 5-10 below.
\begin{theorem}
\begin{enumerate}
\item If $A,B,C$ are rooted trees, $f\colon B\to A$ is light confluent, and $g\colon C\to A$ is confluent, then there is a rooted tree $D$, confluent $f_0\colon D\to B$, and light confluent $g_0\colon D\to C$ with $f\circ f_0=g\circ g_0$ (Theorem \ref{onelight}).
\item Simple confluent epimorphisms form a projective Fra\"{\i}ss\'e family (denoted later as $\treend$, see Theorem \ref{confluent-projective-family}).
\item Simple*-confluent epimorphisms form a projective Fra\"{\i}ss\'e family (denoted later as $\treecon$, see Theorem \ref{confluent-projective-family2}).
\end{enumerate}
\end{theorem}

In Theorems \ref{simple=special} and  \ref{simple=specialstar} we obtain intrinsic characterisations of simple confluent and simple*-confluent epimorphisms. They are interesting on their own, but we will also need them to show that epimorphisms between topological graphs, obtained as  inverse limits,  in
 $\treend$  or in $\treecon$ are well defined.

Main continuum theory consequences are as follows.

\begin{theorem}
\begin{enumerate}
\item The topological realization of the projective Fra\"{\i}ss\'e limit of $\treend$ is the Mohler-Nikiel universal dendroid (Theorem \ref{toprealMN})
\item The topological realization $|\mathbb{T}_{\mathcal{ C}}|$ of the projective Fra\"{\i}ss\'e limit of $\treecon$  is a dendroid
with a dense
set of endpoints and such that each point is either an endpoint or a ramification point (Theorem 
\ref{summerizing-confluent}).
\end{enumerate}

\end{theorem}

\subsection{Background}

We start with some basic definitions and background about projective Fra\"{\i}ss\'e families and limits.

A {\it graph} is an ordered pair $A=( V(A), E(A))$, where $E(A)\subseteq V(A)^2$ is a reflexive and symmetric relation on $V(A)$. The elements of $V(A)$ are called {\it vertices} of graph $A$ and elements of $E(A)$ are called {\it edges}. 
Given vertices $a$ and $b$ in a graph $A$, where $a \not =b$, we will use  $\langle a, b\rangle$ to denote an edge in $A$. Note that since $E(A)$ is reflexive there is always an edge from a vertex in $A$ to itself, such an edge will be called a {\it degenerate} edge.  When it is clear from context that $x$ is a vertex we will use $ x\in A$ for $x\in V(A)$.

Given two graphs $A$ and $B$ a function $f\colon V(A)\to V(B)$ is a {\it homomorphism} if it maps edges to edges, i.e.  $\langle a,b\rangle \in E(A)$ implies $\langle f(a),f(b)\rangle \in E(B)$.  
A homomorphism $f$ is an {\it epimorphism} between graphs if it is moreover surjective
on both vertices and edges. An {\it isomorphism} is an injective epimorphism.

By a {\it rooted graph} we mean a graph $G$ with a distinguished vertex $r_G$. 
 An epimorphism $f\colon G \to H$, between rooted graphs, is an epimorphism between corresponding (unrooted) graphs, such that $f(r_G) = r_H$.

\begin{definition}
A {\it topological graph} ({\it rooted topological graph}) $K$ is a graph $( V(K),E(K))$ (respectively, rooted graph $( V(K),E(K), r_K)$, where $r_K$ is the distinguished vertex), whose domain $V(K)$ is a 0-dimensional, compact, second-countable (thus metrizable) space and $E(K)$ is a closed, reflexive and symmetric  subset of $ V(K)^2$. 
We require that all epimorphisms between topological graphs are continuous.

\end{definition}
Topological graphs and rooted topological graphs are examples of topological $\mathcal{L}$-structures. For a general definition of a topological $\mathcal{L}$-structure see \cite{Pseudo}.

\begin{definition}\label{definition-order}
A {\it tree} $T$ is a finite graph  such that for every two distinct vertices $a,b \in T$ there is a unique finite sequence
$v_0=a,v_1, \dots , v_n=b$ of vertices in $T$ such that for every $i\in \{0,1,\dots, n-1\}$ we have $\langle v_i,v_{i+1}\rangle\in E(T)$
and $v_j\ne v_i$ for $j\neq i$. 
Let $n$ be a natural number, a vertex $p\in T$ has {\it order}  $n$ ($\ord(p)=n$) if there are $n$ non-degenerate edges in $T$ that contain $p$. If $\ord(p)=1$ then $p$ is an {\it end vertex}, $\ord(p)=2$ then $p$ is an {\it ordinary vertex}, and if $\ord(p)\ge 3$ then $p$ a {\it ramification vertex}. 

\end{definition}

The theory of projective Fra\"{\i}ss\'e limits was developed in \cite{Pseudo} and further refined in \cite{Menger}. We literally recall their definitions here. We will restrict to the case of finite (rooted) graphs.

\begin{definition}\label{definition-Fraisse}
Let $\mathcal{F}$ be a nonempty family of  finite graphs  or a family of finite rooted graphs with a fixed family of epimorphisms among
the structures in $\mathcal{F}$.  We say that $\mathcal{F}$ is a projective Fra\"{\i}ss\'e family if
\begin{enumerate}
\item $\mathcal{F}$ is countable up to isomorphism, that is, any sub-collection of pairwise
non-isomorphic structures of $\mathcal{F}$ is countable;
\item epimorphisms in $\mathcal{F}$ are closed under composition and each identity map is in $\mathcal{F}$;
\item for $B,C \in  \mathcal{F}$ there exist $D\in \mathcal{F}$ and epimorphisms $f$ and $g$ in $\mathcal{F}$ such that $f\colon  D \to B$ and $g\colon  D \to C$;
and
\item for every two epimorphisms $f\colon B \to A$ and $g\colon C \to A$ in $\mathcal{F}$ there exist epimorphisms 
$f_0\colon D \to B$ and $g_0\colon D\to C$ in $\mathcal{F}$ such that  $f \circ f_0 = g \circ g_0$, i.e. the diagram (D1) commutes.
\end{enumerate}
\begin{equation}\tag{D1}
\begin{tikzcd}
&B\arrow{ld}[swap]{f}\\
A&&D\arrow[lu,swap,dotted,"f_0"] \arrow[ld,dotted,"g_0"]\\
&C\arrow[lu,"g"]
\end{tikzcd}
\end{equation}
We will refer to property (3) as the joint projection property (JPP) and property (4) as the projective amalgamation property (AP).
\end{definition}

\begin{notation}
Given a sequences $F_n$ of finite (rooted) graphs and epimorphisms $\alpha_n\colon F_{n+1} \to F_n$, which are called bonding maps, we denote the inverse sequence by $\{F_n,\alpha_n\}$. For $m > n$ we let $\alpha_n^m= \alpha_n\circ \dots \circ \alpha_{m-1}$ and note that $\alpha_n^{n+1} = \alpha_n$. Given an inverse sequence $\{F_n,\alpha_n\}$ the associated inverse limit space is the subspace of the product space $\Pi F_i$ determined by $\{(x_1,x_2,\ldots) : x_i \in F_i \text{ and } x_i = \alpha_i(x_{i+1})\}$ and is denoted as ${\bf F}=\iLim\{F_n,\alpha_n\}$.   Further, we denote the canonical projection from the inverse limit space ${\bf F}$ to the $n$th factor space $F_n$ by $\alpha_n^{\infty}$. Note that if $x,y \in \iLim\{F_n,\alpha_n\}$ then the metric $d(x,y) = \sum_{i=1}^\infty d_i(x_i,y_i)/2^i$, where $d_i$ is the discrete metric on $F_i$, is compatible with the topology of the product space. So, in particular, if $a \in F_n$ then 
$\diam((\alpha^\infty_n)^{-1}  (a))\le 1/2^{n-1}$.
\end{notation}

The family $\mathcal F$ of finite graphs or finite rooted graphs and epimorphisms is enlarged to a family $\mathcal F^\omega$ which includes all (rooted) topological graphs obtained as inverse limits of (rooted) graphs in $\mathcal F$ with bonding maps from the family of epimorphisms.  If ${\bf G}=\iLim\{G_n,\alpha_n\} \in \mathcal F^\omega$ and $a=(a_n)$ and $b=(b_n)$ are elements of ${\bf G}$ then $\langle a,b\rangle$ is an edge in ${\bf G}$ if and only if for each $n$, $\langle a_n, b_n\rangle$ is an edge in $G_n$. An epimorphism $h$ between a (rooted) topological graph ${\bf G}=\iLim\{G_n,\alpha_n\}$ in $\mathcal F^\omega$ and  $A\in \mathcal F$ is in the family $\mathcal F^\omega$ if and only if there is an $m$ and an epimorphism  $h'\colon G_m \to A$, $h'\in \mathcal F$, such that $h= h'\circ \alpha_m^\infty$. Finally, if $K$ and $L$ are in $\mathcal F^\omega$ an epimorphism $h\colon L \to K$ is in the family $\mathcal F^\omega$ if and only if for any  $A \in \mathcal F$ and any epimorphism $g\colon K \to A$ in $\mathcal F^\omega$, $g\circ h \in \mathcal F^\omega$. 

\begin{lemma}\label{ladder}
Let $\{F_n,\alpha_n\}$ be an inverse sequence of finite (rooted) graphs, where each $\alpha_n$ is an epimorphism. Let ${\bf F}$ denote its inverse limit.
\begin{enumerate}
\item If $A$ is a finite (rooted) graph and $f\colon {\bf F}\to A$ is an epimorphism, then there is $n$ and an epimorphism $\gamma\colon F_n\to A$ such that $f=\gamma\circ\alpha^\infty_n$.
\item Suppose that $\{G_m,\beta_m\}$ is another inverse sequence of finite (rooted) graphs, where each $\beta_m$ is an epimorphism, such that ${\bf F}$ is isomorphic to its inverse limit.
Then there are subsequences $(n_i)$ and $(m_j)$ and there are epimorphisms $g_i\colon F_{n_i}\to G_{m_i}$ and $h_i\colon G_{m_{i+1}}\to F_{n_i}$ such that
$\alpha^{n_{i+1}}_{n_i}=h_ig_{i+1}$ and $\beta^{m_{i+1}}_{m_i}=g_ih_i$.
\end{enumerate}

\end{lemma}

\begin{proof}
For (1), take any $n$ such that any element of the partition $\{(\alpha^\infty_n)^{-1}(x)\colon x\in F_n\}$ is contained in some element of the partition $\{f^{-1}(x)\colon x\in A\}$. 
For $z\in F_n$ we let $\gamma(z)=f(t)$ for any $t\in (\alpha^\infty_n)^{-1}(z)$. 
Then $f=\gamma\circ\alpha^\infty_n$. Since $f$ and $\alpha^\infty_n$ are epimorphism, by Lemma 2.1 in \cite{Pseudo}, $\gamma$ is an epimorphism.
For (2), if we already constructed $(g_i)$ and $(h_i)$ for $i<i_0$, to get $g_{i_0}$ apply part (1) to $f_1=\beta^\infty_{m_{i_0}}$ and get $n_{i_0}$ and $g_{i_0}$ such that
$f_1=g_{i_0}\circ\alpha^\infty_{n_{i_0}}$. Next, we apply part (1) to $f_2=\alpha^\infty_{n_{i_0}}$ and get $m_{{i_0+1}}$ and $h_{i_0}$ such that 
$f_2=h_{i_0}\circ\beta^\infty_{m_{i_0+1}}$. Those $g_i$ and $h_i$ are as required.
\end{proof}

We need conditions that imply whether or not an epimorphism $h\colon {\bf F}\to A$ belongs to $\mathcal{F}^\omega$
does not depend on the choice of the inverse limit representation of ${\bf F}$.
Note that if ${\bf F}=\iLim\{F_n,\alpha_n\}$, then each $\alpha^\infty_n$ is in $\mathcal{F}^\omega$.

Lemma \ref{ladder} motivates the following definition.
Let $\mathcal F$ be a family of finite (rooted) graphs with a possibly restricted family of epimorphisms. We say that $\mathcal{F}$ is {\it consistent} if for any topological graph ${\bf F}$ in $\mathcal F^\omega$, which can be represented by isomorphic inverse limits $\iLim \{F_n,\alpha_n\}$ and $\iLim \{G_n,\beta_n\} $, where $\alpha_n, \beta_n\in\mathcal F$, there are subsequences $(n_i)$ and $(m_j)$ and  epimorphisms $g_i\colon F_{n_i}\to G_{m_i}$, $h_i\colon G_{m_{i+1}}\to F_{n_i}$ with $h_i, g_i\in\mathcal F$ such that
$\alpha^{n_{i+1}}_{n_i}=h_ig_{i+1}$, $\beta^{m_{i+1}}_{m_i}=g_ih_i$,
$\beta^\infty_{m_{i}}=g_{i}\circ\alpha^\infty_{n_{i}}$,
and $\alpha^\infty_{n_{i}}=h_{i}\circ\beta^\infty_{m_{i+1}}$.

The consistency of $\mathcal{F}$ guarantees that the definition of an epimorphism in $\mathcal{F}^\omega$ does not depend on the inverse limit representation of a structure. Indeed, suppose that $\iLim \{F_n,\alpha_n\}$ and $\iLim \{G_n,\beta_n\} $ are isomorphic, where $\alpha_n, \beta_n\in\mathcal F$, and let
 $h\colon \iLim\{F_n,\alpha_n\}\to A$ be an epimorphims in $\mathcal{F}^\omega$.
  Let $m$ and
$h'\colon F_m \to A$ with $h'\in \mathcal F$, be such that $h= h'\circ \alpha_m^\infty$. 
Take any $n_i\geq m$  and consider $h_i\colon G_{m_{i+1}}\to F_{n_i}$. Then $h= h'\circ \alpha^{n_i}_m\circ h_i\circ \beta^\infty_{m_{i+1}}$ and $h'\circ \alpha^{n_i}_m\circ h_i\in\mathcal F$.
Moreover, the isomorphism induced by $(g_i)$ (equivalently, by $(h_i)$) is the identity.

\begin{corollary}\label{uniqueness}
Let $\mathcal F$ be a  family of epimorphisms between 
finite (rooted)  graphs.
    Suppose that for all epimorphisms  $f$, $g$, and $\phi$  such that $f\circ \phi=g$  we have that if   $g\in \mathcal{F}$, then $f\in\mathcal F$.
    Then $\mathcal{F}$ is consistent.
\end{corollary}
\begin{proof}
Let ${\bf F}$ be represented as $\iLim\{F_n,\alpha_n\} $ and $\iLim \{G_n, \beta_n\}$, where $\alpha_n, \beta_n\in\mathcal F$.
It suffices to check that maps $g_i$ and $h_i$ constructed in Lemma \ref{ladder} are in $\mathcal F$. 
By our assumption, if $\alpha^{n_{i+1}}_{n_i}=h_ig_{i+1}$,
then since $\alpha^{n_{i+1}}_{n_i}$ is in $\mathcal F$, we get that $h_i$ is in $\mathcal F$. Similarly, if 
$\beta^{m_{i+1}}_{m_i}=g_ih_i$, then since $\beta^{m_{i+1}}_{m_i}$ is in  $\mathcal F$, we have that $g_i$ is in $\mathcal F$.

\end{proof}

In Section 8, we will need the following strengthening of Corollary \ref{uniqueness}.
\begin{corollary}\label{uniquegeneral}
Let  $\mathcal F_0$ and  $\mathcal F$ be families 
of epimorphisms between finite (rooted) graphs such that 
$\mathcal F_0\supseteq \mathcal F$. Suppose that for all epimorphisms $f$, $g$, and $\phi$  such that $f\circ \phi=g$  we have: (1) if   $g\in \mathcal{F}_0$, then $f\in\mathcal F_0$,
and (2) if $\phi\in\mathcal F_0$ and $g\in \mathcal{F}$, then $f\in\mathcal F$.  Then $\mathcal{F}$ is consistent.

\end{corollary}

\begin{proof}
Let ${\bf F}$ be represented as $\iLim\{F_n,\alpha_n\} $ and $\iLim \{G_n, \beta_n\}$, where $\alpha_n, \beta_n\in\mathcal F$.
It suffices to check that maps $g_i$ and $h_i$ constructed in Lemma \ref{ladder} are in $\mathcal F$. Since $\mathcal{F}\subseteq \mathcal{F}_0$, as in the proof of  Corollary \ref{uniqueness}, by (1), we obtain that  for all $i$, $g_i, h_i\in \mathcal{F}_0$.
Since
$\alpha^{n_{i+1}}_{n_i}=h_ig_{i+1}$ and 
$\beta^{m_{i+1}}_{m_i}=g_ih_i$, for all $i$,
by  (2), we get in fact $g_i, h_i\in\mathcal F$.    
\end{proof}

In the proof of \cite[Theorem 2.4]{Pseudo} an inverse sequence satisfying certain properties, see definition below, was introduced and used to show the existence and uniqueness of the projective Fra\"{\i}ss\'e limit.

\begin{definition}\label{fund-seq-def}

Given a projective Fra\"{\i}ss\'e family $\mathcal F$ an inverse sequence $\{F_n,\alpha_n\}$ where $F_n\in {\mathcal F}$ and $\alpha_n\colon F_{n+1} \to F_n$ are epimorphisms in $\mathcal F$ is said to be a {\it Fra\"{\i}ss\'e sequence} for $\mathcal F$ if the following two conditions hold.

\begin{enumerate}
    \item For any $G\in {\mathcal F}$ there is an $n$ and an epimorphism in $\mathcal F$ from $F_n$ onto $G$;
    \item For any $n$, any pair $G, H \in {\mathcal F}$, and any epimorphisms $g\colon H \to G$ and $f\colon F_n \to G$ in $\mathcal F$ there exists $m >n$ and an epimorphism $h\colon F_m \to H$ in $\mathcal F$ such that $g\circ h = f\circ \alpha_n^m$  i.e the diagram (D2) commutes.
\end{enumerate}

\begin{equation}\tag{D2}
\begin{tikzcd}
F_n\arrow{d}[swap]{f}&F_m\arrow[d,dotted,"h"]\arrow{l}[swap]{\alpha^m_n} \\
G&\arrow[l,"g"]H
\end{tikzcd}
\end{equation}

\end{definition}

Note that the name Fra\"{\i}ss\'e sequence has not been standardized. Other names that have been used include generic sequence and fundamental sequence. Let us also remark, and we will use it in later sections, that we have the following characterization of a Fra\"{\i}ss\'e sequence.

\begin{proposition}\label{Fraisse equiv}
    Let  $\mathcal F$ be a projective Fra\"{\i}ss\'e family.  An inverse sequence $\{F_n,\alpha_n\}$, where $F_n\in {\mathcal F}$ and $\alpha_n\colon F_{n+1} \to F_n$ are epimorphisms in $\mathcal F$, 
    is  a  Fra\"{\i}ss\'e sequence if and only if for each $m$, $A\in \mathcal F$ and an epimorphism $\phi\colon A\to F_m$ in $\mathcal F$ there exist $n$ and an epimorphism $\psi\colon F_n\to A$ in $\mathcal F$ such that $\phi\circ \psi=\alpha^n_m$.
\end{proposition}
\begin{proof}
Clearly a Fra\"{\i}ss\'e sequence has the property stated in the proposition; just take $f={\rm Id}$. The converse is proved in \cite[Proposition 2.3]{B-C}.
\end{proof}

An epimorphism $f \colon A \to B$ between (rooted) topological graphs is said to be an {\it $\varepsilon$-epimophism} if for every $x \in B$, $\diam(f^{-1}(x))<\varepsilon $.

The following theorem was proved in \cite{Pseudo} for the case when we take all epimorphisms, and then refined to the version below in \cite{Menger}.

\begin{theorem}\label{limit}
 Let $\mathcal {F}$ be a projective Fra\"{\i}ss\'e family with a fixed collection of epimorphisms between the structures of $\mathcal{F}$.  Then there exists a unique topological graph (rooted topological graph) $\mathbb{F}$ in $\mathcal{F}^\omega$ such that
\begin{enumerate}
    \item for each $A\in \mathcal{F}$, there exists an epimorphism in $\mathcal{F}^\omega$
    from $\mathbb{F}$ onto~$A$;
    \item for $A,B \in \mathcal{F}$ and epimorphisms $f\colon \mathbb{F} \to A$ and $g\colon B\to A$
    in $\mathcal{F}^\omega$
    there exists an epimorphism $h\colon \mathbb{F}\to B$
  in $\mathcal{F}^\omega$
    such that $f=g\circ h$.
    \item\label{refinement}
For every $\varepsilon>0$  there is $G\in \mathcal F$ and an epimorphism $f\colon \mathbb{F} \to G$  in $\mathcal{F}^\omega$
such that $f$ is an $\varepsilon$-epimophism.
\end{enumerate}
\end{theorem}

\begin{proof}
The uniqueness of $\mathbb F$ along with the first two conditions are precisely \cite[Theorem 3.1]{Menger}. For the third condition note that in the proof of \cite[Theorem 3.1]{Menger} it is shown that the projective Fra\"{\i}ss\'e limit $\mathbb F$ is the inverse limit of a Fra\"{\i}ss\'e sequence for $\mathcal F$.  Condition~(3) then follows from the definition of a metric on the inverse limit space.  
\end{proof}

The following follows from \cite[Theorem 2.4]{Pseudo}.

\begin{theorem}\label{fund-sequence}
Given a projective Fra\"{\i}ss\'e family $\mathcal F$, there exists a Fra\"{\i}ss\'e sequence $\{F_i,g_i\}$ for $\mathcal F$, and the inverse limit of any Fra\"{\i}ss\'e sequence for $\mathcal F$ is isomorphic to the projective Fra\"{\i}ss\'e limit $\mathbb F$.
\end{theorem}

\begin{definition}
Given a (rooted) topological graph ${\bf K}$, if $E({\bf K})$ is also transitive then it is an equivalence relation and ${\bf K}$ is known as a {\it prespace}. The quotient space ${\bf K}/E({\bf K})$ is called the {\it topological realization} of ${\bf K}$ and is denoted by $|{\bf K}|$.
\end{definition}

\begin{theorem}
Each compact metric space is a topological realization of a topological graph.
\end{theorem}
\begin{proof}
Let $X$ be a compact metric space and let $f\colon C\to X$ be a surjective mapping from the Cantor set $C$. Define a topological graph ${\bf K}$ by putting $V({\bf K})=C$ and $\langle a,b\rangle\in E({\bf K})$ if and only if $f(a)=f(b)$. Then ${\bf K}$ is a compact topological graph, $E({\bf K})$ is transitive, and $|{\bf K}|$ is homeomorphic to $X$.
\end{proof}

\section{Connectedness properties of topological graphs}

In this section we propose definitions of some connectedness properties of topological graphs analogous to
respective definitions for continua. We have decided to keep the terminology original to continuum theory.
Let us start with the definitions of connected and locally connected topological graphs as in \cite{Menger}.

\begin{definition}
  Given a topological graph $G$, a subset $S$ of $V(G)$ is {\it disconnected} if there are two nonempty disjoint closed subsets $P$ and $Q$ of $S$ such that $P\cup Q=S$ and if
  $a\in P$ and $b\in Q$, then $\langle a,b\rangle\notin E(G)$. A subset $S$ of $V(G)$ is {\it connected} if it is not disconnected.
  A graph $G$ is connected if $V(G)$ is connected.

\end{definition}

\begin{definition}
Let $F$ be a  topological graph, $A$ and $B$ be subgraphs of $F$ and $a,b\in F$.
 We say that $A$ is {\it adjacent} to $B$ if $A \cap B = \emptyset$ and there are $x\in A$ and $y \in B$ such that 
 $\langle x, y \rangle \in E(F)$. 
 We say that $a$ is {\it adjacent} to $A$ if $\{a\}$ is adjacent to $A$,  and similarly, $a$ is {\it adjacent} to $b$ if $\{a\}$ is adjacent to $\{b\}$. Two subgraphs are {\it non-adjacent} if they are not adjacent.
\end{definition}

\begin{proposition}\label{connected-image}
If $f\colon G\to H$ is an epimorphism between topological graphs and $G$ is connected, then $H$ is connected.
\end{proposition}
\begin{proof}
Suppose $H$ is disconnected. Then there are two disjoint nonempty closed subsets $A$ and $B$ of $V(H)$ such that $V(H)=A\cup B$ and there are no edges between $A$ and $B$.
So $V(G)=f^{-1}(A)\cup f^{-1}(B)$ and $f^{-1}(A)\cap f^{-1}(B)=\emptyset$. Since $G$ is connected, there are vertices $a\in f^{-1}(A)$ and $b\in f^{-1}(B)$ such that
$\langle a,b\rangle\in E(G)$, and thus $\langle f(a),f(b)\rangle\in E(H)$, a contradiction.
\end{proof}

\begin{definition}
Given a topological graph $G$, a subset $S$ of $V(G)$, and a vertex $a\in S$, the {\it component of $S$ containing $a$} is the
largest connected subset $C$ of $S$ that contains $a$; in other words $C=\bigcup\{P\subseteq S\colon a\in P\ \text{ and }P \text{ is connected} \} $.
\end{definition}

\begin{definition}
A continuum $K$ is {\it hereditarily unicoherent} if for every two nonempty subcontinua $P$ and $Q$ of $K$ the
 intersection $P\cap Q$ is connected. The continuum $K$ is {\it unicoherent} if in addition $P \cup Q = K$.

 A topological graph $G$ is {\it hereditarily unicoherent} if for every two nonempty closed connected topological graphs, $P$ and $Q$, with $V(P) \subseteq V(G)$, $V(Q) \subseteq V(G)$, $E(P) \subseteq E(G)$, and $E(Q) \subseteq E(G)$, the intersection $P\cap Q=(V(P)\cap V(Q), E(P)\cap E(Q))$ is connected. The topological graph $G$ is {\it unicoherent} if in addition $V(P)\cup V(Q)=V(G)$.
 \end{definition}

The reason why we allow $E(P) \neq E(G)\cap V(P)^2$, and $E(Q) \neq E(G)\cap V(Q)^2$ in the definition of hereditary unicoherence, is that, for example, we do not want complete topological graphs to be dendrites (see Definition \ref{def:dendrite}).

\begin{notation}

If $A$ and $B$ are topological graphs by $A \subseteq B$ we mean $V(A) \subseteq V(B)$ and 
$\langle x,y\rangle \in E(A)$ if and only if $\langle x,y\rangle \in E(B)$, for $x,y\in V(A)$. If $A \subseteq B$ by $B\backslash A$ we mean a topological graph  $F \subseteq A$ such that $V(F)=V(B)\backslash V(A)$.

\end{notation}

\begin{definition}\label{arc-def}
  We say that a topological graph $G$ is an {\it arc} if it is connected and there are two vertices $a,b\in V(G)$ such that for every
  $x\in V(G)\setminus \{a,b\}$ the graph $G\setminus \{x\}$ is not connected. The vertices $a$ and $b$ are called {\it end vertices} of the arc and we say
  that $G$ is {\it joining} $a$ and~$b$. We denote the arc having end vertices $a$ and $b$ by $[a,b]$.
\end{definition}

In topology, there is only one, up to homeomorphism, metric arc, while for topological graphs we have finite arcs, countable arcs as well as uncountable arcs.

\begin{example}
  Let $G$ be the topological graph whose set of vertices is the Cantor ternary set, and the edges are defined by $\langle a,b\rangle\in E(G)$
  if and only if $a=b$ or $a$ and $b$ are end vertices of the same deleted interval. Then $G$  is an uncountable arc, whose topological realization is homeomorphic to $[0,1]$.
\end{example}

\begin{definition}
A topological graph $G$ is {\it locally connected} if it admits at every vertex a neighborhood base consisting of open connected sets.

A topological graph $G$ is called {\it arcwise connected} if for every two vertices $a,b\in V(G)$ there is an arc in $G$ joining $a$ and $b$.
  \end{definition}

\begin{definition}\label{def:dendrite}
A continuum that is hereditarily unicoherent and arcwise connected continuum is called a {\it dendroid}. A locally connected dendroid is a {\it dendrite}.

     A hereditarily unicoherent and arcwise connected topological graph is a {\it dendroid}. A locally connected dendroid is a {\it dendrite}.

\end{definition}

\begin{proposition}
  For a connected finite graph $G$ the following conditions are equivalent:
  \begin{enumerate}
    \item $G$ is a tree;
    \item  $G$ is unicoherent;
    \item $G$ is a dendroid;
    \item $G$ is a dendrite;
  \end{enumerate}
\end{proposition}
\begin{proof}
  The implications $(1)\implies(2)\implies(3)\implies(4)$ follow from definitions. To see $(4)\implies(1)$ note that a finite dendrite cannot contain a cycle, so it is a tree.
\end{proof}

The following observations use the fact that if a subset of a topological graph ${\bf G}$
is connected then its topological realization, viewed as a subspace of $|{\bf G}|$,
is connected. Moreover, if $\pi\colon {\bf G}\to |{\bf G}|$ is the quotient map and $A\subseteq |{\bf G}|$ is connected, then $\pi^{-1}(A)$ is connected.

\begin{observation}
If ${\bf G}$ is a hereditarily unicoherent topological graph and $E({\bf G})$ is transitive, then its topological realization $|{\bf G}|$ is a hereditarily unicoherent continuum.
\end{observation}

The quotient map takes arcs to arcs, therefore we have the following observation (compare with Lemma \ref{images-of-arcs}). 
\begin{observation}\label{arcs}
If a topological graph ${\bf G}$ is an arc and $E({\bf G})$ is transitive, then its topological realization $|{\bf G}|$ is an arc or a point.
\end{observation}

\begin{observation}
If a topological graph ${\bf G}$ is arcwise connected and $E({\bf G})$ is transitive, then its topological realization $|{\bf G}|$ is arcwise connected.
\end{observation}

\begin{observation}\label{realization-dendroid}
If a topological graph ${\bf G}$ is a dendroid and $E({\bf G})$ is transitive, then its topological realization $|{\bf G}|$ is a dendroid.
\end{observation}

\begin{observation}\label{top-dendrite}
If a topological graph ${\bf G}$ is a dendrite and $E({\bf G})$ is transitive, then its topological realization $|{\bf G}|$ is a dendrite.
\end{observation}
\begin{proof}
We show that if ${\bf G}$ is locally connected, then so is $|{\bf G}|$.
   Let $p\in U\subseteq |{\bf G}|$, where $U$ is open, and let $\pi\colon {\bf G}\to |{\bf G}|$ be the quotient map. We claim  there is an open connected ${\bf V} \subseteq {\bf G}$ such that $\pi^{-1}(p)\subseteq {\bf V}\subseteq \pi^{-1}(U)$.  To obtain such a ${\bf V}$ simply take the union of open connected neighborhoods of vertices of $\pi^{-1}(p)$, contained in $\pi^{-1}(U)$. This union is connected as $\pi^{-1}(p) $ is connected. Then $\pi({\bf V})$ is a connected neighborhood of $p$. Indeed, it contains an open set  $U'=\{x\in |{\bf G}|\colon \pi^{-1}(x)\subseteq {\bf V}\}$ and $p\in U'$. 
   This implies local connectedness of $|{\bf G}|$ (see for example Nadler~\cite{Nadler}, Exercise 5.22).
  
\end{proof}

The following theorem provides a sufficient condition  for a projective Fra\"{\i}ss\'e family to have a transitive set of edges in the projective Fra\"{\i}ss\'e limit.

\begin{theorem}\label{transitive}
Suppose that $\mathcal G$ is a projective Fra\"{\i}ss\'e family of graphs and for every $G\in \mathcal G$, for every $a,b,c\in V(G)$ such that
$\langle a,b\rangle\in E(G)$ and $\langle b,c\rangle\in E(G)$ there is a graph $H$ and an epimorphism $f^H_G\colon  H\to G$ such that for all vertices $p,q,r\in V(H)$ such that $f^H_G(p)=a$, $f^H_G(q)=b$, and $f^H_G(r)=c$ we have $\langle p,q\rangle\notin E(H)$ or $\langle q,r\rangle\notin E(H)$. Then there are no distinct vertices $a,b,c$ in the projective Fra\"{\i}ss\'e limit $\mathbb G$ such that the edges $\langle a,b\rangle$ and $\langle b,c \rangle$ are in $\mathbb G$ hence $\mathcal G$ has a transitive set of edges. 
\end{theorem}
\begin{proof}
Let $\mathcal G$ be a family that satisfies the assumptions of the theorem.
Suppose on the contrary that there are three vertices $a,b,c\in \mathbb G$ such that $\langle a,b\rangle, \langle b,c\rangle\in E(\mathbb G)$. Let a graph $G\in\mathcal G$ and an epimorphism $f_G\colon \mathbb G\to G$ be such that $f_G(a), f_G(b), \text{  and }f_G(c)$ are three distinct vertices of $G$. Then $\langle f_G(a),f_G(b)\rangle, \langle f_G(b),f_G(c)\rangle\in E(G)$, and thus, by our assumption, there is a graph
$H$ and an epimorphism $f^H_G\colon  H\to G$ such that for an epimorphism $f_H\colon \mathbb G\to H$ satisfying $f_G=f^H_G\circ f_H$ we have  $\langle f_H(a),f_H(b)\rangle\notin E(H)$ or $\langle f_H(b),f_H(c)\rangle\notin E(H)$. This contradicts the fact that $f_H$ maps edges to edges.
\end{proof}

\begin{theorem}\label{limit-of-hu}
If $\mathcal T$ is a projective Fra\"{\i}ss\'e family of  trees, and $\mathbb T$ is a projective Fra\"{\i}ss\'e
limit of $\mathcal T$, then $\mathbb T$ is hereditarily unicoherent.
\end{theorem}
\begin{proof}

Suppose $\mathbb T$ is not hereditarily unicoherent and fix a metric $d$ on $\mathbb T$.  Then there exist closed connected subsets $P$ and $Q$ of $V(\mathbb T)$ such that $P \cap Q$ is not connected. Let $p \in P\backslash Q$ and $q \in Q\backslash P$. Clearly there are closed disjoint sets $A'$ and $B'$ in $\mathbb T$ such that they are not adjacent, $p,q\notin A'\cup B'$, $P\cap Q\subseteq A'\cup B'$, $A'\cap (P\cap Q)\neq\emptyset$, and $B'\cap (P\cap Q)\neq\emptyset$.  
Since $E(\mathbb{T})$ is closed and $\mathbb T$ is 0-dimensional, there exist clopen neighborhoods  $A$ of $A'$ and $B$ of $B'$, which satisfy the same properties as $A'$ and $B'$, and additionally $d(P\setminus (A\cup B), Q\setminus (A\cup B))>0$.

We show that there is $\varepsilon_1>0$ such that for any $\varepsilon_1$-epimorphism, $f\colon  \mathbb T \to F$, where $F\in\mathcal{T}$,    $f(A)$ and $f(B)$ disjoint and non-adjacent. Towards the contradiction, assume that for all $\varepsilon >0$ and for all $\varepsilon$-epimorphisms  $f_{\varepsilon}\colon \mathbb T \to F_{\varepsilon}$, where $F_{\varepsilon}\in\mathcal{T}$, $f_{\varepsilon}(A)$ and $f_{\varepsilon}(B)$ are adjacent. Let $\varepsilon_n \to 0$ and $f_n\colon \mathbb T \to F_n$ be an $\varepsilon_n$-epimorphism such that $\langle a_n,b_n\rangle \in E(F_n)$, $a_n\in f_n(A)$, and $b_n\in f_n(B)$. Since $f_n$ is an epimorphism there exists $a_n' \in f_n^{-1}(a_n)$ and $b_n' \in f_n^{-1}(b_n)$ such that $\langle a_n', b_n'\rangle \in E(\mathbb T)$. By compactness of $\mathbb T$ we may assume that $a_n' \to a$ and $b_n' \to b$, for some $a,b\in \mathbb T$. Since $E(\mathbb T)$ is closed, $\langle a, b\rangle \in E(\mathbb T)$. Because $f_n$ is an $\varepsilon_n$-epimophism we have that $d(a_n', A) < \varepsilon_n$ and $d(b_n', B) < \varepsilon_n$. So $a \in A$ and $b \in B$, which gives a contradiction. We fix such an $\varepsilon_1>0$.

Let $\varepsilon_1>\varepsilon>0$ be such that $d(A,B)>\varepsilon$ and $d(P\setminus (A\cup B), Q\setminus (A\cup B))>\varepsilon$.
By Condition \ref{refinement} of Theorem \ref{limit}, there exists a tree $G$ and an $\varepsilon$-epimophism, $f_G\colon  \mathbb T \to G $, such that $f_G(p) \not \in f_G(Q)$ and  $f_G(q)\not \in f_G(P)$.
In particular, $f_\epsilon(P) \not \subseteq f_\epsilon (Q)$ and $f_\varepsilon (Q) \not \subseteq f_\varepsilon(P)$.
The sets $f_G(P)$ and $f_G(Q)$ are clearly closed and connected. Moreover, 
$f_G(P) \cap f_G(Q)\subseteq f_G(A)\cup f_G(B)$. 
Indeed, if  $x\in P$ and $y\in Q$ are such that $f_G(x)=f_G(y)\notin f_G(A)\cup f_G(B)$, then $x\in P\setminus (A\cup B)$, $y\in Q\setminus (A\cup B)$, and hence $d(x, y)>\varepsilon$, which is impossible as $f_G$ is an $\varepsilon$-epimophism.
Let $H=f_G(A)\cap (f_G(P) \cap f_G(Q))$ and $K=f_G(B)\cap (f_G(P) \cap f_G(Q))$. Note that $H\supseteq f_G(A\cap P\cap Q)\neq\emptyset$ and $K\neq\emptyset$.  Then $H$ and $K$ are disjoint closed non-adjacent sets and $H \cup K = f_G(P) \cap f_G(Q)$ so $f_G(P) \cap f_G(Q)$ is not connected. This contradicts the fact that $G$ is a tree.

\end{proof}

\section{Monotone epimorphisms}

In this section we investigate a family of  trees with monotone epimorphisms. We show that this family, which we denote by $\mathcal T_{\mathcal M3}$, is a projective Fra\"{\i}ss\'e family  (Theorem~\ref{amalgamation-monotone}) and that the topological realization of $\mathcal T_{\mathcal M3}$ is homeomorphic to $D_3$, the standard universal denrite of order 3, also known as the Wa\. zewski denrite of order 3  (Theorem~\ref{D3}). 
Injective Fra\"{\i}ss\'e constructions of $D_3$ and other Ważewski dendrites were explored by Kwiatkowska \cite{K-D3} and Duchesne \cite{Du}. These constructions are very different from the one we are about to present.

First, we want to show an example that the family of all epimorphisms between trees, even trees with all vertices having order less than or equal to three, do not have amalgams that are connected graphs.
\begin{example}

There exist a triod $T$, arcs $I$ and $J$, epimorphisms $f\colon I\to T$ and $g\colon J\to T$, for which there is no connected graph $G$ and epimorphisms $f_0$ and $g_0$ such that the diagram below commutes.

\begin{equation}\tag{D2}
\begin{tikzcd}
&I\arrow{ld}[swap]{f}\\
T&&G\arrow[lu,swap,dotted,"f_0"] \arrow[ld,dotted,"g_0"]\\
&J\arrow[lu,"g"]
\end{tikzcd}
\end{equation}

The triod $T$ has the center $b$ and the end vertices $a$, $c$, and $d$; the arc $I$ has vertices $p_1,p_2,p_3,p_4$, and $p_5$, similarly
the arc  $J$ has vertices $q_1,q_2,q_3,q_4$, and $q_5$. The epimorphisms $f$ and $g$ are pictured below. Precisely, we have $f(p_1)=d$,
$f(p_2)=b$, $f(p_3)=c$, $f(p_4)=b$, and $f(p_5)=a$; similarly,  $g(q_1)=d$,
$g(q_2)=b$, $g(q_3)=a$, $g(q_4)=b$, and $g(q_5)=c$.

\begin{center}
\begin{tikzpicture}[scale=0.75]

  \draw (3,0) -- (3,3);
   \draw (3.7,-0.7) -- (3.7,2.3) -- (6,2.3);
   \draw (3.7,3.7) -- (6,3.7) -- (0,3.7);
  \draw (0,3) -- (3,3);
  \draw (6,3) -- (3,3);
\draw (6,2.3) arc (-90:90:0.7);

  \node at (3,-0.3) {$d$};

  \draw (6,3) -- (3,3);
  \node at (0,2.7) {$a$};
\node at (2.7,2.7) {$b$};
\node at (6,2.7) {$c$};
\node at (4,-0.7) {$p_1$};
\node at (4,2) {$p_2$};
\node at (7,3) {$p_3$};
\node at (3,4) {$p_4$};
\node at (0,4) {$p_5$};
\draw (3,3) circle (0.03);
\draw (3,3) circle (0.015);

\draw (3,0) circle (0.03);
\draw (3,0) circle (0.015);

\draw (0,3) circle (0.03);
\draw (0,3) circle (0.015);

\draw (6,3) circle (0.03);
\draw (6,3) circle (0.015);

\draw (3,3.7) circle (0.03);
\draw (3,3.7) circle (0.015);

\draw (0,3.7) circle (0.03);
\draw (0,3.7) circle (0.015);

\draw (6.7,3) circle (0.03);
\draw (6.7,3) circle (0.015);

\draw (3.7,2.3) circle (0.03);
\draw (3.7,2.3) circle (0.015);

\draw (3.7,-0.7) circle (0.03);
\draw (3.7,-0.7) circle (0.015);

  \draw (12,0) -- (12,3);
    \draw (9,3) -- (12,3);
  \draw (15,3) -- (12,3);
    \draw (9,3.7) -- (15,3.7);

    \node at (12,-0.3) {$d$};

  \node at (9,2.7) {$a$};
\node at (12.3,2.7) {$b$};
\node at (15,2.7) {$c$};

\draw (9,3.7) arc (90:270:0.7);

  \draw (9,2.3) -- (11.3,2.3) -- (11.3,-0.7);

  \node at (11,-0.7) {$q_1$};
\node at (11,2) {$q_2$};
\node at (8,3) {$q_3$};
\node at (12,4) {$q_4$};
\node at (15,4) {$q_5$};

\draw (12,3) circle (0.03);
\draw (12,3) circle (0.015);

\draw (12,0) circle (0.03);
\draw (12,0) circle (0.015);

\draw (9,3) circle (0.03);
\draw (9,3) circle (0.015);

\draw (15,3) circle (0.03);
\draw (15,3) circle (0.015);

\draw (15,3.7) circle (0.03);
\draw (15,3.7) circle (0.015);

\draw (12,3.7) circle (0.03);
\draw (12,3.7) circle (0.015);

\draw (8.3,3) circle (0.03);
\draw (8.3,3) circle (0.015);

\draw (11.3,2.3) circle (0.03);
\draw (11.3,2.3) circle (0.015);

\draw (11.3,-0.7) circle (0.03);
\draw (11.3,-0.7) circle (0.015);

\node at (3,-2) {$f\colon I\to T$};
\node at (12,-2) {$g\colon J\to T$};

\end{tikzpicture}
\end{center}

Suppose that there is a connected graph $G$ and epimorphisms $f_0\colon G\to I$ and $g_0\colon G\to J$ such that the diagram (D2) commutes. Let $x_0\in (f_0)^{-1}(p_1)$ and let
$x_0,x_1,\dots x_n$ be a sequence of vertices of $G$ such that:
\begin{enumerate}
    \item for each $i \in \{0,1,\dots, n-1\}$ we have $\langle x_i,x_{i+1}\rangle\in E(G)$;
    \item $x_0,x_1,\dots, x_{n-1}\in (f\circ f_0)^{-1}(\{b,d\})$;
    \item $x_{n}\notin (f\circ f_0)^{-1}(\{b,d\})$.
\end{enumerate}
Then we have
\begin{enumerate}
    \item $f_0(x_0)=p_1$,
    \item $f_0(x_0), f_0(x_1), \dots, f_0(x_{n-1})\in \{p_1,p_2\}$;
    \item $f_0(x_n)=p_3$;
    \item $g_0(x_0)=q_1$,
    \item $g_0(x_0), g_0(x_1), \dots, g_0(x_{n-1})\in \{q_1,q_2\}$;
    \item $g_0(x_n)=q_3$.
    \end{enumerate}
By conditions (3) and (6) we have that $f(f_0(x_n))=c$, while $g(g_0(x_n))=a$, so the diagram (D2) does not commute, a contradiction.

\end{example}

\begin{definition}
Given two topological graphs $G$ and $H$ an epimorphism $f\colon G\to H$  is called {\it monotone} if the preimage of a closed connected
subset of $H$ is a connected subset of $G$. 
\end{definition}

The following is an immediate consequence of the definition.
\begin{remark}\label{disjointimage}
 If $f\colon G\to H $ is a monotone epimorphism between topological graphs, $U,V\subseteq G$ are connected and disjoint, then $f(U)$ and $f(V)$ are disjoint.    
\end{remark}

\begin{lemma}
    Let $G$ and $H$ be topological graphs and let $f\colon G\to H$ be an epimorphism. If the preimage of every
vertex in $H$ is connected, then $f$ is monotone.
\end{lemma}
\begin{proof}
    Suppose towards the contradiction that the preimage of a closed connected subset $C$ of $V(H)$ is not connected.
    Let $D,E\subseteq G$ be disjoint nonempty closed sets such that $f^{-1}(C)=D\cup E$. By the assumption, note that for every $c\in C$, we have $f^{-1}(c)\subseteq D$ or $f^{-1}(c)\subseteq E$. Therefore $f(D)$ and $f(E)$ are disjoint. Since they are closed  nonempty, and $C=f(D)\cup f(E)$, we get a contradiction with connectedness of $C$.
\end{proof}

\begin{lemma}\label{mon-comp}
Consider the epimorphisms $f\colon F\to G$ and $g\colon G\to H$ between topological graphs. 
\begin{itemize}
\item[(a)] If $f$ and $g$ are monotone, then so is $g\circ f$.
\item[(b)] If the composition $g\circ f\colon F\to H$ is monotone, then $g$ is monotone.
\end{itemize}
\end{lemma}
\begin{proof}
Part (a) is clear. For part (b), write $h=g\circ f$ and let $c\in V(H)$. Then $g^{-1}(c)=f(h^{-1}(c))$. By the assumption, $h^{-1}(c)$ is connected. Since images by epimorphisms of connected sets are connected, $g^{-1}(c)$ is connected.
\end{proof}

The following example shows that the concept of monotone epimorphism for graphs is not exactly how continuum theory sees monotone maps between continua.

\begin{example}
 There is a monotone epimorphism from a cyclic graph onto an arc.

 Let $G$ be a complete graph with three vertices $a,b$, and $c$, i.e. a graph where
 $E(G)=V(G)^2$, and let $H$ be an arc with two vertices $p$ and $q$, and $E(H)=V(H)^2$. Define $f\colon G\to H$ by $f(a)=p$ and $f(b)=f(c)=q$. The reader can
 verify that $f$ is a monotone epimorphism.
\end{example}

The following result, which is well known in continuum theory, also holds in the setting of topological graphs and is used in Theorem~\ref{D3}.

\begin{lemma}\label{images-of-arcs}
If $f\colon G\to H$ is a monotone epimorphism between topological graphs and $G$ is an arc, then $H$ is an arc and the images of end vertices of $G$ are end vertices of  $H$.
\end{lemma}

\begin{proof}
Denote the end vertices of $G$ by $a$ and $b$. We need to show that every vertex in $H\setminus \{f(a),f(b)\}$ disconnects $H$. Let $y\in H\setminus \{f(a),f(b)\}$; then, since
$G$ is an arc the graph $G\setminus f^{-1}(y)$ is disconnected. Let $G\setminus f^{-1}(y)$ be the union of two disjoint graphs $G\setminus f^{-1}(y)=U\cup V$. Thus $H\setminus \{y\}=f(U)\cup f(V)$ and $f(U)\cap f(V)=\emptyset$ by Remark \ref{disjointimage}, so $H\setminus \{y\}$ is disconnected as needed.
\end{proof}

It is well known in continuum theory that the inverse limit of arcs with monotone maps is an arc. It is easy to see that the same result holds for graphs.
\begin{lemma}
 The inverse limit of finite arcs with monotone bonding epimorphisms is an arc.   
\end{lemma}

\begin{proposition}\label{monotone-projections}
If $\mathcal G$ is a projective Fra\"{\i}ss\'e family of graphs with monotone epimorphisms and $\mathbb G$ is a projective Fra\"{\i}ss\'e 
limit of $\mathcal G$, then for every $G \in \mathcal G$ any epimorphism $f_G\colon \mathbb G\to G$ is monotone.
\end{proposition}
\begin{proof}
Let $\{F_n,\alpha_n\}$ be a Fra\"{\i}ss\'e sequence for $\mathcal F$. Fix $m$ and let $x\in F_m$. We show that $T=(\alpha^\infty_m)^{-1}(x)$ is connected. By monotonicity of each $\alpha_m^n$, we have that for every $n>m$, $(\alpha^n_m)^{-1}(x)$ is connected. Suppose towards
the contradiction that $T$ is not connected. Then there exist nonempty closed disjoint sets $R$ and $S$ such that $R\cup S=T$. Choose $\varepsilon>0$ such that $d(a,b)>\varepsilon$
for every $a\in R$ and $b\in S$, where $d$ is a metric on $\mathbb G$. Let $n_0$ be large enough so that $\alpha^\infty_{n_0}$ is an $\varepsilon$-epimophism. Then 
 $\alpha^\infty_{n_0}(R)$ and $\alpha^\infty_{n_0}(S)$ are two disjoint nonempty closed subsets of $F_{n_0}$ such that $(\alpha^{n_0}_m)^{-1}(x)=\alpha^\infty_{n_0}(R)\cup \alpha^\infty_{n_0}(S)$, contrary to monotonicity of $\alpha_m^{n_0}$. This finishes the proof of monotonicity of each  $\alpha^\infty_{n}$.

Finally, there is an $m_0$ and a monotone epimorphism $f\colon F_{m_0} \to G$ such that $f_G= f\circ \alpha_{m_0}^\infty$.  As $\alpha_{m_0}^\infty $ is monotone, so is $f_G$.

\end{proof}

\begin{proposition}\label{arwise-connected-limits}
If $\mathcal T$ is a projective Fra\"{\i}ss\'e family of  trees with monotone epimorphisms, and $\mathbb T$ is a projective Fra\"{\i}ss\'e 
limit of $\mathcal T$, then $\mathbb T$ is arcwise connected.
\end{proposition}

\begin{proof}
Let $\{T_n,\alpha_n\}$ be a Fra\"{\i}ss\'e sequence for $\mathbb T$ and take $a=(a_n), b=(b_n) \in \mathbb T$.
For each $n$ let $I_n$ be the arc joining $a_n$ and $b_n$. Since each tree $T_n$ is uniquely arcwise connected, we have $\alpha_n(I_{n+1})=I_n$. Therefore $\iLim\{I_n,\alpha_n| I_n\}$ is an arc joining $a$ and $b$.

\end{proof}

\begin{corollary}\label{dendrites}
If $\mathcal T$ is a projective Fra\"{\i}ss\'e family of trees with monotone epimorphisms, and $\mathbb T$ is a projective Fra\"{\i}ss\'e 
limit of $\mathcal T$, then $\mathbb T$ is a dendrite.
\end{corollary}

\begin{proof}
The limit $\mathbb T$ is arcwise connected by Proposition~\ref{arwise-connected-limits}, it is hereditarily unicoherent by Theorem \ref{limit-of-hu}, and it is locally connected by \cite[Theorem 2.1]{Menger}.
\end{proof}

The next example shows that it is not always possible to find an amalgamation of trees with monotone epimorphisms when the order of some of the ramification vertices is greater than 3.

\begin{example}\label{no-4od}
    Let $A$ be the 4-od with end vertices $a$, $b$, $c$, $d$ and order 4 ramification vertex $x$. Let $B$ be the tree having end vertices $a_B$, $b_B$, $c_B$, $d_B$ and order 3 ramification vertices $x_1$ and $x_2$ with edges $\langle a_B,x_1\rangle$, $\langle b_B,x_1\rangle$, $\langle x_1,x_2\rangle$, $\langle c_B,x_2\rangle$, and $\langle d_B,x_2\rangle$ and let $C$ be the tree having end vertices $a_C$, $b_C$, $c_C$, $d_C$ and order 3 ramification vertices $x'_1$ and $x'_2$ with edges $\langle a_C,x'_1\rangle$, $\langle c_C,x'_1\rangle$, $\langle x'_1,x'_2\rangle$, $\langle b_C,x'_2\rangle$, and $\langle d_C,x'_2\rangle$. Let $f\colon B \to A$ be the monotone epimorphism such that $f(x_1)=f(x_2)=x$ and otherwise maps corresponding letters. Finally, let $g\colon C \to A$ be the monotone epimorphism such that $g(x_1')=g(x_2')= x$ and otherwise maps corresponding letters.

\end{example}

\begin{center}
\begin{tikzpicture}[scale=0.6]

    \draw (-2,1) -- (0,-1);
    \draw (0,1) -- (-2,-1);
      \filldraw[black] (-1,0) circle (1pt);
      \filldraw[black] (-2,1) circle (1pt);
      \filldraw[black] (0,1) circle (1pt);
      \filldraw[black] (-2,-1) circle (1pt);
      \filldraw[black] (0,-1) circle (1pt);
      \node at (-0.5,0) {$x$};
    \node at (-2,1.3) {$a$};
    \node at (0,1.3) {$b$};
    \node at (-2,-1.3) {$c$};
    \node at (0,-1.3) {$d$};
    \node at (-1,-2) {$A$};

    \draw (3,5) -- (4,4) -- (4,3) -- (3,2);
    \draw (5,5) -- (4,4) -- (4,3) -- (5,2);
      \filldraw[black] (3,5) circle (1pt);
      \filldraw[black] (4,4) circle (1pt);
      \filldraw[black] (4,3) circle (1pt);
      \filldraw[black] (3,2) circle (1pt);
      \filldraw[black] (5,5) circle (1pt);
      \filldraw[black] (5,2) circle (1pt);
    \node at (3,5.3) {$a_B$};
    \node at (4.5,4) {$x_1$};
    \node at (4.5,3) {$x_2$};
    \node at (3,1.7) {$c_B$};
    \node at (5,5.3) {$b_B$};
    \node at (5,1.7) {$d_B$};
    \node at (4,1) {$B$};

    \draw (3,-5) -- (4,-4) -- (4,-3) -- (3,-2);
    \draw (5,-5) -- (4,-4) -- (4,-3) -- (5,-2);
      \filldraw[black] (3,-5) circle (1pt);
      \filldraw[black] (4,-4) circle (1pt);
      \filldraw[black] (4,-3) circle (1pt);
      \filldraw[black] (3,-2) circle (1pt);
      \filldraw[black] (5,-5) circle (1pt);
      \filldraw[black] (5,-2) circle (1pt);
    \node at (3,-5.3) {$b_C$};
    \node at (4.5,-4) {$x'_2$};
    \node at (4.5,-3) {$x'_1$};
    \node at (3,-1.7) {$a_C$};
    \node at (5,-5.3) {$d_C$};
    \node at (5,-1.7) {$c_C$};
    \node at (4,-6) {$C$};

    \draw (1,2) -- (2.5,3.5);
    \draw (1,2) -- (1.1,2.3);
    \draw (1,2) -- (1.3,2.1);
    \node at (1.5,3) {$f$};

    \draw (1,-2) -- (2.5,-3.5);
    \draw (1,-2) -- (1.1,-2.3);
    \draw (1,-2) -- (1.3,-2.1);
    \node at (1.5,-3) {$g$};

    \node at (9,0) {$D$};
    \draw [dashed] (7.5,2) -- (6,3.5);
    \draw (6,3.5) -- (6.1,3.2);
    \draw (6,3.5) -- (6.3,3.4);
    \node at (7.2,3) {$f_0$};

    \draw [dashed] (7.5,-2) -- (6,-3.5);
    \draw (6,-3.5) -- (6.1,-3.2);
    \draw (6,-3.5) -- (6.3,-3.4);
    \node at (7.2,-3) {$g_0$};

\end{tikzpicture}
\end{center}

Assume towards the contradiction that there is a tree $D$ and monotone epimorphisms $f_0\colon D \to B$ and $g_0\colon D \to C$ such that $f\circ f_0 = g \circ g_0$.
Let $a_D, b_D, c_D, d_D$ be vertices in $D$ which are mapped by $f\circ f_0$ to $a$, $b$, $c$, and $d$ respectively. Recall that we denote the arc joining vertices $x$ and $y$ by $[x,y]$. Since $f_0([a_D,b_D])$ is connected it contains $x_1$. If $f_0([a_D,b_D])$ contained  $x_2$ then $f_0^{-1}(x_1)$ would be disconnected, contradicting $f_0$ being monotone. Thus  $f_0([a_D,b_D])=\{a_B,x_1, b_B\}$. Similarly, $f_0([c_D,d_D])=\{c_B,x_2,d_B\}$,
$g_0([a_D,c_D])=\{a_C,x'_1, c_C\}$, and $g_0([b_D,d_D])=\{b_C,x'_2,d_C\}$.
Therefore $[a_D,b_D]\cap [c_D,d_D]=\emptyset$ and $[a_D,c_D]\cap [b_D,d_D]=\emptyset$. It is not hard to see that $[a_D,b_D]\cup [b_D,d_D]\cup [c_D,d_D]\cup [a_D,c_D]$ contains a cycle, which contradicts that $D$ is a tree.

Let $\treet$ be the family of trees having at least two vertices, such that all vertices have order $\leq 3$, with monotone epimorphisms.

\begin{theorem}\label{amalgamation-monotone}
$\treet$ has the JPP and AP.
\end{theorem}

\begin{proof}

For the proof of the JPP, let $B, C\in \treet$ and let $e_B$ and $e_C$ be fixed end vertices of $B$ and $C$ respectively.
Take a tree $D$ to be the disjoint union of $B$ and $C$ with the vertices $e_B$ and $e_C$ identified. Let $f\colon D \to B$ be the epimorphism that is the identity on $B$ and takes all vertices of $C$ to $e_B$.  Similarly define $g\colon D \to C$.  

For the proof of the AP, let $A,B,C$ be trees such that each vertex has order $\leq 3$, and let   $f\colon B\to A$  and $g\colon C\to A$ be monotone. 

We claim that if $x\in A$ is such that $\ord(x)=3$ and $a,b,c$ are all vertices adjacent to $x$, then for any $a_1\in f^{-1}(a)$, $b_1\in f^{-1}(b)$ and $c_1\in f^{-1}(c)$, there is a unique $x_1\in f^{-1}(x)$ such that any two of the arcs: $[a_1,x_1], [b_1,x_1], [c_1,x_1]$ intersect exactly in $x_1$. 
Indeed, let $x_a, x_b, x_c\in f^{-1}(x)$, $a'\in f^{-1}(a)$, $b'\in f^{-1}(b)$, and $c'\in f^{-1}(c)$ be such that $\langle x_a, a'\rangle,
\langle x_b, b'\rangle,
\langle x_c, b'\rangle\in E(B)$. By the monotonicity of $f$ there is a unique choice of such vertices. In the case that two or three of $x_a, x_b, x_c$ are equal, we take $x_1$ to be that common vertex. In the case when $x_a, x_b, x_c$ are pairwise different, there is, since $B$ is a tree, a unique minimal arc $q$ joining $x_c$ to a vertex in the arc $[x_a,x_b]$. The common vertex of $[x_a,x_b]$ and $q$ is the required $x_1$.
 Call such a vertex $x_1$ a {\it center} for $f$ and $x$.

We first focus on $f\colon B\to A$.
Now let $i_f\colon A\to B$ be the following injection. Given $x\in A$, if   $\ord(x)=3$ we let  $i_f(x)$ to be the center for $f$ and $x$, if $\ord(x)=2$, let $i_f(x)$ be any vertex in $f^{-1}(x)$, and if $\ord(x)=1$, let $i_f(x)$ be any end vertex in $f^{-1}(x)$. 
We analogously define $i_g$ for $g$.

Fix an edge $e=\langle a,b\rangle\in E(A)$ and let $$i_f(a)=a^B_0, a^B_1, a^B_2,\ldots, a^B_{m_B}, b^B_1, b^B_2,\ldots, b^B_{n_B}, b^B_{n_B+1}=i_f(b)$$ be the arc joining $i_f(a)$ and $i_f(b)$, where
$f(a^B_i)=a$, $i=1,2,\ldots, m_B$, and $f(b^B_j)=b$, $j=1,2,\ldots, n_B$. 
Then, for each $i\in\{1, \ldots, m_B\}$,  $B\setminus \{a^B_i\}$ has a component containing $i_f(a)$, a component containing $i_f(b)$, and possibly a third component.  If there is a third component denote it by $\hat{a}^B_i$, and note that  $f(\hat{a}^B_i)=a$. Furthermore, it can happen that $\ord(a)=2$ but $\ord(i_f(a))=3$. In that case there is a unique component $\hat{a}^B$ of $B\setminus \{i_f(a)\}$, not containing vertices from $i_f(A)$ such that $f(\hat{a}^B)=a$.
Let $x^B_i\in \hat{a}^B_i$ be such that $\langle a^B_i, x^B_i\rangle\in E(B)$, $i=1,2,\ldots, m_B$, in case  $\hat{a}^B_i\neq\emptyset$. Let $a^B\in\hat{a}^B$ be such that $\langle a^B, i_f(a)\rangle \in E(B)$, in case $\hat{a}^B\neq\emptyset$.
In a similar manner define components $\hat{b}^B_j$ for each $b_j$, $j=1,2,\ldots, n_B$.
Some of those components can be empty. This time we have that $f(\hat{b}_j^B)=b$. Let $y^B_j\in \hat{b}^B_j$ be such that $\langle 
b^B_j, y^B_j\rangle\in E(B)$, $j=1,2,\ldots, n_B$, in case  $\hat{b}^B_j\neq\emptyset$.

 We do a similar analysis for $g$ and $C$ and an edge $e=\langle a,b\rangle\in E(A)$. If 
 $i_g(a)=a^C_0, a^C_1, a^C_2,\ldots, a^C_{m_C}, b^C_1, b^C_2,\ldots, b^C_{n_C}, b^C_{n_C+1}=i_g(b)$ is the arc  joining $i_g(a)$ and $i_g(b)$, we define components $\hat{a}^C_i$, $\hat{b}^C_j$, and vertices $x^C_i$, $y^C_j$, analogously as above. Define also $\hat{a}^C$ and $a^C$.

 To get $D$ we just have to put together all those new components. Here is how we build $D$. 
To simplify the notation, below we will often write $a$ for $i_f(a)$ and for $i_g(a)$, where $a$ is a vertex in $A$.
 For every edge $e=\langle a,b\rangle\in E(A)$ we do the following. Replace $e$ by edges: $$\langle a, a^B_1 \rangle,\ldots, \langle a^B_{m_B}, a^C_1 \rangle, \ldots,
 \langle a^C_{m_C}, b^B_1 \rangle, \ldots,  \langle b^B_{n_B}, b^C_1\rangle, \ldots,
 \langle b^C_{n_C}, b\rangle.  $$
Then we attach  $\hat{a}^B_i$, $i=1,\ldots, m_B$, $\hat{a}^C_i$,   $i=1,\ldots, m_C$, as well as $\hat{b}^B_j$, $j=1,2,\ldots, n_B$, $\hat{b}^C_j$,  $j=1,2,\ldots, n_C$, in appropriate places. Specifically, take the disjoint union of all  nonempty
 $\hat{a}^B_i$, $\hat{a}^C_i$, $\hat{b}^B_j$,  $\hat{b}^C_j$, 
as well as the arc 
 $a, a^B_1,\ldots, a^B_{m_B}, a^C_1 , \ldots,
 a^C_{m_C}, b^B_1 , \ldots,  b^B_{n_B}, b^C_1, \ldots,
 b^C_{n_C}, b$ and add appropriate edges $\langle x^B_i , a^B_i \rangle$, $\langle x^C_i , a^C_i \rangle$,
 $\langle y^B_j , b^B_j \rangle$, $\langle y^C_j , b^C_j \rangle$.
 Next, for every $a\in A$ such that $\ord(a)=2$, take  a vertex $a'$ and an edge $\langle a,a'\rangle$. In case $\hat{a}^B$ was nonempty take it together with the edge $\langle a', a^B\rangle$. Similarly, if  $\hat{a}^C$ was nonempty take it together with the edge $\langle a', a^C\rangle$. This finishes the definition of $D$.
 
We now have to specify monotone epimorphisms $f_0\colon D\to B$ and $g_0\colon D\to C$ such that $f\circ f_0=g\circ g_0$. Let us first define $f_0$. For every edge  $e=\langle a,b\rangle\in E(A)$, let $f_0(\hat{a}_i^C)=f_0(a_i^C)=a_{m_B}^B$, $f_0(\hat{b}_j^C)=f_0(b_j^C)=b$, and for every $a\in V(A)$, let $f_0(\hat{a}^C)=f_0(a')=a$. We let $f_0$ to be the identity otherwise.
We now define $g_0$. For every edge  $e=\langle a,b\rangle\in E(A)$, let
$g_0(\hat{a}_i^B)=g_0(a_i^B)=a$, $g_0(\hat{b}_j^B)=g_0(b_j^B)=b_1^C$, and for every $a\in V(A)$, let $g_0(\hat{a}^B)=g_0(a')=a$. We let $g_0$ to be the identity otherwise. The maps $f_0,g_0$ are as required.

\end{proof}

The {\it Wa\. zewski denrite} $D_3$, see for example \cite{JJC-Wazewski},  is a dendrite characterized by the following conditions:
\begin{enumerate}
    \item each ramification point is of order 3;
    \item the set of ramification points is dense.
\end{enumerate}

We show that the topological realization of the projective Fra\"{\i}ss\'e limit of the family $\treet$ is homeomorphic to the dendrite $D_3$. We divide the proof into two steps.

\begin{theorem}\label{dense}
The projective Fra\"{\i}ss\'e limit of $\treet$ has transitive set of edges and the topological realization of the projective Fra\"{\i}ss\'e limit of $\treet$ is a dendrite with a dense set of ramification vertices.
\end{theorem}
\begin{proof}

 First, we show that the family $\treet$ satisfies the hypothesis of Theorem \ref{transitive} and thus has a transitive set of edges. Let $G\in \treet$  and $a,b,c\in V(G)$ such that $\langle a,b\rangle, \langle b,c\rangle \in E(G)$. Define $H\in \treet$  by $V(H)=V(G)\cup \{a',b'\}$, $E(H)=E(G)\cup\{\langle a,a'\rangle, \langle a',b'\rangle, \langle b',b\rangle\}\setminus\{\langle a,b\rangle\}$, and the epimorphism $f^{H}_G$ where $f^{H}_G(a')=a$, $f^{H}_G(b')=b$ and $f^{H}_G$ is the identity otherwise.  Then, for  $p\in\{a,a'\}$, $q\in \{b,b'\}$, and $r=c$ we have either $\langle p,q\rangle\not \in E(H) $ or $\langle q,r\rangle \not\in E(H)$.

The projective Fra\"{\i}ss\'e limit
$\mathbb D$ of $\treet$ is a dendrite by Corollary \ref{dendrites}, so its topological realization $|\mathbb D|$ is a (topological) dendrite by Observation \ref{top-dendrite}. It remains to show that the set of ramification points of $|\mathbb D|$ is dense.

Let us introduce the necessary notation. Let $\varphi\colon \mathbb D\to |\mathbb D|$ be the quotient mapping. To prove the density of the set of ramification points of $|\mathbb D|$ suppose on the contrary that $U$ is an open connected subset of $|\mathbb D|$ that contains no ramification point, i.e $U$ is homeomorphic to $(0,1)$. Then $\varphi^{-1}(U)$
is an open subset of $\mathbb D$. By Property (3) of Theorem \ref{limit}, there is a tree $G$, an epimorphism $f_G\colon \mathbb D\to G$ in $\treet$, and $a\in V(G)$ such that
$f_G^{-1}(a)\subseteq \varphi^{-1}(U)$.  Define a graph $H$ by putting $V(H)=V(G)\cup \{b,c,d, a',c',d'\}$ and $E(H)=E(G)\cup \{\langle a',b\rangle, \langle b,c'\rangle, \langle b,d'\rangle,
\langle a',a\rangle, \langle c,c'\rangle, \langle d,d'\rangle
\}$. Finally, define the monotone epimorphism $f^H_G\colon H\to G$ by
$$
f^H_G(x)=\begin{cases}  x \text{ if } x\in G\\
a \text{ if } x\in \{b,c,d, a',c',d'\}.
\end{cases}
$$
Let $f_H\colon \mathbb D\to H$ be an epimorphism such that $f_G=f^H_G\circ f_H$; note that by Proposition~\ref{monotone-projections} the epimorphism  $f_H$ is monotone.
Thus, the three sets $A=\varphi(f_H^{-1}(\{a, a', b\}))$,
 $B=\varphi(f_H^{-1}(\{b,c,c'\}))$, and  $C=\varphi(f_H^{-1}(\{b,d,d'\}))$ are continua that satisfy $ A\cup B\cup C\subseteq U$, $A\not\subseteq B\cup C$, $B\not\subseteq A\cup C$, $C\not\subseteq A\cup B$, and $A \cap B \cap C \not = \emptyset$. 
For example, $A\not\subseteq B\cup C$ follows from the fact that $\phi(f_H^{-1}(a))\subseteq A\setminus (B\cup C)$. 
 This contradicts the fact that $U$ is homeomorphic to $(0,1)$.

\end{proof}

\begin{theorem}\label{D3}
 The topological realization of the projective Fra\"{\i}ss\'e limit of $\treet$ is homeomorphic to the  dendrite $D_3$ .
\end{theorem}

\begin{proof}
By Theorem \ref{dense} the topological realization $|\mathbb D| $ of the projective Fra\"{\i}ss\'e limit of $\treet$ is a dendrite with a dense set of ramification points; it remains to show that every ramification point of $|\mathbb D| $ has order 3. Suppose on the contrary that there is a ramification point $p\in |\mathbb D|$ of order 4 or more. Then there are four arcs $A, B, C, D$ in $|\mathbb D|$ such that the intersection of any two of them is $\{p\}$.

Denote $\varphi^{-1}(p)=\{p_1,p_2\}$ with a possibility that $p_1=p_2$ and note that the sets $\varphi^{-1}(A)$, $\varphi^{-1}(B)$, $\varphi^{-1}(C)$, and $\varphi^{-1}(D)$ contain arcs $\alpha, \beta, \gamma, \delta$ in $\mathbb D$ and that the intersection of any two of these arcs is $\{p_1,p_2\}$.

Let $G$ be a tree and $f_G\colon \mathbb D\to G$ be a monotone epimorphism such that:
\begin{enumerate}
    \item $f_G(\alpha)$, $f_G(\beta)$, $f_G(\gamma)$, and $f_G(\delta)$ are non-degenerate subsets of $V(G)$;
    \item for any two adjacent vertices of $G$ at most one of them has order~3.
\end{enumerate}
We can achieve (2) since for any $G\in \treet$, if we split each edge of $G$ into two, that is, replace each edge $\langle a,b \rangle$ of $G$ by the edges $\langle a,x \rangle$ and $\langle x,b \rangle$ where $x$ is a ``new'' vertex, the resulting graph will have the required property.
By Lemma \ref{images-of-arcs} the sets  $f_G(\alpha)$, $f_G(\beta)$, $f_G(\gamma)$, and $f_G(\delta)$ are arcs.
We claim that the intersection of any two of them is $\{f_G(p_1),f_G(p_2)\}$. Suppose on the contrary that $r\in f_G(\alpha)\cap f_G(\beta)\setminus \{f_G(p_1),f_G(p_2)\}$; then
$f_G^{-1}(r)$ is a connected subset of $V(\mathbb D)$ that intersects both $\alpha\setminus \{p_1,p_2\}$ and  $\beta\setminus \{p_1,p_2\}$ contrary to hereditary unicoherence of $\mathbb D$. The claim contradicts the condition that both  $f_G(p_1)$ and $f_G(p_2)$ are points of order at most 3 and not both of them are of order 3.
\end{proof}

\begin{remark}
    The family of all  trees is cofinal in $\treet$. Indeed, for a  tree $T$ and a vertex $v$ of order $n>3$ we can find a  tree $T'$ obtained from $T$ by replacing $v$ with two vertices joined by an edge, such that one of the vertices has order 3 and the other has order $n-1$. This means that the family of all trees with monotone epimorphisms has the cofinal amalgamation property.
\end{remark}

Codenotti-Kwiatkowska \cite{Co-Kw} generalized Theorems \ref{amalgamation-monotone} and \ref{D3} and constructed all generalized Ważewski dendrites in a Fra\"{\i}ss\'e theoretic framework. 
On the one hand, this work was an inspiration for the article \cite{Co-Kw}. On the other hand, the proof of Theorem \ref{amalgamation-monotone} presented above is a simplification and correction of the one presented in an earlier version of this article and is inspired by the work in \cite{Co-Kw}.

\section{Confluent epimorphisms}\label{confluent}

Motivated by the definition of a confluent map between continua (see \cite{JJC-Confluent}) we give an analogous definition for confluent epimorphisms between graphs and develop various tools which will be used in studying them.

\begin{definition}
Given two topological graphs $G$ and $H$, an epimorphism $f\colon G\to H$  is called {\it confluent} if for every closed connected subset $Q$ of $V(H)$ and every component $C$ of $f^{-1}(Q)$ we have $f(C)=Q$. Equivalently, if for every closed connected subset $Q$ of $V(H)$ and every vertex $a\in V(G)$ such that $f(a)\in Q$ there is a connected set $C$ of $V(G)$ such that $a\in C$ and $f(C)=Q$. Clearly, every monotone epimorphism is confluent.
\end{definition}

\begin{observation}
If $f\colon X \to Y$ and $g\colon Y \to Z$ are confluent epimorphisms between topological graphs, then $g\circ f\colon X \to Z$ is a confluent epimorphism.
\end{observation}

\begin{proposition}\label{confluent-edges}
 Given two finite graphs $G$ and $H$ the following conditions are equivalent for an epimorphism $f\colon G\to H$:
\begin{enumerate}
  \item $f$ is  confluent;
  \item for every edge $P\in  E(H)$ and for every vertex $a\in V(G)$ such that $f(a)\in P$, there is an edge $E\in E(G)$ and a connected set $R\subseteq V(G)$ such that  $E\cap R\ne \emptyset$, $f(E)=P$, $a\in V(R)$, and $f(R)=\{f(a)\}$.
  \item  for every edge $P\in  E(H)$ and every component $C$ of $f^{-1}(P)$ there is an edge $E$ in $C$ such that $f(E)=P$.
\end{enumerate}
\end{proposition}

\begin{proof}  (1) $\Rightarrow$ (2): Let $P\in E(H)$ and $a \in V(G)$ such that $f(a) \in P$. Then there exists a connected subset $C\subseteq G$ such that $a \in C$ and $f(C)= P$. Since $C$ is connected there is an edge $E\in C$ such that $f(E)=P$.  Let $R$ be the component of $f^{-1}(f(a))$ containing $a$. Then (2) is satisfied.

(2) $\Rightarrow$ (3): This is immediate.

(3) $\Rightarrow$ (1): Suppose towards the contradiction that $f$ is not confluent. Then there is a  connected set $Q \subseteq V(H)$ and a component $C$ of $f^{-1}(Q)$ such that $f(C) \subsetneq Q$. Since $H$ is connected there are vertices $a,b \in Q$ with $a \in f(C)$,  $b \in Q\setminus f(C)$, and $P=\langle a,b\rangle\in E(H)$. Let $D$ be the component of $f^{-1}(P)$ that intersects $C$. Then $D \subseteq C$. By (3), there is an edge $E \in C$ such that $f(E) = P$. Hence $b \in f(C)$, which gives a contradiction. 

\end{proof}

\begin{proposition}\label{composition-confluent}
Consider the epimorphisms $f\colon F\to G$ and $g\colon G \to H$ between topological graphs. If the composition $g\circ f\colon F\to H$ is confluent, then $g$ is confluent.
\end{proposition}
\begin{proof}
Let $E$ be a closed connected subset of $H$, and let $a\in G$ be a vertex such that $g(a)\in E$. We need to find a closed connected subset of $S$ of $G$ that contains $a$ and such that $g(S)=E$. Since $g\circ f$ is confluent, there is a closed connected subset $C$ of $F$ that contains a vertex in $f^{-1}(a)$ and such that $(g\circ f)(C)=E$. Then 
$S=f(C)$ satisfies the requirements.
\end{proof}  

\begin{proposition}\label{confluent-projections}
If $\mathcal F$ is a projective Fra\"{\i}ss\'e family of graphs with confluent epimorphisms and $\mathbb F$ is a projective Fra\"{\i}ss\'e limit of $\mathcal F$, then
for every graph $G\in\mathcal F$ every epimorphism $f_G\colon \mathbb F\to G$ is confluent.
\end{proposition}

\begin{proof}

We first show that the canonical projections $\alpha^\infty_m \colon \iLim\{F_n,\alpha_n\} \to F_m$ are confluent.
 Let $m$ and $\langle a_m, b_m\rangle\in E(F_m)$ be given. Fix $a=(a_n)\in \mathbb F$, so we have $\alpha^\infty_m(a)=a_m$.
We will find $b\in \mathbb F$ with $\alpha^\infty_m(b)=b_m$, such that there is an arc $J$ joining $a$ and $b$, for which $\alpha^\infty_m(J)=\{a_m,b_m\}$. We construct $J$ as the inverse limit of finite arcs. Let $J_m=\{a_m, b_m\}$, and for $k<m$,
let $J_k=\alpha^m_k(J_m)$. By the confluence of $\alpha_m$, there is an arc $J_{m+1}\subseteq F_{m+1}$ such that $a_{m+1}$ is an end vertex of $J_{m+1}$, $\alpha_m(J_{m+1})=J_m$, and $\alpha_m(b_{m+1})=b_m$, where $b_{m+1}$ is
the other end vertex of $J_{m+1}$. 
Suppose that we constructed an arc $J_n$, for some $n>m$, such that  $\alpha_m^n(J_{n})=J_m$, $a_n$ is an end vertex of $J_n$, and $\alpha_m^n(b_{n})=b_m$, where $b_n$ is
the other end vertex of $J_{n}$.  By the confluence of $\alpha_n$, there is an arc $J_{n+1}\subseteq F_{n+1}$ such that $a_{n+1}$ is an end vertex of $J_{n+1}$, $\alpha_n(J_{n+1})=J_n$, and $\alpha_n(b_{n+1})=b_n$, where $b_{n+1}$ is
the other end vertex of $J_{n+1}$. This finishes the construction of $\{J_n\}$ and $(b_n)$. We let $J=\iLim\{J_n,\alpha_n|J_n\}$ and $b=(b_n)$.

Finally, there is an $m_0$ and a confluent epimorphism $f\colon F_{m_0} \to G$ such that $f_G= f\circ \alpha_{m_0}^\infty$.  As $\alpha_{m_0}^\infty $ is confluent, so is $f_G$.

\end{proof}

\begin{definition}
An epimorphism $f\colon G\to H$ between graphs is called {\it light} if for any vertex $h\in V(H)$
and for any $a,b\in f^{-1}(h)$ we have $\langle a,b \rangle\notin E(G)$.

\end{definition}

\begin{definition}\label{standamal}
Recall the diagram (D1).

\begin{equation}\tag{D1}
\begin{tikzcd}
&B\arrow{ld}[swap]{f}\\
A&&D\arrow[lu,swap,dotted,"f_0"] \arrow[ld,dotted,"g_0"]\\
&C\arrow[lu,"g"]
\end{tikzcd}
\end{equation}
For given finite graphs $A,B,C$ and epimorphisms $f\colon B\to A$, $g\colon C\to A$ by {\it standard amalgamation procedure} we mean the graph $D$ defined by $V(D)=\{( b,c)\in B\times C:f(b)=g(c)\}$; $E(D)=\{\langle ( b,c), (b',c') \rangle:\langle b,b'\rangle\in E(B)  \text{ and } \langle c,c'\rangle\in E(C)\}$; and
$f_0(( b,c))=b$, $g_0(( b,c))=c$. Note that $f_0$ and $g_0$ are epimorphisms and $f\circ f_0=g\circ g_0$.
\end{definition}

\begin{proposition}\label{light}
Using the standard amalgamation procedure and the notation of diagram {\rm (D1)}, if the epimorphism $f$
is light, then the epimorphism $g_0$ is light.
\end{proposition}

\begin{proof}
Suppose $g_0$ is not light.  Then there is a vertex $h\in V(C)$ such that  $(a_1,h),(a_2,h)\in g_0^{-1}(h)$, and $\langle (a_1,h),(a_2,h) \rangle\in E(D)$.  Then, $\langle a_1,a_2\rangle \in E(B)$ and $a_1, a_2\in f^{-1}(g(h))$, contradicting $f$ being light.
\end{proof}

\begin{proposition}\label{monotone-implies-monotone}
Using the standard amalgamation procedure and the notation of the diagram {\rm (D1)}, if the epimorphism $f$
is monotone, then the epimorphism $g_0$ is monotone.
\end{proposition}

\begin{proof}
Fix a vertex $c\in V(C)$ and consider $g_0^{-1}(c)=\{\langle x,c\rangle :f(x)=g(c)\}$. This set is isomorphic to
$f^{-1}(g(c))$ which is connected by monotonicity of $f$.
\end{proof}

\begin{proposition}\label{standard-confluent}
Using the standard amalgamation procedure and the notation of diagram {\rm (D1)}, if the epimorphism $f$ is confluent,  then the epimorphism $g_0$ is confluent.
\end{proposition}

\begin{proof}
Let $\langle c,c'\rangle \in E(C)$ and $(b,c)\in V(D)$. By Proposition \ref{confluent-edges} we need to find an edge 
$\langle d_1,d_2\rangle \in E(D)$ that is mapped by $g_0$ onto  $\langle c,c'\rangle$ and a connected set $R$ such that 
$(b,c), d_1\in R$, and $g_0(R)=\{c\}$. Since $f$ is confluent, there are vertices $b_1,b_2\in B$ and a connected set $R_B$ such that $\langle b_1,b_2\rangle\in E(B)$, $b,b_1\in R_B$, $f(R_B)=\{f(b)\}=\{g(c)\}$, and $f(\langle b_1,b_2\rangle)=\langle g(c),g(c')\rangle$. It is enough to put $d_1=( b_1,c)$, $d_2=( b_2,c')$, and $R=R_B\times \{c\}$. 

\end{proof}

\begin{lemma}\label{restcomp}
    Let $f\colon B\to A$ be a confluent epimorphism between (not necessarily connected) graphs. Let $X\subseteq A$ be connected and let $Y$ be a component of $f^{-1}(X)$. Then $f|_Y$ is confluent.
\end{lemma}
\begin{proof}

    Let $P\in E(X)$ and let $a\in Y$ be such that $f(a)\in V(P)$. Since $f$ is confluent, by Proposition \ref{confluent-edges}(2), there is an edge $E \subseteq E(B)$ and a connected set $R \subseteq V(B)$ such that $E \cap R \not = \emptyset$, $f(E)=P$, $a \in V(R)$, and $f(R)=\{f(a)\}$.  Since $R$ is connected, $f(R)=\{f(a)\}\subseteq V(P)\subseteq V(X)$, and $Y$ is a component of $f^{-1}(X)$, it follows that $R \subseteq Y$ and the edge $E$ is in $Y$. Then the connected set $R$ and the edge $E$ satisfy Condition (2) of Proposition \ref{confluent-edges} for $f_0$.

\end{proof}

Proposition \ref{standard-confluent} and Lemma \ref{restcomp} yields the following corollary.
\begin{corollary}\label{confluent-components}
Using the standard amalgamation procedure and the notation of diagram {\rm (D1)}, if the graphs $A,B,C$ are connected and the epimorphisms $f$ and $g$ are confluent, then for each component $P$ of the graph $D$ we have $f_0(P)=B$ and $g_0(P)=C$. Moreover, $f_0|_P$ and $g_0|_P$ are confluent.
\end{corollary}

\begin{corollary}\label{confluent-amalgamation}
The family of finite connected graphs with confluent epimorphisms is a projective Fra\"{\i}ss\'e family.
\end{corollary}
\begin{proof}
The AP follows from
Corollary \ref{confluent-components}.    The JPP for graphs $B$ and $C$ follows from AP if $A$, in the notation of diagram {\rm (D1)}, is a graph with only one vertex.
\end{proof}

We finish this section with the following decomposition results.

\begin{proposition}\label{m-l-factorization}
Given an epimorphism $f\colon G\to H$ between graphs, there is a graph $M$ and epimorphisms
$m\colon G\to M$, $l\colon M\to H$ such that $f=l\circ m$,  $m$ is monotone, and $l$ is light. 
\end{proposition}
\begin{proof}
Define an equivalence relation $\sim$ on $V(G)$ by $x\sim y$ if and only if $f(x)=f(y)$ and $x$ and $y$ are in the same
component of $f^{-1}(f(x))$. Putting $V(M)=V(G)/\!\!\!\sim$, $m\colon V(G)\to V(M)$ as the projection,
$\langle [x]_\sim ,[y]_{\sim} \rangle\in E(M)$
if and only if there are $x' \in [x]_\sim, y' \in [y]_{\sim}$ such that $\langle x',y'\rangle\in E(G)$, and $l([x]_\sim)=f(x)$, one can verify that the conclusions of the proposition are satisfied.
\end{proof}

\begin{corollary}\label{m-l-factorization-confluent}
Under the hypothesis of Proposition \ref{m-l-factorization}, if the epimorphism $f$ is confluent, then both $m$ and $l$ are confluent.
\end{corollary}
\begin{proof}
   The epimorphism $m$ is confluent since it is monotone, while $l$ is confluent by Proposition \ref{composition-confluent}.
\end{proof}

\section{Rooted trees}\label{rooted-trees}

The following example shows that confluent epimorphisms on (unrooted) trees do not behave well. Because of this we concentrate on families of rooted trees and develop several results that will be used later.

\begin{example} \label{no-confluent}Let $A$ consist of just one edge $\langle 0,1\rangle$, $B$ be an arc with three vertices $\{a,b,c\}$ and two edges $\langle a,c\rangle$, $\langle c,b\rangle$, and $C$ be again an arc with three vertices
$\{p,q,r\}$ and two edges $\langle p,r\rangle$, $\langle r,q\rangle$. Let $f\colon B\to A$ be such that $f(a)=f(b)=0$
and $f(c)=1$ and let $g\colon C\to A$ be such that $g(p)=g(q)=1$ and $g(r)=0$, see figure below.
It is easy to see that $f$ and $g$ are confluent epimorphisms.

\begin{center}
\begin{tikzpicture}[scale=1]

  \draw (0,1) -- (0,2);
  \filldraw[black] (0,1) circle (1pt);
   \filldraw[black] (0,2) circle (1pt);
  \draw (1.5,2) -- (2,3) -- (2.5,2);
     \filldraw[black] (1.5,2) circle (1pt);
     \filldraw[black] (2,3) circle (1pt);
    \filldraw[black] (2.5,2) circle (1pt);
  \draw (1.5,1) -- (2,0) -- (2.5,1);
     \filldraw[black] (1.5,1) circle (1pt);
    \filldraw[black] (2,0) circle (1pt);
    \filldraw[black] (2.5,1) circle (1pt);
 \node at (-0.3,2) {\scriptsize 1};
 \node at (-0.3,1) {\scriptsize 0};
 \draw (1.5,2.5) --  (0.3,1.7);
 \draw  (0.3,1.7) --(0.4,1.9);
  \draw  (0.3,1.7) --(0.5,1.7);
  \node at (1.5,1.8) {\scriptsize a};
  \node at (2.5,1.8) {\scriptsize b};
\node at (2,3.2) {\scriptsize c};
 \draw (1.5,0.5) --  (0.3,1.3);
 \draw  (0.3,1.3) --(0.5,1.3);
  \draw  (0.3,1.3) --(0.4,1.1);
 \draw [dotted] (3.7,1.8) --  (2.5,2.5);
 \draw  (2.5,2.5) --(2.7,2.5);
  \draw  (2.5,2.5) --(2.6,2.3);
  \node at (4,1.5) { D};

 \draw [dotted] (3.7,1.2) --  (2.5,0.5);
 \draw  (2.5,0.5) --(2.7,0.5);
  \draw  (2.5,0.5) --(2.6,0.7);

 \node at (0.9,2.5) {$f$};
 \node at (0.9,0.5) {$g$};
  \node at (3.3,2.5) {$f_0$};
 \node at (3.3,0.5) {$g_0$};

  \node at (1.5,1.2) {\scriptsize p};
  \node at (2.5,1.2) {\scriptsize q};
\node at (2,-0.2) {\scriptsize r};

\end{tikzpicture}
\end{center}

Note that the standard amalgamation procedure gives the graph $D$ to be the cycle $( c,p)$, $( a,r)$, $( c,q)$, $( b,r)$, $( c,p)$ hence not a tree. 

In fact, no amalgamation of the trees $A$, $B$, and $C$ yields a tree.  To see this assume there is an amalgamation of $A$, $B$, and $C$, that is, there is  tree $D$ and confluent epimorphisms $f_0\colon D\to B$ and $g_0\colon D \to C$ such that $g\circ g_0 = f\circ f_0$.

Without loss of generality, $f_0$ and $g_0$  are light confluent. Indeed, 
let $\sim$ be the  equivalence relation on $D$ given by $p\sim q$ if and only if the arc joining $p$ and $q$ is contained in a  component of $(f\circ f_0)^{-1}(a)$, for some $a\in A$. Take $D'=D/\!\!\sim$, $f'_0=f_0/\!\!\sim\colon D'\to B$ and $g'_0=g_0/\!\!\sim\colon D'\to C$ induced from $f_0$, $g_0$, and $\sim$. As $f$ and $g$ are light,  $f'_0$ and $g'_0$ are well defined and light. Moreover, $D'$ is still a tree, $f'_0$, $g'_0$ are confluent, and  $g\circ g'_0 = f\circ f'_0$.

But now, by confluence of $f$ and $g$, the tree $D$ is such that any vertex is adjacent to at least two other vertices.
Indeed, for any $x\in D$, either $f_0(x)=c$ or $g_0(x)=r$, and each of $c$ and $r$ is adjacent to two other vertices.
Therefore $D$ is an infinite tree. A contradiction.

\end{example}

We will see that rooted trees with simple-confluent epimorphisms as well as with simple*-confluent epimorphisms form projective Fra\"{\i}ss\'e families (see Section 8).

\begin{definition}

By a {\it rooted tree} we mean a rooted graph which is a tree.
Given a rooted tree $T$,
we define an order $\le_T$ by $x\le_T y$ if the arc containing the root of $T$, $r_T$, and $y$ contains $x$.  By $a <_T b$ we mean $b \not \le_T a$. If there is no confusion we will leave off the subscript $T$. In the remainder of this article whenever we have an edge  $\langle a, b\rangle$ in a rooted tree we will mean that $a < b$.

\end{definition}

\noindent {\bf{Assumption.}}
We require epimorphisms $f\colon T\to S$ between rooted trees to be {\it order-preserving}, that is, if $a,b \in T$ with $a \le b$ then $f(a) \le f(b)$. For rooted trees we do not include the root in the set of end vertices.

\smallskip

Epimorphisms between rooted trees need not be confluent as the following example shows.

\begin{example}
Let $S$ be rooted tree which is a triod with vertices $\{A,B,C,D\}$, edges $\{\langle A,B\rangle,\langle B,C\rangle, \langle B,D\rangle\} $, and $A$ is the root of $T$.  
Let $T$ be another rooted tree with vertices $\{a, b_1,b_2,c, d_1,d_2\}$, edges $\{\langle a,b_1\rangle, \langle a,b_2\rangle, \langle b_1,c\rangle, 
\langle b_1,d_1\rangle, \langle b_2,d_2\rangle\}$, and $a$ is the root of $S$.  Let $f\colon T \to S$ be given by $f(a)=A$, $f(b_1)=f(b_2)= B$, $f(c)=C$, 
and $f(d_1)=f(d_2)=D$.  Then $f$ is an order-preserving epimorphism on rooted trees but is not confluent.

\begin{center}

\begin{tikzpicture}[scale=0.5]

  \draw (0,4) -- (1,2);
  \filldraw[black] (0,4) circle (1pt);
  \filldraw[black] (1,2) circle (1pt);
  \draw (2,4) -- (1,2);
  \node at (0.4,4) {\scriptsize $c$};
  \node at (2.4,4) {\scriptsize $d_1$};  
  \filldraw[black] (1,0) circle (1pt);
\filldraw[black] (2,4) circle (1pt);
  \node at (1.4,0) {\scriptsize $a$};
    \node at (1.4,2) {\scriptsize $b_1$};
\node at (2.4,2) {\scriptsize $b_2$};
\draw (1,2) -- (1,0); 
\draw (3,4) -- (1,0); 
\filldraw[black] (2,2) circle (1pt);  
\filldraw[black] (3,4) circle (1pt); 
\node at (3.4,4) {\scriptsize $d_2$};  
\draw (4,2) -- (6,2) -- (5.6, 2.4); 
 \draw (6,2) -- (5.6, 1.6);  
 
\draw (7,4) -- (8,2) -- (8,0);
\draw (8,2) -- (9,4);  
\filldraw[black] (7,4) circle (1pt); 
\filldraw[black] (8,2) circle (1pt); 
\filldraw[black] (8,0) circle (1pt); 
\filldraw[black] (9,4) circle (1pt); 
\node at (8.4,0) {\scriptsize $A$};
 \node at (8.4,2) {\scriptsize $B$};
 \node at (7.4,4) {\scriptsize $C$};
 \node at (9.4,4) {\scriptsize $D$}; 
 \node at (1,-1) {\scriptsize $T$};
 \node at (8,-1) {\scriptsize $S$}; 
\end{tikzpicture}
\end{center}
\end{example}

The following example shows that the standard amalgamation procedure, when applied to rooted trees, does not necessarily yield a tree.

\begin{example}\label{rooted-standard-not-tree}
Let $A$ consist of a single vertex, $B$ and $C$ be edges, where a vertex in each edge is designated as the root. Then the standard amalgamation procedure gives a complete graph on four vertices.
\end{example}

\begin{definition} \label{chain-def}

By a {\it chain} in an  rooted tree $G$ we mean sequence $a_1, a_2, \dots a_n$ of vertices of $G$ such that:

\begin{enumerate}
\item $a_1=r_G$;
\item $a_1\le a_2\le \dots \le a_n$.
\end{enumerate}
A chain is {\it proper} if moreover, it satisfies the condition
\begin{enumerate}
\setcounter{enumi}{2}
    \item $a_1\ne a_2\ne \dots \ne a_n$.
\end{enumerate}

A {\it branch} is a maximal proper chain.
\end{definition}

\begin{observation}\label{monotone-relative}
If $f\colon G\to H$ is an epimorphism between rooted trees and $C$ is a chain in $G$, then $f|_C$ is monotone.
\end{observation}

\begin{proposition}\label{arwise-connected}
If $\mathcal T$ is a projective Fra\"{\i}ss\'e family of rooted trees, and $\mathbb T$ is a projective Fra\"{\i}ss\'e 
limit of $\mathcal T$, then $\mathbb T$ is arcwise connected.
\end{proposition}

\begin{proof}
Let $\{T_n,\alpha_n\}$ be a Fra\"{\i}ss\'e sequence for $\mathbb T$ and take $x=(x_n) \in \mathbb T$. For any $n$ let $I_n$ be the proper chain in $T_n$ from $r_{T_n}$ to $x_n$.  By Observation \ref{monotone-relative}, we have that $\mathbb I=\iLim\{I_n,\alpha_n|_{I_n}\}$ is an arc in $\mathbb T$ joining $x$ and $r_{\mathbb T}$.

\end{proof}

By Theorem \ref{limit-of-hu} and Proposition \ref{arwise-connected} we get the following corollary.

\begin{corollary}\label{dendroid}
If $\mathcal T$ is a projective Fra\"{\i}ss\'e family of rooted trees, and $\mathbb T$ is a projective Fra\"{\i}ss\'e
limit of $\mathcal T$, then $\mathbb T$ is a dendroid.
\end{corollary}

Recall the following definition for continua.
\begin{definition}
A continuum $X$ which is a dendroid is said to be a {\it smooth dendroid} if there exists a point $p\in X$ such that if $x_n\in X$ is a sequence of points that converges to a point $x\in X$ then the sequence of arcs $[p,x_n]$ converges to the arc $[p,x]$.  This is equivalent to saying the order $\le$, defined by $x\le y$ if every arc joining $p$ and $y$ contains $x$, is closed. 
\end{definition}

We now give an analogous definition for a rooted topological graph to be a smooth dendroid.

\begin{definition}
A rooted topological graph $X$ with the root $r_X$ is called a {\it smooth dendroid} if $X$ is a dendroid according to Definition \ref{def:dendrite} and the order $\le$, defined by $x\le y$ if every arc joining $r_X$ and $y$ contains $x$, is closed.
\end{definition}

\begin{observation}
If a rooted topological graph $G$ is a smooth dendroid and $E(G)$ is transitive, then its topological realization $|G|$ is a smooth dendroid in the topological sense.
\end{observation}

\begin{proposition}\label{smooth-dendroid}
If $\mathcal T$ is a projective Fra\"{\i}ss\'e family of rooted trees, and $\mathbb T$ is a projective Fra\"{\i}ss\'e 
limit of $\mathcal T$, then $\mathbb T$ is a smooth dendroid.
\end{proposition}

\begin{proof} 
First, note that by Corollary \ref{dendroid}, $\mathbb T$ is a dendroid. 
Next, we claim the order $\le$
is closed on $\mathbb T$. Indeed, let $\{T_n,\alpha_n\}$ be a Fra\"{\i}ss\'e sequence for $\mathbb T$ and take $x,y \in \mathbb T$.
Then $x\leq y$ iff for every $n$, $x_n\leq y_n$. By the definition of the product topology, $\leq $ is closed on $\mathbb T$. Therefore the dendroid $\mathbb T$ is smooth. 
\end{proof}

The Lelek fan can be characterized as a smooth fan (smooth dendroid having exactly one ramification point) with a dense set of endpoints, \cite{Lelek-fan}. In  \cite{B-K}, it is shown that the topological realization of the projective  Fra\"{\i}ss\'e limit for the family of all rooted trees with order-preserving epimorphisms is the Lelek fan.

\begin{definition}
We say that an epimorphism $f\colon S\to T$ between rooted trees is {\it end vertex preserving} if, besides preserving order, we have that an image of an end vertex in $S$ is an end vertex in $T$.
\end{definition}

\begin{remark}\label{m-l-factorization-rem}
    If in Theorem \ref{m-l-factorization}, $G$ and $H$ are rooted trees and $f$ is order-preserving, then $M$ is a rooted tree and $l$, $m$ are order-preserving. If additionally $f$ is end vertex preserving, then so are $l$ and $m$. 
\end{remark}

The following will be used in Section 8.
\begin{lemma}\label{end-h=gf}
   Consider  epimorphisms $f\colon F\to G$ and $g\colon G\to H$ between rooted trees. If the composition $g\circ f\colon F\to H$ is end vertex preserving then $g$ is end vertex preserving.
\end{lemma}
\begin{proof}
    Let $e$ be an end vertex in $G$.   Because $f$ is order-preserving there is $e' \in f^{-1}(e)$ which is an end vertex of $H$. Since $h$ is end vertex preserving, $g(e)=h(e')$ is an end vertex.
\end{proof}

 Recall the standard amalgamation procedure (Definition \ref{standamal}).

  \begin{theorem}\label{onelight}
Let $A,B,C$ be rooted trees.
 Let $f\colon B\to A$ be light confluent and let $g\colon C\to A$ be confluent. Then $D$ obtained via the standard amalgamation procedure is a rooted tree. The obtained $f_0\colon D\to B$ is confluent and  $g_0\colon D\to C$ is light confluent.
 Moreover, in case $g$ is light then so is $f_0$. In case both $f,g$ are end vertex preserving, then so are $f_0, g_0$.
  \end{theorem}

 \begin{proof}
  
{\bf (Connected)}  We first show that $D$ is arcwise connected. Let $r_A$, $r_B$, and $r_C$ be the roots of $A$, $B$ and $C$ respectively and
 $(b_1,c_1), (b_2,c_2)\in D$. Since $f(b_1)=g(c_1)$, $f,g$ are order-preserving, and there is a unique  arc in $A$  joining $r_A$ to $f(b_1)=g(c_1)$, we have that $f|_{[r_B, b_1]}$ and $g|_{[r_C,c_1]}$ are monotone and onto  the arc $[r_A, f(b_1)=g(c_1)]$. Therefore we can find   an arc joining  $(r_B,r_C)$  and $(b_1,c_1)$ in $D$. Similarly, we can find  an arc joining  $(r_B,r_C)$  and $(b_2,c_2)$ in $D$. The union of the two arcs contains an arc joining $(b_1,c_1)$ and $(b_2,c_2)$.

Note that the map $g_0$ is a light confluent epimorphism by Propositions~\ref{light} and \ref{standard-confluent} and $f_0$ is a confluent epimorphism by Propositions~\ref{standard-confluent}. For the same reason, if $g$ is light then so is $f_0$.

{\bf (Tree)} Next, we show that the graph $D$ is a tree. We first show that there is no pair $(b_1,c_1)$ and $(b_2,c_2)$ such that  $b_1 <_B b_2$ and $c_2 <_C c_1$. Indeed, since $f$ is light, $f(b_1) \not = f(b_2)$, so $f(b_1) <_A f(b_2)$, as $f$ is an epimorphism. Then $g(c_1)=f(b_1) <_A f(b_2)=g(c_2)$. However, since $g$ is an epimorphism, we have $g(c_2) \le_A g(c_1)$, a contradiction.   

Now let $(b_1,c_1)\leq^c_D (b_2,c_2)$ if and only if $b_1\leq_B b_2$ and $c_1\leq_C c_2$. Suppose $D$ is not a tree, then there is a cycle in $D$: $d_1,d_2,\ldots, d_n$, that is, $d_i$ is adjacent to $d_{i+1}$ and $d_n$ is adjacent to $d_1$, and all $d_i$ are pairwise distinct. Then either (a) for every $i$, $d_i\leq^c_D d_{i+1}$ and $d_n\leq^c_D d_1$ or (b) for every $i$, $d_{i+1}\leq^c_D d_{i}$ and $d_1\leq^c_D d_n$   or (c) there is $i$ such that  $d_{i+1}\leq^c_D d_{i}$ and $d_{i-1}\leq^c_D d_{i}$ (addition and subtraction modulo $n$).

The cases (a) and (b) are not possible since, by the transitivity of $\leq^c_D$, each of them implies that both $d_1\leq^c_D d_2$ and $d_2\leq^c_D d_1$ hold. The claim below shows that (c) is also impossible, concluding the proof that $D$ is a tree.

   {\bf Claim.} There are no pairwise different $p=(b_1,c_1), q=(b_2,c_2), r=(b_3,c_3)$ such that $p$ and $q$ are adjacent and $p<^c_D q$ as well as $q$ and $r$ are adjacent and $r<^c_D q$.
     \begin{proof}[Proof of Claim] 
         Suppose towards the contradiction that there are such $p,q,r$. Then since $B$ and $C$  are trees we have $b_2=b_3$ or $b_1=b_2$ or $b_1=b_3$, and respectively $c_2=c_3$ or $c_1=c_2$ or $c_1=c_3$. Because $g_0$ is light, only $c_1=c_3$ is possible. Then $f(b_1)=f(b_3)$ and, since $f$ is light, $b_1$ and $b_3$ are  non-adjacent and distinct. If $b_1=b_2$ ($b_2=b_3$), then because $q$ is adjacent to $r$ ($p$ is adjacent to $q$) it follows that $b_1$ and $b_3$ are adjacent. A contradiction.
     \end{proof}

{\bf (Order preserving)} Next, we show that $f_0$ and $g_0$ preserve the ``rooted tree order''. Declare the root of $D$ to be  $r=(r_B,r_C)$. There are two orders on $D$, the usual order on rooted trees given by $(b_1,c_1)\leq_D (b_2,c_2) $ if and only if the arc joining $r$ and $(b_2,c_2)$ contains $(b_1,c_1)$ and the component order given by $(b_1,c_1)\leq^c_D (b_2,c_2)$ if and only if $b_1\leq_B b_2$ and $c_1\leq_C c_2$. We now show that $(b_1,c_1)\leq_D (b_2,c_2) $ if and only if $(b_1,c_1)\leq^c_D (b_2,c_2)$ which implies that $f_0$ and $g_0$ preserve the rooted tree order.

 Suppose that $(b_1,c_1)\leq^c_D (b_2,c_2)$.  Note that $f|_{[r_B,b_2]}$ is monotone and onto $[r_A,f(b_2)]$. Similarly, $g|_{[r_C,c_2]}$ is monotone and onto $[r_A,f(b_2)]$. Therefore we can find  an arc joining $(r_B,r_C)$ and $(b_2,c_2)$, which contains $(b_1,c_1)$. This shows that if $(b_1,c_1)\leq^c_D (b_2,c_2) $ then $(b_1,c_1)\leq_D (b_2,c_2)$.

 Suppose now that $(b_1,c_1) \leq_D (b_2,c_2)$ and $(b_1,c_1)$ and $(b_2,c_2)$ are adjacent. By the definition of adjacent $(b_1,c_1) \not = (b_2,c_2)$. Denote $p=f(b_1)=g(c_1)$ and $q=f(b_2)=g(c_2)$. In the case that $b_1 \not = b_2$ we have, since $f$ is light, that $p\neq q$.  Since $b_1$ and $b_2$ are adjacent, they are comparable wrt $\leq_B$, and analogously, $c_1$ and $c_2$ are comparable wrt $\leq_C$. Therefore $p$ and $q$ are comparable  wrt $\leq_A$. Thus $p<_A q$ or $q<_A p$. In the first case, since $f$ and $g$ are order-preserving, we get $(b_1,c_1)\leq^c_D (b_2,c_2)$. In the second case, we would get $(b_2,c_2)\leq^c_D (b_1,c_1)$. However, the second case is impossible, since it implies $(b_2,c_2)<_D (b_1,c_1)$, as we saw in the previous paragraph.

 Now consider the case that $b=b_1=b_2$ so $c_1 \not = c_2$, since $g_0$ is light. We want to show that $c_1 \leq_C c_2$. Suppose towards a contradiction that $c_2<_C c_1$. Let $p=f(b)=g(c_1)=g(c_2)$. 
 Note that $f|_{[r_B,b]}$ is monotone and onto $[r_A,p]$. Similarly, $g|_{[r_C,c_1]}$ is monotone and onto $[r_A,p]$. Since $c_2\in [r_C,c_1]$ the arc joining $(r_B,r_C)$ and $(b,c_1)$ contains $(b,c_2)$. This gives a contradiction with the fact that  $(b,c_1)\leq_D (b,c_2)$.
 
The general case, where $(b_1,c_1)\leq_D (b_2,c_2)$ but $(b_1,c_1)$ and $(b_2,c_2)$ are not necessarily adjacent, we can reduce to the case just discussed by taking the arc joining  $(b_1,c_1)$  and $(b_2,c_2)$ and considering consecutive pairs in it one by one.

{\bf (End vertex preserving)} Finally, we show that if $f$ and $g$ are end vertex preserving then $f_0$ and $g_0$ are end vertex preserving. Suppose $f$ and $g$ are end vertex preserving and let $(b,c)$ be an end vertex of $D$, i.e. a maximal vertex with respect to $\leq_D$. We have to show that $b$ is an end vertex of $B$ and $c$ is an end vertex of $C$.

 Using that $f$ is light we show first that $c$ is an end vertex of $C$. 
Suppose towards a contradiction that $c$ is not maximal with respect to $\leq_C$. Then there is $c<_C t$ with $\langle c,t\rangle\in E(C)$. In case $g(c)=g(t)$, we have $(b,c)<_D (b,t)$, which contradicts the maximality of $(b,c)$, so suppose that  $g(c)<_A g(t)$. Since $f$ is light and confluent there is $b<_B s$ with  $\langle b, s\rangle\in E(B)$, such that $f(s)=g(t)$. But then $(b,c)<_D (s,t)$, which again contradicts the maximality of $(b,c)$. This concludes the proof that $c$ is maximal wrt $\leq_C$, i.e. it is an end vertex of $C$. Now as $g$ is end vertex preserving, we have that $g(c)$ is an end vertex of $A$. In case there is $b<_B s$ with  $\langle b, s\rangle\in E(B)$, since $f$ is order-preserving, we must have $f(s)=f(b)=g(c)$, which is impossible as $f$ is light. This shows that $b$ is an end vertex of $B$.

 \end{proof}

\begin{definition}\label{def:elc}
An epimorphism $f:G\to H$ between rooted trees is called {\it elementary light confluent} if there is a vertex $v\in G$ and two components $C_1$ and $C_2$ of $G\setminus \{v\}$ such that:
\begin{enumerate}
     \item $f|_{C_1}$ and $f|_{C_2}$ are one-to one;
     \item $f(C_1)=f(C_2)$ is a component of
     $H\setminus f(v)$;
     \item $f|_{ (G\setminus (C_1\cup C_2))}$ is one-to-one.
     \end{enumerate}

 A typical elementary light confluent epimorphism is pictured in Figure \ref{fig:def-elem-light}.
\end{definition}

\begin{figure}[!ht]
	\includegraphics*[scale=1]{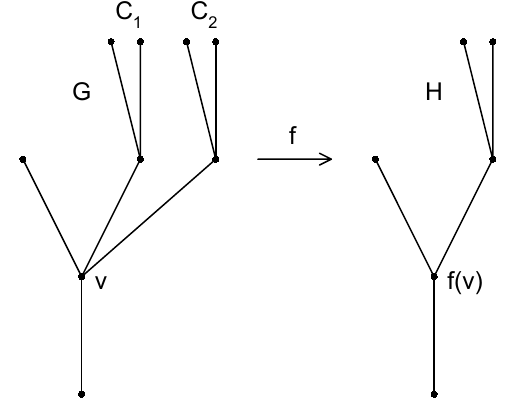}
	\caption{Elementary light confluent epimorphism }
	\label{fig:def-elem-light}
\end{figure}
It will follow from Corollary \ref{elemenlightc} that any light confluent epimorphism decomposes into elementary light confluent epimorphisms.

\section{Kelley topological graphs}

We first recall definitions concerning Kelley continua and adapt these to the setting of topological graphs. Results of this section will be applied in Theorem~\ref{summerizing-confluent}.

\begin{notation}
For a given metric space $X$ with a metric $d$, the $r$-neighborhood of a closed set $A\subseteq X$ is
$N(A,r)=\{x\in X: \text{ there is } a\in A:d(x,a)<r\} $.
The Hausdorff distance $H$ between closed subsets of $X$ is
defined by $H(A,B)=\inf\{r>0: A\subseteq N(A,r) \text{ and } B\subseteq N(A,r)\}$.
\end{notation}

\begin{definition}
A continuum $X$ is said to be {\it Kelley continuum} if for every subcontinuum $K$ of $X$,
every $p\in K$, and every sequence $p_n\to p$ in $X$ there are subcontinua $K_n$ of $X$ such that $p_n\in K_n$ and
$\lim K_n=K$. By compactness, this definition is equivalent to saying that for every $\varepsilon>0$ there is $\delta>0$ such that for each two points $p,q\in X$ satisfying $d(p,q)<\delta$, and for each subcontinuum $K\subseteq X$ such that $p\in K$ there is a subcontinuum $L\subseteq X$ satisfying $q\in L$ and $H(K,L)<\varepsilon$.
\end{definition}
It is well known that every Kelley dendroid is smooth, see \cite{Czuba}.

\begin{definition}
A topological  graph $X$ is called {\it Kelley} if $X$ is connected and for every closed and connected set $K\subseteq X$,
every vertex $p\in K$, and every sequence $p_n\to p$ of vertices in $X$ there are closed and connected sets $K_n$ such that $p_n\in K_n$ and
$\lim K_n=K$. By compactness, this definition is equivalent to saying that for every $\varepsilon>0$ there is $\delta>0$ such that for each two vertices $p,q\in X$ satisfying $d(p,q)<\delta$, and for each closed connected set $K\subseteq X$ such that $p\in K$ there is a closed connected set $L\subseteq X$ satisfying $q\in L$ and $H(K,L)<\varepsilon$. 

\end{definition}
Note that every finite graph is Kelley. Indeed, if we take the discrete metric on finite graphs to have the value of $1$ when vertices are not equal then taking $\delta= 1/2$ in the definition implies $p=q$ and so we can take $K_n$ to equal $K$. 
\begin{observation}\label{Kelley-observation}
If a  topological graph $G$ is Kelley and $E(G)$ is transitive, then its topological realization $|G|$ is a Kelley continuum.
\end{observation}

In \cite[Theorem 3.1]{Ingram} it is shown that if $\{X_i,f_i\}$ is an inverse sequence where the $X_i$'s are Kelley continua and the $f_i$'s are confluent bonding maps then the inverse limit space $\iLim\{X_i,f_i\}$ is a Kelley continuum. The proof of the following theorem is essentially the one given in \cite{Ingram} adapted for the present setting.

\begin{proposition}\label{Kelley-theorem}
If $\mathcal G$ is a projective Fra\"{\i}ss\'e family of connected finite graphs with confluent epimorphisms, then the projective Fra\"{\i}ss\'e limit is Kelley.
\end{proposition}

\begin{proof}

Let $\{G_n,\alpha_n\}$ be a Fra\"{\i}ss\'e sequence for $\mathcal G$ so ${\mathbb G} = \iLim\{G_n,\alpha_n\}$. Let the metric 
on ${\mathbb G}$ be given by $d(x,y)=\sum_{i=1}^\infty d_i(x_i,y_i)/2^i$ where $d_i$ is 
the discrete metric on the finite graphs $G_i$. Note that if $x,y \in {\mathbb G}$ and $N$ 
is a positive integer then $d(x,y)\leq 1/2^N$ if and only 
if $\alpha^\infty_i (x)=\alpha^\infty_i (y)$ for all $i \le N$. Thus, if $K$ and $L$ 
are closed connected graphs in ${\mathbb G}$ such that $\alpha_N(K) = \alpha_N(L)$ then $H(K,L) \leq  1/2^N$. 

Let $\varepsilon>0$ be given and let $N$ be such that $1/2^N < \varepsilon$.
Suppose $K$ is a closed connected graph in ${\mathbb G}$, $p$ is a vertex in $K$, $q$ is a vertex in ${\mathbb G}$, and $d(p,q) < 1/2^N$. 
Then $\alpha^\infty_N(q) \in \alpha^\infty_N(K)$. For each $m>N$ let $L_m$ be the component of $(\alpha_N^m)^{-1}(\alpha_N^\infty(K))$ that contains $\alpha_m^\infty(q)$ and for $m \le N$ let $L_m = \alpha_m^N(\alpha_N^\infty(K))$. Then $L=\iLim\{L_i,\alpha_i|_{L_i}\}$ being the inverse limit of closed connected graphs is a closed connected topological graph in ${\mathbb G}$ containing $q$ and $H(K,L)< \varepsilon$.

\end{proof}

\section{End vertices and ramification vertices in topological graphs}

In this section we define ramification and end vertices for topological graphs and study how they project onto the topological realization of a projective Fra\"{\i}ss\'e limit.

\begin{definition}\label{rampt-topo-dend}
For a topological graph $D$ that is a dendroid, we say that a vertex $v$ has {\it order} at least $n$, in symbols $\ord(v)\ge n$,
if there are arcs $A_1, A_2, \dots A_n$ such that $A_i\cap A_j=\{v\}$ for $i,j\in \{1,2, \dots n\}$ satisfying $i\ne j$. We
define $\ord(v)=n$ if $\ord(v)\ge n$ and $\ord(v)\ge n+1$ is not true. If $\ord(v) \ge n$ for each positive integer $n$, then we say that $v$ has infinite order, in symbols $\ord(v)=\infty$.
Vertices of order 1 are called {\it end vertices} and vertices of order $\ge 3$ are called {\it ramification vertices}.
Note that these definitions agree with Definitions \ref{definition-order} in the case of trees.

We call an end vertex {\it isolated} if it belongs to a non-degenerate edge. Note that, according to this definition, every end vertex in a finite graph is isolated.
\end{definition}

Let us recall a definition from continuum theory.

\begin{definition}
For a continuum $X$ a point $p\in X$ is called a {\it ramification point of $X$ in the classical sense} if there are three arcs $A,B,C$ in $X$ such that
$A\cap B=A\cap C=B\cap C=\{p\}$.
Similarly, a point $p\in X$ is called an {\it endpoint of $X$ in the classical sense} if it is an endpoint of every arc in $X$ that contains $p$.
\end{definition}

Suppose that $D$ is a dendroid, $e$ is an isolated end vertex, and $r$ is a ramification vertex such that $\langle r, e\rangle\in E(D)$. If $E(D)$ is transitive, $\varphi\colon D\to |D|$ is the quotient map, then, since $\varphi(r)= \varphi(e)$, $\varphi(e)$ is not an endpoint. The proposition below captures that a converse to this holds.
\begin{proposition}\label{endpts-to-endpts}
Suppose  $D$ is a topological graph that is a dendroid, the set of edges of $D$ is transitive and $e$ is a non-isolated end vertex of $D$. Let $\varphi$ be the quotient map from $D$ onto its topological realization $|D|$. Then $\varphi(e)$ is an endpoint of $|D|$ in the classical sense.
\end{proposition}

\begin{proof}
 Assume the conclusion is not true.  Then there exists an arc $A \subseteq |D|$ such that $A\setminus\{\varphi(e)\}=H \cup K$ where $H$ and $K$ are nonempty segments such that $\overline H \cap K = H \cap \overline K = \emptyset$. 
The set $\varphi^{-1}(H\cup K)$ is not connected since 
$\varphi(\varphi^{-1}(H\cup K))=H\cup K$ and connectedness is preserved by continuous maps.  However the sets $\varphi^{-1}(H)$ and $\varphi^{-1}(K)$ are each connected.  Let $a$ and $b$ be vertices in $\varphi^{-1}(H)$ and $\varphi^{-1}(K)$ respectively and $A_H$ and $A_K$ be arcs in $D$ joining $a$ to $e$ and $b$ to $e$. The arcs $A_H$ and $A_K$ then lie in $\varphi^{-1}(H)\cup \{e\}$ and $\varphi^{-1}(K)\cup \{e\}$ respectively. These sets are closed and they are connected as $e$ is non-isolated. The intersection of $A_H$ and $A_K$ is $\varphi^{-1}(\varphi(e))=\{e\}$ since $e$ is a non-isolated end vertex of $D$. Thus $\ord(e)\ge 2$ so $e$ is not an end vertex of $D$ contrary to the hypothesis.
\end{proof}

\begin{proposition}\label{end-to-end}
Let $f\colon X\to Y$ be an order-preserving epimorphism between topological graphs, which are smooth dendroids. Then, for every end vertex $y\in Y$ there is an end vertex $x\in X$ such that $f(x)=y$.
\end{proposition}
\begin{proof}
Let $z$ be any vertex in $X$ such that $f(z)=y$, and let $x$ be an end vertex of $X$ satisfying $z\le x$. Since $f$ is order-preserving we have $f(z)=y\le f(x)$, but $y$ is an end vertex, so $f(z)=f(x)=y$.
\end{proof}

\begin{proposition}\label{ram-to-ram}
Suppose  $D$ is a topological graph that is a dendroid, the set of edges of $D$ is transitive and $r$ is a ramification vertex of $D$. Let $\varphi$ be the quotient map from $D$ onto its topological realization $|D|$. If $\varphi^{-1}(\varphi(r))$ contains no isolated end vertices,  then $\varphi(r)$ is a ramification point of $|D|$ in the classical sense.
\end{proposition}

\begin{proof}
Suppose $r$ is a ramification vertex of $D$ and let $A$, $B$, and $C$ be arcs in $D$ such that $A\cap B=B\cap C=C\cap A=\{r\}$. The images $\varphi(A)$, $\varphi(B)$, and $\varphi(C)$ are arcs, possibly  degenerate, with pairwise intersections equaling $\varphi(r)$.  Suppose $\varphi(A)$ is degenerate then $A$ consists of a single edge $\langle r,e\rangle$ and $A=\varphi^{-1}(\varphi(r))$ contains the isolated end vertex $e$ contrary to the hypothesis.  Thus $\varphi(r)$ is a ramification point of $|D|$.

\end{proof}

\begin{definition}\label{splitting map}
    Let $S,T$ be rooted trees. An order-preserving epimorphism $f\colon S\to T$ will be called a {\it splitting edge map} if $S$ is obtained from $T$ as follows. Let $V(S)=V(T)\cup\{x\}$, where $x$ is a ``new'' vertex, and for some edge 
$\langle a,b\rangle \in E(T)$, we have  
$E(S)=(E(T)\setminus \langle a,b\rangle )\cup \{\langle a,x\rangle , \langle x,b\rangle\}$.
The map $f$ is the identity on $T$ and it maps $x$ either to $a$ or to $b$.
\end{definition}

\begin{theorem}\label{rooted-end}
Suppose that $\mathcal T$ is a projective Fra\"{\i}ss\'e family of rooted trees with a subfamily of confluent epimorphisms, which contains all  splitting edge maps.
Then the projective Fra\"{\i}ss\'e limit of $\mathcal T$ has no isolated end vertices.

\end{theorem}

\begin{proof}

      Let $\mathcal T$ be a family that satisfies the assumptions of the theorem, and let $\mathbb T$ be the projective Fra\"{\i}ss\'e limit of $\mathcal T$. Suppose $e$ is an isolated end vertex of $\mathbb T$ then there exists $a\in \mathbb T$ such that $\langle a,e\rangle\in E(\mathbb T)$. Let $T\in \mathcal T$ and $f_T:\mathbb T \to T$ be an epimorphism such that $f_T(a)\not = f_T(e)$, denote $a_T=f_T(a)$ and $e_T=f_T(e)$, then $\langle a_T,e_T\rangle \in E(T)$. Let $f\colon S\to T$ be the splitting edge map that replaces the edge $\langle a_T,e_T\rangle$ with $\langle a_T,x\rangle$ and $\langle x,e_T\rangle$ and $f(x) = e_T$. Denote the copies of $e_T$ and $a_T$ in $S$ by $e_S$ and $a_S$ respectively.  Let $f_S\colon \mathbb T\to S$ be an epimorphism such that $f_T=f\circ f_S$. 
    Since $f_S$ is edge preserving, we must have $f_S(e)=x$ (it carries $\langle a,e\rangle$ to  $\langle a_S,x\rangle$).
    Then the component of $f_S^{-1}(\{x, e_S\})$ containing $\{e\}$ must be equal to $\{e\}$, which is impossible as $\{e\}$ is a singleton and $f_S$ is confluent. 
  
\end{proof}

\section{Simple- and simple*-confluent epimorphisms and rooted trees}

In this section, we use the results obtained 
earlier to investigate families of rooted trees with confluent order-preserving epimorphisms.

 \subsection{Epimorphisms between rooted trees}

Families of rooted trees with all confluent or even all monotone epimorphisms do not form a projective Fra\"{\i}ss\'e family, they do not amalgamate, as the following example shows. Therefore, we introduce subfamilies of monotone epimorphisms consisting of simple-monotone or   simple*-monotone epimorphisms. Composing these epimorphisms with light confluent epimorphisms we obtain simple-confluent and simple*-confluent epimorphisms, which have the amalgamation property.

\begin{example}\label{not-order-pres}

Let $A$ be the tree with vertices $V(A)=\{a_A,b_A,c_A,x_A\}$ and edges $E(A)=\{\langle x_A,a_A\rangle,  \langle x_A,b_A\rangle, \langle x_A,c_A\rangle\}$ with $x_A$ being the root of $A$, $B$ be the tree with vertices $V(B)=\{a_B,b_B,c_B,x_B,y_B\}$ and edges $E(B)=\{\langle y_B,a_B\rangle,\langle y_B,b_B\rangle,\langle x_B,y_B\rangle, \langle x_B,c_B\rangle\}$ with $x_B$ being the root of $B$, and $C$ be the tree with vertices $V(C)=\{a_C,b_C,c_C,x_C,y_C\}$ and edges $E(C)=\{\langle x_C,a_C\rangle,\langle y_C,b_C\rangle,\langle y_C,c_C\rangle, \langle x_C,y_C\rangle\}$ with $x_C$ being the root of $C$. For epimorphisms $f\colon B \to A$ and $g \colon C \to A$ let $f(x_B)=f(y_B)=x_A=g(x_C)=g(y_C)$ and otherwise map letters correspondingly; see Figure \ref{fig2}.

 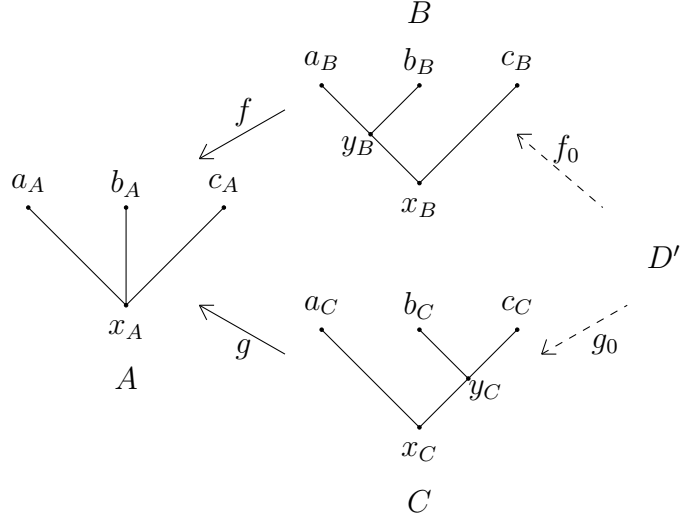
\begin{figure}[!ht] 
\begin{center}
\begin{tikzpicture}[scale=0.65]

  \draw (0,2.5) -- (-2,4.5);
  \draw (0,2.5) -- (0,4.5);
  \draw (0,2.5) -- (2,4.5);
    \filldraw[black] (0,2.5) circle (1pt);
    \filldraw[black] (-2,4.5) circle (1pt);
    \filldraw[black] (0,4.5) circle (1pt);
    \filldraw[black] (2,4.5) circle (1pt);
    \node at (0,2.0) {$x_A$};
    \node at (0,1.0) {$A$};
    \node at (-2,5) {$a_A$};
    \node at (0,5) {$b_A$};
    \node at (2,5) {$c_A$};

    \draw (6,5) -- (4,7);
    \draw (6,5) --(8,7);
    \draw (5,6) -- (6,7);
        \filldraw[black] (6,5) circle (1pt);
        \filldraw[black] (4,7) circle (1pt);
        \filldraw[black] (5,6) circle (1pt);
        \filldraw[black] (6,7) circle (1pt);
        \filldraw[black] (8,7) circle (1pt);
    \node at (6,4.5) {$x_B$};
    \node at (4,7.5) {$a_B$};
    \node at (6,7.5) {$b_B$};
    \node at (6,8.5) {$B$};
    \node at (8,7.5) {$c_B$};
    \node at (4.75,5.75) {$y_B$};

    \draw (6,0) -- (4,2);
    \draw (6,0) --(8,2);
    \draw (7,1) -- (6,2);
        \filldraw[black] (6,0) circle (1pt);
        \filldraw[black] (4,2) circle (1pt);
        \filldraw[black] (7,1) circle (1pt);
        \filldraw[black] (8,2) circle (1pt);
        \filldraw[black] (6,2) circle (1pt);
    \node at (6,-0.5) {$x_C$};
    \node at (6,-1.5) {$C$};
    \node at (4,2.5) {$a_C$};
    \node at (6,2.5) {$b_C$};
    \node at (8,2.5) {$c_C$};
    \node at (7.35,0.75) {$y_C$};

    \node at (11,3.5) {$D'$};

    \draw (1.5,5.5) -- (3.25,6.5);
    \draw (1.5,5.5) -- (1.8,5.5);
    \draw (1.5,5.5) -- (1.65,5.8);
    \node at (2.4,6.45) {$f$};

    \draw (1.5,2.5) -- (3.25,1.5);
    \draw (1.5,2.5) -- (1.8,2.5);
    \draw (1.5,2.5) -- (1.65,2.2);
    \node at (2.4,1.6) {$g$};

    \draw [dashed] (8,6) -- (9.75,4.5);
    \draw (8,6) -- (8.3,6);
    \draw (8,6) -- (8.1,5.7);
    \node at (9,5.7) {$f_0$};

    \draw [dashed] (8.5,1.5) -- (10.25,2.5);
    \draw (8.5,1.5) -- (8.8,1.5);
    \draw (8.5,1.5) -- (8.6,1.8);
    \node at (9.8,1.7) {$g_0$};

\end{tikzpicture}
\caption{Failure of amalgamation for monotone epimorphisms and rooted trees}\label{fig2}
\end{center}
\end{figure}

\end{example}

Clearly, $f$ and $g$ are monotone epimorphisms. By Theorem \ref{amalgamation-monotone}, when $A$, $B$, and $C$ are viewed as unrooted trees, there is an amalgamation. In particular, the tree $D$ having vertices 
$\{a_D,y_D,b_D,x_D,z_D,c_D\}$ and edges $\{\langle y_D,a_D\rangle, \langle x_D,y_D\rangle, \langle x_D,b_D\rangle, \langle x_D,z_D\rangle, \langle z_D,c_D\rangle\}$ is an amalgamation with $f_0(x_D)=y_B$, $f_0(y_D)=y_B$, $f_0(z_D)=x_B$, $g_0(y_D)=x_C$, and $g_0(x_D)=g_0(z_D)=y_C$. No matter which vertex is selected as the root of $D$, either $f_0$ or $g_0$ does not map it to the root of $B$ or $C$. So this is not an amalgamation of rooted trees as one of $f_0$ or $g_0$ is not order-preserving.

  In fact, there is no amalgamation of these rooted trees. Suppose, towards a contradiction, that there is a rooted tree $D'$ and confluent epimorphisms $f_0\colon D'\to B$, $g_0\colon D'\to C$  satisfying  $f\circ f_0= g\circ g_0$. Since $f \circ f_0$ is an epimorphism there exists $t, s\in D'$  such that $\langle t,s\rangle\in E(D')$, $f\circ f_0(t)=x_A$, and $f\circ f_0(s)=b_A$. Then $f_0(t)=y_B$ and $g_0(t)=y_C$. Therefore, the set $Y=\{t\in D'\colon f_0(t)=y_B,\ g_0(t)=y_C\}$ is nonempty. Let $t \in Y$ and consider the arc $[r_{D'},t]$ from $r_{D'}$ to $t$. Let $u$ be the largest vertex on this arc that is not in $Y$. Then there is vertex $t' \in Y \cap$ $[r_{D'},t]$ such that $\langle u,t'\rangle$ is an edge in $D'$. Since $u \not \in Y$,  $u < t'$, $u$ and $t'$ are adjacent, we have $f_0(u)=x_B$ or $g_0(u)=x_C$; without loss of generality, $g_0(u)=x_C$.

  Let  $K$ be the component of $g_0^{-1}(y_C)$ that contains $t'$. Then $t'$ is the least vertex of $K$. If $v$ adjacent to $K$ and $v \not = u$ then $v > t$ for some $t \in K$. Since $t'\in Y$, we have $f_0(t')=y_B$. So, because $f_0$ and $g_0$ are order preserving, we have $f_0(v)\in\{y_B, a_B, b_B\}$ and $g_0(v)\in\{b_C, c_C\}$. Therefore, $f\circ f_0(v)\in\{a_A, b_A\}$ and $g\circ g_0(v)\in\{ b_A, c_A\}$, which implies $f\circ f_0(v)=g\circ g_0(v)=b_A$.  This gives a contradiction with confluence of $g_0$ (applied to the edge $\langle y_C, c_C\rangle$ and the component $K$ of $g_0^{-1}(y_C)$).

\vskip 10 pt

To obtain a projective Fra\"{\i}ss\'e families of  rooted trees we start with the following definitions.

\begin{definition}\label{def:sc}
We will call an epimorphism {\it simple-monotone} if it is a composition of splitting edge maps, see Definition~\ref{splitting map}.
We call an epimorphism a {\it simple-confluent} if it is a composition of light confluent epimorphisms and simple-monotone epimorphisms. Specifically, $f$ is simple-confluent if there are $f_1, f_2, \ldots, f_n$, for some $n$, such that $f=f_1\circ f_2\circ\ldots\circ f_n$, where each $f_i$ is light confluent or simple-monotone.

 \end{definition}

\begin{definition}\label{def:sstarc}
Let $S,T$ be rooted trees. An order-preserving epimorphism $f\colon S\to T$ will be called an {\it adding edge map} if $S$ is obtained from $T$ as follows. Let
$V(S)=V(T)\cup\{y\}$, where $y$ is a ``new'' vertex, and for some vertex  $v\in T$  we have
$E(S)=E(T)\cup\{\langle v,y \rangle\}$.
 The map $f\colon S\to T$ takes $y$ to $v$ and is equal to the identity otherwise.
Define simple*-monotone epimorphisms to be compositions of splitting edge maps and adding edge maps. Let simple*-confluent epimorphisms be compositions of simple*-monotone epimorphisms and light confluent epimorphisms. 
\end{definition}

Note that simple-confluent epimorphisms are end vertex preserving. Also, the epimorphisms in Example \ref{not-order-pres} are neither simple-confluent nor simple*-confluent.

\begin{notation}
The family of all rooted trees and simple-confluent epimorphisms is denoted by $\treend$ and the family of all rooted trees and simple*-confluent  epimorphisms is denoted by $\treecon$. We assume that every rooted tree has at least two vertices.
\end{notation}

\subsection{Amalgamation of simple-confluent epimorphisms}

The goal of this subsection is to prove that the family $\treend$  of all rooted trees with simple-confluent epimorphisms is a projective Fra\"{\i}ss\'e family, see Theorem~\ref{confluent-projective-family}.

\begin{definition}\label{cone-over}

Let $T$ be a tree and $a \in T$. A vertex $b \in T$ is an {\it immediate successor of $a$} if $ \langle a,b\rangle\in E(T)$ and $ a<_T b$. Similarly, $b$ is an {\it immediate predecessor of $a$} if $ \langle b,a\rangle\in E(T)$ and $ b<_T a$. 

We define recursively the {\it height} of each vertex in $T$. If $r$ is the root of $T$, we let $\htt(r)=0$.  If $a$ is such that $\htt(a)=n$ and $b$ is an immediate successor of $a$ then we let $\htt(b)=n+1$. It is not hard to see that for an $a\in T$, $\htt(a)$ is just the number of edges in the arc joining the root and $a$. Finally, let $\htt(T)=\max\{\htt(a)\colon a\in T\}$, be the height of $T$.  

For a vertex $a \in T$ we denote the set of immediate successors of $a$ by $\im(a)$ and the {\it successor order} of $a$, denoted $\sord(a)$,
will simply be $|\im(a)|$. Let $\sord(T)=\max\{\sord(a)\colon a\in T\}$.

A rooted tree $T$ is {\it regular} if for each vertex $a \in T$ that is not an end vertex \sord(a)= \sord(T) and for each end vertex $e\in T$, $\htt(e)=\htt(T)$. 

If $T$ is a rooted tree and $a\in T$, let  $C_a=\{x\in T\colon x> a\}$ be the set of vertices above $a$ with the edge relation induced from $T$. We also let $\overline{C}_a=\{x\in T\colon x\geq  a\}$. 
\end{definition}

\begin{theorem}\label{Joint-Proj-pro}
The family $\treend$ has the JPP. 
\end{theorem}

\begin{proof}
Let $A$ and $B$ be rooted trees with $k=\max\{\htt(A), \htt(B)\}$ and $m = \max\{\sord(A), \sord(B)\}$.
 It suffices to show that if $T$ is a rooted tree, $k\geq \htt(T)$, and $m\geq \sord(T)$, then  for the regular rooted tree $S$, with $\htt(S)=k$ and 
 $\sord(S)=m$, we can find a simple-confluent epimorphism from $S$ to $T$. 

 First, we get a rooted tree $R$ such that there is a 
 simple-monotone epimorphism from $R$ onto $T$ and each end vertex of $R$ has height $k$. For that, we proceed with the following algorithm. For any end vertex $e\in T$, if $a_e$ is the immediate predecessor of $e$, and $\htt(e)=k'<k$, replace the edge $\langle a_e, e\rangle$ by $k-k'+1$ many edges. This gives the required $R$. Mapping the added vertices to the respective end vertices gives a simple-monotone epimorphism from $R$ to $T$.

To get from $R$ to the required tree $S$,  we will proceed recursively as follows. Suppose that we have already constructed a tree $S_i$ with $k=\htt(S_i)$, all branches of $S_i$ are of the same height, $m\geq \sord(S_i)$, and such that for some vertex $b \in S_i$, $\sord(b) < m$ and for every vertex $a\in S_i$ with $a>b$, $\sord(a)=m$, as well as a light confluent epimorphism $f_i\colon S_i\to R$. Suppose $\sord(b)=m'$.
 Pick a component $C$ of $C_b$ and replace $C$ by $m-m'+1$ many copies of $C$. This gives a tree $S_{i+1}$ with one fewer vertex having successor order less than $m$. Mapping all of the new components in $S_{i+1}$ to $C$ gives a light confluent epimorphism from $S_{i+1}$ to $S_i$.

 \end{proof}

\begin{theorem}\label{confluent-projective-family}
The family  $\treend$ of all rooted trees with simple-confluent epimorphisms is a projective Fra\"{\i}ss\'e family.

\end{theorem}

We need to verify condition (4) of Definition \ref{definition-Fraisse}.  Theorem \ref{confluent-projective-family} will follow from  Theorem \ref{onelight} and the two lemmas below.
Let  $f\colon B\to A$ and $g\colon C\to A$ be simple-confluent epimorphisms. We find $D\in \treend$ and simple-confluent epimorphisms $f_0\colon D\to B$ and $g_0\colon D\to C$ such that $f\circ f_0=g\circ g_0$. We will consider  the following  three cases: 
\begin{itemize}
    \item[(i)] $f,g$ are light confluent. Then we find $f_0, g_0$ light confluent. (Theorem \ref{onelight})
    
    \item[(ii)]  $f,g$ are simple-monotone. Then we find $f_0, g_0$  simple-monotone. (Lemma \ref{2monot})

    \item[(iii)] $f$ is simple-monotone and $g$ is light confluent. Then we find $f_0$ light confluent and $g_0$ simple-monotone. (Lemma \ref{monoandlight})
\end{itemize}
This will suffice as every simple-confluent epimorphism is a composition of light confluent and simple-monotone epimorphisms.  Diagram (D3) illustrates the scheme of the general proof for the case when each $f$ and  $g$ is a composition of a single simple-monotone and a single light confluent epimorphism.

\begin{equation}\tag{D3}
\begin{tikzcd}
&&B\arrow[ld,swap,"m"]\\
&*\arrow[ld,swap,"l"]&&* \arrow[ld,swap,dotted,"m"]\arrow[lu,dotted,swap,"l"]\\
A&&*\arrow[lu,dotted,swap,"l"]\arrow[ld,dotted,"l"]&&D\arrow[lu,dotted,swap,"m"]\arrow[ld,dotted,"m"]\\
&*\arrow[lu,"l"]&&*\arrow[lu,dotted,"m"]\arrow[ld,dotted,"l"]\\
&&C\arrow[lu,"m"]
\end{tikzcd}
\end{equation}

  \begin{lemma}\label{monoandlight}
 Let $f\colon B\to A$ be a simple-monotone epimorphism and let $g\colon C\to A$ be a light confluent epimorphism. Let $D$ be obtained via the standard amalgamation procedure and let $f_0$ and $g_0$ be the corresponding projections. 
 Then $f_0$ is light confluent and $g_0$ is simple-monotone.
 \end{lemma}
 
 \begin{proof}
 By Theorem \ref{onelight},  all we are left to show is that $g_0$ is simple-monotone.
 
 Since any simple-monotone epimorphism is a composition of splitting edge maps, we can assume without loss of generality that 
 $f$ is a  splitting edge map. Therefore we may assume that $f$ replaces an edge $\langle a,b\rangle\in E(A)$ by edges
 $\langle a,x\rangle$ and $\langle x,b\rangle$, and that $f(x)=a$.
 Let  $\langle a_1,b_1\rangle,  \langle a_2,b_2\rangle, \ldots,  \langle a_n,b_n\rangle$ be all edges mapped to $\langle a,b\rangle$ by $g$. Note that $D$ is obtained from $C$ by splitting each of
  $\langle a_i,b_i\rangle$, $i=1,2,\ldots,n$, by some vertex $x_i$ that is mapped by $g_0$ to $a_i$. Therefore $g_0$ is a composition of $n$ many  splitting edge maps, as needed.

 \end{proof}

\begin{lemma}\label{2monot}
 Let $f\colon B\to A$ and $g\colon C\to A$ be simple-monotone epimorphisms. Then there is a rooted tree $D$ and simple-monotone epimorphisms $f_0\colon D\to B$ and $g_0\colon D\to C$ such that
 $f\circ f_0=g\circ g_0$. 
 \end{lemma}

   \begin{proof}
It suffices to show that if 
 $f\colon B\to A$ and $g\colon C\to A$  are  splitting edge maps, then we can find a rooted tree $D$ and  splitting edge maps $f_0\colon D\to B$ and $g_0\colon D\to C$ such that
 $f\circ f_0=g\circ g_0$. 

The $B$ be obtained from $A$ by splitting some edge $e_B $ of $A$. The new vertex $x$ is mapped by $f$ to one of the end vertices of $e_B$. Similarly,
$C$ is  obtained from $A$ by splitting some edge $e_C $ of $A$. The new vertex $y$ is mapped by $g$ to one of the end vertices of $e_C$. 

Let $D$ be obtained from $A$ as follows. In case $e_B\neq e_C$, then we obtain $D$ by simultaneously splitting $e_B$ and $e_C$. If $e_B=e_C=e$, then we split $e$ twice. It is clear how the required $f_0$ and $g_0$ are defined and that they are   splitting edge maps.

 \end{proof}

\begin{notation}
Let $\mathbb{T}_{\mathcal{ CE}}$ be the projective Fra\"{\i}ss\'e limit of the family~$\treend$.
\end{notation}

\subsection{More on simple-confluent epimorphisms}\label{sec:simplecon}

We introduce a family of epimorphisms between rooted trees that we call special. 
It is then shown that special epimorphisms are precisely simple-confluent epimorphisms, see Theorem \ref{simple=special}. The definition of a special epimorphism is intrinsic, unlike that of a simple-confluent epimorphism, which allows us to more easily work with the family of simple-confluent epimorphisms.
 We conclude this subsection by showing that $\mathcal{T}_{\mathcal{ CE}}$ is consistent, see Corollary~\ref{treeendunique}.

Recall that if $T$ is a rooted tree and $a\in T$, then  $C_a=\{x\in T\colon x> a\}$ is the set of vertices above $a$ with the edge relation induced from $T$. Let $f\colon S\to T$ be a confluent epimorphism between  rooted trees,  let $p\in T$, and $q\in S$ be such that $f(q)=p$. 
Note that if $E$ is a component of $C_q$ and $D$ is a component of $C_p$ such that $f(E)\cap D\neq\emptyset$ then, by confluence of $f$, we get $D\subseteq f(E)$. Therefore to each component $E$ of $C_q$ we can assign the collection of all components of $C_p$ contained in $f(E)$. The union of those components is equal to $f(E)\setminus\{p\}$.

\begin{definition}\label{def-special}

Let $f\colon S \to T$ be a confluent epimorphism between  rooted trees, and let $p$ be a ramification vertex in $T$. Call a vertex $q \in S$ with $f(q)=p$ a {\it special} vertex for $p$ (wrt $f)$ if the following hold:

\begin{enumerate}
    \item for any component $E$ of $C_q$ there is exactly one component $D$ of $C_p$ such that $D\subseteq f(E)$; 
    \item for any component $D$ of $C_p$ there is a component $E$ of $C_q$ such that $D\subseteq f(E)$.
\end{enumerate}

Note that $D\subseteq f(E)$ implies $f(E)=D$ or $f(E)=D\cup\{p\}$.
\smallskip

If $f$ is understood from the context, and $q$ is special for $p=f(q)$ (wrt~$f$), we simply say that {\it $q$ is special}. If $q_1\in S$ is such that there exists $q_1<q_2$ with $f(q_1)=f(q_2)$, and $q_2$ is special, we say that $q_1$ is {\it below a special vertex}. Similarly, if $q_2\in S$ is such that there exists $q_1<q_2$ with $f(q_1)=f(q_2)$, and $q_1$ is special, we say that $q_2$ is {\it above a special vertex}. 

We denote by $P_f(p)$ the set of vertices in $S$, which are special for $p$. 

\end{definition}

Note that we can rephrase (1) in the definition into:  for any component $E$ of $C_q$, $f(E)\setminus\{p\}$ is a component of $C_p$.
Note also that any two different vertices in $P_f(p)$  are incomparable. 
If $q$ is special for $p$, we let 
 $\alpha_f(q,p)$ to be the induced
 surjection from the set of components of $C_q$ to the set of components of $C_p$, that is,
 $\alpha_f(q,p)(E)=D$ if and only if 
 $D=f(E)\setminus\{p\}$.

 \begin{definition}

 A confluent end vertex preserving epimorphism $f\colon S\to T$ between rooted trees will be called a {\it special epimorphism} if
 for every ramification vertex $x\in T$ and $b\in S$ such that $f(b)>x$, there exists $x'<b$, which is special for $x$ (wrt $f$). 
 \end{definition}
 
 Note that if $f\colon S\to T$ is a special epimorphism then any vertex in~$S$, which is mapped to a ramification vertex, either is special, or it is above a special vertex, or else it is below a special vertex.

\begin{lemma}\label{image-special}
    Let  $f\colon S\to T$ be a special epimorphism between  rooted trees, let $x\in T$ be a ramification vertex, and let $y\in S$ be such that $f(y)=x$. Then:
    \begin{enumerate}
        \item $y$ is special, or
        \item there is a component $D$ of $C_x$ such that for any component $E$ of $C_y$ it holds $f(E)\setminus\{x\}=D$, or else
        \item for any component $E$ of $C_y$ it holds
        $f(E)=C_x\cup \{x\}$.
    \end{enumerate}
    
\end{lemma}
\begin{proof}
    Indeed, if $y$ is above a special vertex,
    then (2) holds. If $y$ is below a special vertex,
    then (3) holds. In this last case, we additionally notice, using just the definition of a special epimorphism, that any component of $C_y$ contains a vertex, which is special for $x$.
\end{proof}

 \begin{theorem}\label{simple=special}
     An epimorphism between  rooted trees is simple-confluent if and only if
     it is special.
 \end{theorem}

\begin{proof}
First, we argue that any simple-confluent epimorphism is special.
Indeed,  in case $f$ is simple-monotone and $p$ is a ramification vertex,  then $q$ is special for $p$ iff it is the unique ramification vertex with $f(q)=p$.
 In case $f$ is light confluent,
then $P_f(p)=f^{-1}(p)$. In both cases $f$ is a special epimorphism.
 Then proceed along  a decomposition into light confluent and simple-monotone epimorphisms,  noticing that 
 if $f=f_2\circ f_1$, where $f_1$ is simple-monotone or light confluent, $f_2  $ is simple-confluent, and $p$ is a ramification vertex, then $P_f(p)=\bigcup_{x\in P_{f_2}(p)} P_{f_1}(x)$.

For the opposite direction, suppose that the conclusion holds for any special epimorphism 
whose domain has $<n$ vertices. Let now
$h\colon H\to F$ be special and such that $|V(H)|=n$. We show that if $h$ is not an isomorphism, then there are $f\colon G\to F$ special and $g\colon H\to G$ simple-confluent, which is not an isomorphism, such that $h=f\circ g$. Once we show this last statement, we are done with the proof. Indeed, since $|V(G)|<n$, by the inductive assumption $f$ is simple-confluent, and hence $h=f\circ g$ must be simple-confluent.

{\bf Case 1.} There is a vertex $a\in H$ of order 2 and $a'\neq a$ adjacent to $a$ such that $h(a)=h(a')$.
In that case we let $G$ to be $H$ with $a$ and $a'$ identified, and let $g,f$ be the appropriately induced maps. Specifically, let
  $\sim$ be the  equivalence relation on $H$, which identifies $a$ and $a'$.
Define $G$ as the graph $H/\!\!\sim$ and let $g\colon H\to G$ be the quotient epimorphism. Then $g$ is a simple-monotone epimorphism. Define the epimorphism $f\colon G\to F$ by the condition $h=f\circ g$. Moreover for any ramification vertex $p\in F$, 
$P_f(p)=g(P_h(p))$ and $f$ is special. 

{\bf Case 2.} Case 1 does not hold.
Consider the subset  $S$ of $H$ defined by the conditions $x\in S$ if and only if
there are $y,z\in H$ such that:
\begin{enumerate}
    \item $y$ and $z$ are incomparable; 
    \item $x< y$ and $x< z$;
    \item $h(y)=h(z)$.
\end{enumerate}
Since $h$ is not an isomorphism, $S\ne \emptyset$. Indeed, this follows from the following claim.

\medskip

{\bf{Claim 1.}} Suppose  that there are $ a<b$ are such that $h(a)=h(b)$. Then $a\in S $ or $b\in S$.

\begin{proof}[Proof of Claim 1]
Since we are in Case 2,   $a$ is a ramification vertex, and $b$ is either a ramification vertex or an end vertex.  Denote $c=h(a)=h(b)$. In case $c$ is an end vertex or an ordinary vertex, we get $a\in S$. 
Indeed, in the case that $c$ is an end vertex, then all vertices above $a$ get mapped to $c$, so $a \in S$. In the case that $c$ is an ordinary vertex, let $d$ be the vertex that is adjacent and above $c$. No component of $C_a$ is mapped by $h$ onto $c$ as $h$ is  end vertex preserving.  Thus there are vertices in different components of $C_a$ that are mapped to $d$, again making $a\in S$.

Now consider the case that $c$ is a ramification vertex.
Then either $a$ is special, or it is above a special vertex, or it is below a special vertex. In the first two cases $b$ is above a special vertex.
Then, by Lemma \ref{image-special}(2),
we get $b\in S$, as  
there is a component $D$ of $C_c$ such that 
for any component $E$ of $C_{b}$ it holds $f(E)\setminus\{c\}=D$.
In case $a$ is below a special vertex, we get, by Lemma \ref{image-special}(3), $a\in S$, as any component of $C_a$ is mapped onto $C_c\cup\{c\}$. 
\end{proof}

Let $v$ be a maximal element in $S$ with respect to the order $\le$, let $y,z$ be two vertices as in the definition of $S$, and denote $u=h(v)$. Let $C_1$ and $C_2$, $C_1\neq C_2$, be components
of $C_v$ that contain $y$ and $z$ respectively.

By confluence of $h$, each of $h(C_1)$, $h(C_2)$ contains one or more components of $C_u$. 

\medskip

{\bf{Claim 2.}}
$h(C_1)=h(C_2)$. 
\begin{proof}[Proof of Claim 2]

(a) Suppose that $u$ is an end vertex. Then $h(C_1)=h(C_2)=\{u\}$.

(b) Suppose that (i) $u$ is an ordinary vertex or (ii)  $u$ is a ramification vertex and $v$ is special or $v$ is above a special vertex. The following argument suffices for both (i) and~(ii).

It follows from Lemma~\ref{image-special}(1),(2) that for each component  $E$ of $C_v$, $h(E)$ contains exactly one component of $C_u$. We have to verify that both $h(C_1)$ and $h(C_2)$ contain the same component, and that  $h(C_1)$ or $h(C_2)$ does not contain $u$. 
For that, we will show  that 
 there is no $v'>v$ with $h(v')=u$. Indeed, if there is such a $v'$, since we are not in Case 1 and not in (a) of this proof, such a $v'$ would have to be a ramification vertex. Furthermore, there is a component $D$ of $C_u$ such that for  any component $E$ of $C_{v'}$,
 it holds $f(E)\setminus\{u\}=D$,
 which implies 
$v'\in S$, resulting in a contradiction of the maximality of $v$. Therefore  $h(y)=h(z)>u$, and $h(C_1)=h(C_2)$ is a single component of $C_u$, .

(c) Suppose that $u$ is a ramification vertex and $v$ is below a special vertex. Then, by Lemma \ref{image-special}(3),
$h(C_1)=h(C_2)=C_u\cup~\{u\}$.

\end{proof}

{\bf{ Claim 3.}} Each of the restrictions $h|_{C_1}$ and $h|_{C_2}$ is injective. 
\begin{proof}[Proof of Claim 3]
Fix one of the $C_1$ or $C_2$ and call it $C$. In case there are incomparable $y',z'\in C$ such that $h(y')=h(z')$, then  there is $x\in C$ such that $x<y'$ and $x<z'$ (we can, for example, take $x$ to be the meet of $y'$ and $z'$), making $x$ in $S$, and we have a contradiction with maximality of $v$.

Now suppose there are $a,b \in C$ such that $v< a<b$ and $h(a)=h(b)$. Apply Claim 1 and get $a\in S$ or $b\in S$, which is a contradiction with the maximality of $v$.

\end{proof}

To finish the proof of Theorem \ref{simple=special} define the equivalence relation $\sim$
by $p\sim q$ if and only if $h(p)=h(q)$ and $p,q\in C_1\cup C_2$. Define $G$ as the graph $H/\!\!\sim$ and let $g\colon H\to G$ be the quotient epimorphism. Then $g$ is an elementary light confluent epimorphism. Define the epimorphism $f\colon G\to F$ by the condition $h=f\circ g$ and note that by Proposition \ref{composition-confluent}, $f$ is confluent. Since $h$ is end vertex preserving, Lemma \ref{end-h=gf} implies that $f$ is end vertex preserving. Moreover for any ramification vertex $p\in F$, 
$P_f(p)=g(P_h(p))$ and $f$ is special.  
    \end{proof}

   As a by-product of the proof of Theorem \ref{simple=special}, we obtain the following.
\begin{corollary}\label{elemenlightc}
 Light confluent epimorphisms between rooted trees are compositions of elementary light confluent epimorphisms.   
\end{corollary}

Our goal now is to prove that $\treend$ is consistent (Corollary \ref{treeendunique}). We first need the following theorem. 

 \begin{theorem}\label{simpleg}
Let $f\colon H\to G$ and $g\colon G\to F$ be confluent epimorphisms between  rooted trees.
If the composition $g\circ f\colon H\to F$ is simple-confluent, then $g$ is simple-confluent. 
 \end{theorem}

\begin{proof}
    Because of Theorem \ref{simple=special} we will show that $g$ is special assuming that the composition $h=g\circ f$ is special. We have that $g$ is end vertex preserving by Lemma \ref{end-h=gf}.

    Let $x$ be a ramification vertex in $F$ and  $u\in G$ be such that $g(u)>x$. Take $v\in H$ such that  $f(v)=u$. Let $x''<v$ be special for $x$ (wrt $h$). We claim that $x'=f(x'')$ is special for $x$ (wrt $g$). We know that $g$ is an epimorphism between rooted trees, in particular, it is order-preserving, therefore $x'<u$. We have to verify (1) and (2) in the definition of a special vertex.
    
    To prove (2), let $A$ be a component of $C_x$ and let $D$ be a component of $C_{x''}$ such that $h(D\cup\{x''\})=A\cup\{x\}$; it exists since $h$ special. Let $B$ be a component of $C_{x'}$ intersecting $f(D)$. From confluence of $f$, in fact $B\subseteq f(D)$. Then, using confluence of $g$, it holds $g(B\cup\{x'\})\supseteq A\cup\{x\}$.
    
    To prove (1), suppose towards the contradiction that there are components $A_1, A_2$, $A_1\neq A_2$, of $C_x$ such that for some component $B$ of $C_{x'}$, we have $g^{-1}(A_1), g^{-1}(A_2)\subseteq B$. We find $y\in H$ which is special for $x$ (wrt $h$) and satisfies $f(y)=x'$, and  $E$, a component of $C_{y}$, such that $f(E)\cap B\neq\emptyset $.
     To find such a $y$ first take any $b\in B$ with $g(b)\neq x$.  Take $c,d\in H$  such that $f(c)=x'$, $f(d)=b$, and $c<d$, which exist as $f$ is an epimorphism.   Now let $y<d$ be special for $x$ (wrt $h$). 
    Let $E$ be a component of $C_y$ containing $d$. By confluence of $f$, $f(E\cup\{y\})\supseteq B\cup\{x'\}$. Hence $h(E\cup\{y\})\supseteq A_1\cup A_2$, which contradicts that  $y$ is special for $x$ (wrt $h$). 
    \end{proof}

\begin{corollary}\label{treeendunique}

 $\treend$ is consistent.
\end{corollary}
\begin{proof}
    Apply Corollary \ref{uniquegeneral} with $\mathcal F_0$ equal to the family of rooted trees with confluent epimorphisms and $\mathcal F$ equal to the family of rooted trees with simple-confluent epimorphisms. The assumption (1) in its statement follows from Proposition \ref{composition-confluent} and (2) follows from Theorem \ref{simpleg}.
\end{proof}

\subsection{Simple*-confluent epimorphisms and the  family  $\treecon$}
We introduce a family of epimorphisms, which we call special*, and show that it is equal to the family of simple*-confluent epimorphisms. Then we prove results parallel to the ones in the the previous subsection. We conclude with Theorem~\ref{summerizing-confluent} which summarizes the properties of $|\mathbb{T}_{\mathcal C}|$, the topological realization of the projective Fra\"{\i}ss\'e limit of the family all rooted trees with simple*-confluent epimorphisms.

\begin{theorem}\label{confluent-projective-family2}
  The  family  $\treecon$ of  all  rooted trees  with 
simple*-confluent epimorphisms  is a projective Fra\"{\i}ss\'e family.

\end{theorem}

It is easy to see that Theorem \ref{Joint-Proj-pro} also shows that the family $\treecon$  has the JPP.
The amalgamation of $\treecon$
follows from Theorem \ref{onelight} and the two lemmas below.

   \begin{lemma}\label{monoandlight2}
 Let $f\colon B\to A$ be a simple*-monotone epimorphism and let $g\colon C\to A$ be a light confluent epimorphism. Let $D$ be obtained via the standard amalgamation procedure and let $f_0$ and $g_0$ be the corresponding projections. 
 Then $f_0$ is light confluent and $g_0$ is simple*-monotone. 
 \end{lemma}
 \begin{proof}

 By Theorem \ref{onelight},  all we are left to show is that $g_0$ is simple*-monotone.

A simple*-monotone epimorphism is a composition of splitting edge maps and adding edge maps, and by Lemma \ref{monoandlight}, we can assume without loss of generality that 
 $f$ is an  adding edge map. Then the proof is analogous to the proof of Lemma \ref{monoandlight}.
 \end{proof}

 \begin{lemma}\label{2monot2}
 Let $f\colon B\to A$ and $g\colon C\to A$ be simple*-monotone epimorphisms. Then there is a rooted tree $D$ and simple*-monotone epimorphisms $f_0\colon D\to B$ and $g_0\colon D\to C$ such that
 $f\circ f_0=g\circ g_0$. 
 \end{lemma}

   \begin{proof}
It suffices to show that if each of
 $f\colon B\to A$ and $g\colon C\to A$  is a splitting edge map or an  adding edge map, then we can find a rooted tree $D$, $f_0\colon D\to B$, and $g_0\colon D\to C$ such that each of $f_0$ and $g_0$
is a splitting edge map or an  adding edge map, and
$f\circ f_0=g\circ g_0$.

The proof is analogous to the proof of Lemma~\ref{2monot}.
 \end{proof}

\begin{definition}
Let $f\colon S \to T$ be a confluent epimorphism between  rooted trees, and let $p$ be a ramification vertex. Call a vertex $q \in S$ with $f(q)=p$ a {\it special*}  for $p$ (wrt $f$) if the following hold:
\begin{enumerate}
    \item for any component $E$ of $C_q$ either $f(E)=\{p\}$, or there is exactly one component $D$ of $C_p$ such that $D\subseteq f(E)$; 
    \item for any component $D$ of $C_p$ there is a component $E$ of $C_q$ such that $D\subseteq f(E)$.
\end{enumerate}
\end{definition}

 \begin{definition}

A confluent epimorphism $f\colon S\to T$ between rooted trees will be called a {\it special* epimorphism} if for every ramification vertex $x\in T$ and $b\in S$ such that $f(b)>x$, there exists $x'<b$, which is special* for $x$ (wrt to $f$).
 \end{definition}

Then we have the following analogs of Lemma \ref{image-special}, Theorem \ref{simple=special}, and Theorem
\ref{simpleg}.

\begin{lemma}\label{image-specialstar}
    Let the epimorphism $f\colon S\to T$ between  rooted trees be special*, let $x\in T$ be a ramification vertex, and let $y\in S$ be such that $f(y)=x$. Then:
    \begin{enumerate}
        \item $y$ is special*, or
        \item there is a component $D$ of $C_x$ such that for any component $E$ of $C_y$ it holds $f(E)\setminus\{x\}=D$ or $f(E)=\{x\}$, or else
        \item for any component $E$ of $C_y$ it holds
        $f(E)=C_x\cup \{x\}$ or $f(E)=\{x\}$.
     
    \end{enumerate}
    
\end{lemma}
Note that unless $y$ and $x$ in the statement of Lemma \ref{image-specialstar} are such that $f(C_y)=\{x\}$, then (2) corresponds to the case when $y$ is above a special vertex, and (3) corresponds to the case when $y$ is below a special vertex.

\begin{theorem}\label{simple=specialstar}
     An epimorphism between  rooted trees is simple*-confluent if and only if it is special*.
 \end{theorem}

\begin{proof}
A proof that every simple*-confluent epimorphism is special* is essentially the same as that every simple-confluent epimorphism is special. 

The opposite direction will follow from the following. 
We prove that a special* epimorphism can be written as 
a composition of a special epimorphism
and a simple*-confluent epimorphism.  Then the conclusion will follow from Theorem \ref{simple=special} since the composition of simple*-confluent and simple-confluent is simple*-confluent.

{\bf{Claim.}}
Let $h\colon H\to F$ be a special* epimorphism. 
Suppose that there is a vertex $a\in H$ and a component $E$ of $C_a$ such that 
$h(E)=\{h(a)\}$. Then there is a rooted tree $G$, a special* epimorphism $f\colon G\to F$, and a simple*-confluent epimorphism $g\colon H\to G$ with $g(E)=\{g(a)\}$ and $h=f\circ g$.
In fact, $g=g_2\circ g_1$, where $g_2$ is an edge adding map, and $g_1$ is a simple-confluent epimorphism.

\begin{proof}[Proof of Claim]
  We let
  $\sim$ to be the  equivalence relation on $H$, which identifies all vertices in $E\cup \{a\}$.
Define $G$ as the graph $H/\!\!\sim$ and let $g\colon H\to G$ be the quotient epimorphism. We will see in a moment that $g$ is a simple*-confluent epimorphism. Define the epimorphism $f\colon G\to F$ by the condition $h=f\circ g$. Then it is easy to see that $f$ is special*. 

We let $g_2\colon G'\to G$ be the map that adds an edge with one of its end vertices being $g(a)$. Denote  $a'=g(a)$ and let $b'$ be the other end vertex of the added edge. Let
$g'_1\colon E\cup\{a\}\to (\{a',b'\}, \langle a', b'\rangle)$ be any simple-confluent epimorphism, and finally, let $g_1\colon H\to G'$ be equal to $g'_1$ on $E$ and to the identity otherwise. Then $g_1$ is simple-confluent as well.
This in particular shows that $g$ is simple*-confluent. 

  \end{proof} 
Let $h\colon H\to F$ be a special* epimorphism. Apply Claim as many times as it is possible and obtain a decomposition $h=h_2\circ h_1$, where $h_1\colon H\to G$ is simple*-confluent, $h_2\colon G\to F$ is special*, and for no ordinary or ramification vertex $a$ there is a component $E$  of $C_a$ with $h_2(E)=\{h_2(a)\}$. Then $h_2$ is special. Indeed, all that is left to show is that $h_2$ is end vertex preserving.
Suppose that it is not the case and let $e$ be an end vertex of $G$ such that $x=h_2(e)$ is not an end vertex. Let $y$ be the immediate predecessor of $e$ and let $E$ be the component of $C_y$ that contains $e$, i.e. $E=\{e\}$.

If $x=h_2(y)=h_2(e)$, then  $h_2(E)=\{x\}$, which is impossible. On the other hand,  if $h_2(y)\neq h_2(e)$, then by the confluence of $h_2$, we have $h_2(E)\supseteq E'$, where $E'$ is the component of $C_{h_2(y)}$ containing $x$. But this is also impossible as $x$ is not an end vertex and $|E'|\geq 2$. 
  
\end{proof}

 The proof of the next theorem is identical to that of Theorem \ref{simpleg}.

 \begin{theorem}\label{simplegstar}
Let $f\colon H\to G$ and $g\colon G\to F$ be confluent epimorphisms between  rooted trees.
If the composition $g\circ f\colon H\to F$ is simple*-confluent, then $g$ is simple*-confluent.
 \end{theorem}

Therefore by Theorem \ref{simplegstar} and Corollary \ref{uniquegeneral}, we have:

\begin{corollary}

 $\treecon $ is consistent.
\end{corollary}

\begin{notation}
Let $\mathbb{T}_{\mathcal C}$ be the projective Fra\"{\i}ss\'e limit of the family~$\treecon$.
\end{notation}

\subsection{The continuum $|\mathbb{T}_{\mathcal C}|$
}
We exhibit several properties of the continuum $|\mathbb{T}_{\mathcal C}|$. We do know, by Proposition \ref{smooth-dendroid}, that $|\mathbb{T}_{\mathcal C}|$ is a smooth dendroid.  However, we do not know a characterization of the continuum $|\mathbb{T}_{\mathcal C}|$ and believe it has not been previously studied. 
On the other hand, in the remaining sections, we show that $|\mathbb{T}_{\mathcal{ CE}}|$ is homeomorphic to the Mohler-Nikiel universal dendroid.

If $\{F_n, f_n\}$ is a Fra\"{\i}ss\'e sequence for $\mathcal{T}_{\mathcal C}$ then $(r_{F_1},r_{F_2},\ldots)$ is a vertex in $\mathbb{T}_{\mathcal C}$ which we designate as the root of $\mathbb{T}_{\mathcal C}$.  We also have the order on $\mathbb{T}_{\mathcal C}$ induced by this root.

Recall the definitions of $C_a$ and $\overline{C}_a$ (see Definition \ref{cone-over}), for a rooted tree $T$ and $a\in T$.

\begin{theorem}\label{only-rampt-or-endpt}
A vertex $r\in \mathbb{T}_{\mathcal C}$ is
either an end vertex or is a ramification vertex or else is adjacent to a ramification vertex of $\mathbb{T}_{\mathcal C}$.
\end{theorem}

\begin{proof}

We will show that for any vertex $r\in \mathbb{T}_{\mathcal C}$ that is not an end vertex of $\mathbb{T}_{\mathcal C}$ there are two arcs joining $r$ to end vertices of $\mathbb{T}_{\mathcal C}$. If $r$ has no immediate successor, then the intersection of the two arcs will be $\{r\}$, otherwise it will be $r$ and the immediate successor of $r$. These two arcs together with the arc from the root of $\mathbb{T}_{\mathcal C}$ to $r$ witness that $r$ or an immediate successor of $r$ is a ramification vertex of $\mathbb{T}_{\mathcal C}$. Note, by Theorem~\ref{rooted-end}, $r$ is not adjacent to an end vertex of $\mathbb{T}_{\mathcal C}$. 

Fix a Fra\"{\i}ss\'e sequence $\{F_n, f_n\}$ for $\mathcal{T}_{\mathcal C}$. Let $r=(r_n)\in\iLim\{F_n,f_n\}$  be a vertex in  $\mathbb T_{\mathcal C}$ different from an end vertex.  Since $r$ is not an end vertex, 
there is an $s=(s_n)>r$, and hence there is $N$ such that for all $n>N$ we have $s_n>r_n$. That implies that there is $N$ such that for all $n>N$ we have $f_n (C_{r_{n+1}})\neq \{r_n\}$.
Without loss of generality, we can assume that for all $n$, $f_n(C_{r_{n+1}})\neq \{r_n\}$. Note that a subsequence of a Fra\"{\i}ss\'e sequence is again a Fra\"{\i}ss\'e sequence with inverse limit equal to $\mathbb{T}_{\mathcal C}$.

Let $\im^*(r_n)$ denote the set of immediate successors $a$ of $r_n$ such that  $f_{n-1}(\overline{C}_a)\neq \{r_{n-1}\}$. Without loss of generality, by passing to a subsequence of $(r_n)$, we can assume that either (i) each $r_n$ is below a special vertex, or (ii) each $r_n$ is a special vertex, or else (iii) each $r_n$ is above a special vertex. Indeed, by Lemma \ref{image-specialstar}, note that if for some $n$, $r_{n+1}$ is above a special vertex (wrt $f_n$), then for every $m>0$,   $r_{n+m}$ is above a special vertex (wrt $f_n^{n+m}$). Likewise, if for some $n$, $r_{n+1}$ is below a special vertex (wrt $f_n$), then for every $m>0$,   $r_{n+m}$ is below a special vertex (wrt $f_n^{n+m}$). Moreover, if $r_{n+2}$ is special (wrt $f_{n+1}$), $r_{n+1}$ is special (wrt $f_n$), then $r_{n+2}$ is special (wrt $f^{n+2}_n$).  

We claim that there is a sequence $(a_n)$, with $a_n\in\im^*(r_n)$, such that for every $n$,  $f_n(\overline{C}_{a_{n+1}})$ contains $\overline{C}_{a_{n}}$. 
For (i), by Lemma \ref{image-specialstar}(3), any sequence $(a_n)$ with $a_n\in\im^*(r_n)$ works. For (ii),  by Lemma \ref{image-specialstar}(1), we can inductively build such a sequence $(a_n)$.  For (iii),  by Lemma \ref{image-specialstar}(2), for any $n$ there is exactly one $a_n \in\im^*(r_n)$ such that for any $b\in \im^*(r_{n+1})$ we have that $f_n(\overline{C}_b)$ contains $\overline{C}_{a_n}$. This is the required $(a_n)$.

We construct a Fra\"{\i}ss\'e sequence $\{K_n,u_n\}$ for $\mathcal{T}_{\mathcal C}$ with $r=(r'_n)$, $r'_n\in K_n$ together with  arcs $A_n=[r'_n, e^1_n]$, $B_n=[r'_n,e^2_n]$, where $e^1_n$ and $e^2_n$ are end vertices, such that $A_n\cap B_n=\{r'_n,x_n\}$, where $x_n$ is an immediate successor of $r'_n$. We will have that $u_n|_{A_{n+1}}$ is monotone and onto $A_n$, and similarly,  $u_n|_{B_{n+1}}$ is monotone and onto $B_n$.  Then $\iLim A_n$ and $\iLim B_n$ are arcs in $\mathbb{T}_{\mathcal C}$ with intersection equal to  $\{r\}$ or $\{r,a\}$ where $a$ is an immediate successor of $r$.

Given $F_{k_n}$, which contains two arcs $A$ and $B$ joining $r_{k_n}$ with end vertices of $F_{k_n}$ and meet at a vertex $m\geq a_{k_n}$, we construct a tree $K_n$, arcs $A_n, B_n$ in $K_n$, and $g_n\colon K_n\to F_{k_n}$, $g_n \in \mathcal{T}_{\mathcal C}$.   The epimorphism $g_n$ is the composition of a splitting edge map and light confluent map obtained as follows. First, let $g_n' \colon K_n' \to F_{k_n}$ be the splitting edge map that splits the edge $\langle r_{k_n}, a_{k_n}\rangle$. Call the new vertex $x$ and let $g_n'(x)=a_{k_n}$. Then double the $C_x$. That is, let $K_n$ be the graph with $V(K_n)= V(K_n')\setminus V(C_x) \cup (V(C_x^1) \cup V(C_x^2))$ where $C_x^i$ are copies of $C_x$ and let $E(K_n)= E(K_n') \setminus E(C_x) \cup (E(C_x^1) \cup E(C_x^2) \cup \{\langle x, a_{k_n}^1\rangle, \langle x, a_{k_n}^2\rangle\})$ where $a_{k_n}^i$ are copies of $a_{k_n}$ in $C_x^i$. Let $g'' \colon K_n \to K_n'$ be the map that if $y^1$ and $y^2$ are copies of the vertex $y \in C_x$ in $C_x^1$ and $C_x^2$  respectively then $g''(y^i)=y$ and $g''$ is the identity otherwise. Then $g_n=g_n' \circ g_n''$.  The required arcs $A_n$ and $B_n$ join $x$ to end vertices of $K_n$ and are contained in $C^1_x\cup\{x, r_{k_n}\}$ and $C^2_x\cup\{x, r_{k_n}\}$ respectively.

Next, since $\{F_n, f_n\} $ is a Fra\"{\i}ss\'e sequence, take $h_n\colon F_{k_{n+1}}\to K_n$, $h_n \in \mathcal{T}_{\mathcal C}$, such that $f^{k_{n+1}}_{k_n}=g_n h_n$. Denote $r'_n=r_{k_n}\in K_n$. Note that $g_n^{-1}(r_{k_n})=\{r'_n\}$ implies that $h_n(r_{k_{n+1}})=r'_n$. Moreover, $f^{k_{n+1}}_{k_n}(\overline{C}_{a_{k_{n+1}}})$ contains $\overline{C}_{a_{k_n}}$ together with $g_n^{-1}(\overline{C}_{a_{k_n}})=\overline{C}_x $ implies that 
$h_n(\overline{C}_{a_{k_{n+1}}})$ contains $\overline{C}_x $.
Take any arcs $A', B'$  joining $r_{k_{n+1}}$ with end vertices of $F_{k_{n+1}}$,  which meet at a vertex $m'\geq a_{k_{n+1}}$, $h_n|_{A'}$ is monotone and onto $A_n$, and $h_n|_{B'}$ is monotone and onto $B_n$.
This finishes the inductive step of the construction. 
Finally we take $u_n=h_n g_{n+1}$.

\end{proof}

\begin{lemma}\label{limit-dense-endpts}
Let $\mathbb{T}_{\mathcal C}$ be the projective Fra\"{\i}ss\'e limit of the family $\treecon$. Then the set of end vertices is dense in $\mathbb{T}_{\mathcal C}$.
\end{lemma}

\begin{proof}
Let $p\in \mathbb{T}_{\mathcal C} $ and let $\varepsilon >0$ be given. We want to find an end vertex $e\in\mathbb{T}_{\mathcal C}$ such that
$d(p,e)<\varepsilon$. Let $G$ be a tree and $f_G\colon \mathbb{T}_{\mathcal C}\to G$ be an
$\varepsilon$-epimophism. Define a tree $H$ such that $V(H)=V(G)\cup \{q\}$ and $E(H)=E(G)\cup \{\langle f_G(p),q\rangle\}$, and let
$f^H_G$ be an epimorphism defined by ${f^H_G}|_{G}$ is the identity and $f^H_G(q)=f_G(p)$.  Let $f_H\colon \mathbb{T}_{\mathcal C}\to H$
be an epimorphism such that $f_G^H \circ f_H=f_G$. By Proposition \ref{end-to-end} there is $e\in \mathbb{T}_{\mathcal C}$ which is an end vertex and $f_H(e)=q$.  Then
$d(p,e)<\varepsilon$ since $f_G(p)=f_G(e)$ and $f_G$ is an $\varepsilon$-epimophism.

\end{proof}

We summarize the results of this section in the following theorem.
\begin{theorem}\label{summerizing-confluent}
The projective Fra\"{\i}ss\'e limit $\mathbb{T}_{\mathcal C}$ of the family $\treecon$ has transitive set of edges  and its topological realization $|\mathbb{T}_{\mathcal C}|$
is a Kelley dendroid (thus smooth) with a dense set of endpoints and such that each point is either an endpoint or a ramification point in the classical sense.
\end{theorem}

\begin{proof}

Proceeding as in the proof of Theorem \ref{dense} we may show that the family $\treecon$ satisfies the hypothesis of Theorem \ref{transitive}. Hence the set of edges is transitive. Knowing that $E(\mathbb{ T}_{\mathcal C})$ is transitive we see from Proposition \ref{smooth-dendroid}, Proposition \ref{Kelley-theorem} and Observation \ref{Kelley-observation} that the topological realization $|\mathbb{ T}_{\mathcal C}|$ is a Kelley dendroid.

To see that the set of endpoints is dense in $|\mathbb{ T}_{\mathcal C}|$, take a nonempty open set $U$ in $|\mathbb{ T}_{\mathcal C}|$.  The set $\varphi^{-1}(U)$, where $\varphi$ is the quotient map, is open in $\mathbb{ T}_{\mathcal C}$. So by Lemma \ref{limit-dense-endpts} there is an end vertex $e\in \varphi^{-1}(U)$. It follows from Theorem \ref{rooted-end} that $\mathbb{ T}_{\mathcal C}$ has no isolated end vertices.  Then by Proposition \ref{endpts-to-endpts}, $\varphi(e)$ is an endpoint in $U$.   Hence the set of endpoints is dense in $|\mathbb{ T}_{\mathcal C}|$.

To complete the argument we need to show that every point in $|\mathbb{ T}_{\mathcal C}|$ is either an endpoint or a ramification point. Note that if $p$ is a point in $|\mathbb{ T}_{\mathcal C}|$ then $\varphi^{-1}(p)$ contains at most two vertices.  If $\varphi^{-1}(p)$ is a single vertex $q$ then, by Theorem \ref{only-rampt-or-endpt}, $q$ is an end vertex or a ramification vertex of $\mathbb{ T}_{\mathcal C}$. If $q$ is an end vertex then Proposition \ref{endpts-to-endpts}  shows that $p$ is an endpoint of $|\mathbb{ T}_{\mathcal C}|$.  If $q$ is a ramification vertex then by  Proposition \ref{ram-to-ram} and Theorem \ref{rooted-end},  $p$ is a ramification point of $|\mathbb{ T}_{\mathcal C}|$.  Similarly,  if $\varphi^{-1}(p)=\{q,r\}$ then it follows from Proposition \ref{endpts-to-endpts} that one of these vertices must be a ramification vertex and hence, by Proposition \ref{ram-to-ram} and Theorem \ref{rooted-end},  $p$ is a ramification point of $|\mathbb{ T}_{\mathcal C}|$. 
\end{proof}

\begin{remark}
Note $|\mathbb{ T}_{\mathcal C}|$ is not a dendrite since it is not embeddable in the Wa\. zewski universal dendrite $D_{\omega}$.

\end{remark}

 \section{Mohler-Nikiel dendroids}\label{M-NSec}

In this section we give some background on the Mohler-Nikiel dendroid. We first recall its construction.

Let $d_1,d_2,d_3, \ldots$ be a (finite or infinite) sequence of numbers in $[0,1)$.
We construct an inverse sequence $\{Y_n,g_n\}$, where each $Y_n$ is a rooted tree and  $g_n$ is a confluent map.
Let $Y_0=[0,1]$. 
Let $Y_1$ be the disjoint union of $Y_0\times\{0\}$ and  $Y_0\times\{1\}$ where $(y,0)$ and $(y,1)$ are identified if and only if $y \le d_1$. Define $g_0\colon Y_1 \to Y_0$ by $g_0((y,j))=y$. 
Suppose that trees $Y_1,\ldots, Y_n$ and the confluent bonding maps $g_{k-1}\colon Y_{k}\to Y_{k-1}$, $k\leq n$, are already given. 
We now construct $Y_{n+1}$ and $g_{n}$. For that first take the disjoint union  $Y_n\times\{0\}$ and  $Y_n\times\{1\}$ of $Y_n$. The tree $Y_{n+1}$ is obtained by identifying  $(y,0)$ and $(y,1)$ for each $y$ such that $g^n_0(y)\leq d_{n+1}$. Let $g_n\colon Y_{n+1}\to Y_n$ map $(y,0)$ and $(y,1)$ to $y$, for each $y\in Y_n$. This finishes the construction of 
$\{Y_n,g_n\}$. 

In case  $D=(d_1,d_2,d_3, \ldots, d_n)$ is finite, we say that $Y_n$ constructed as above is the tree obtained by 
{\it following the 
MN-construction} for $D$ and denote it by $T_D$.

Suppose that  $d_1,d_2,d_3, \ldots$ is a dense sequence of numbers in $[0,1)$ such that 0 is in the sequence and each number that appears in this sequence appears infinitely many times, that is, for each $n$ there exists $m >n$ with $d_n = d_m$. Then 
$\iLim\{Y_n,g_n\}$ is a universal smooth dendroid constructed by Mohler and Nikiel in \cite{Dendroid}. We will denote it by $MN_{(d_n)}$ and call it a {\it Mohler-Nikiel dendroid}.

It seems that it is general ``folk'' knowledge that all Mohler-Nikiel dendroids are homeomorphic but we have been unable to find a proof in the literature. For completeness, we present a proof of this fact below.

\begin{theorem}\label{uniqueMN}
Let $d_1,d_2,d_3, \ldots$ and $e_1,e_2,e_3,\ldots $ be dense sequences of numbers in $[0,1)$ such that 0 is in the sequence and each number that appears in a sequence appears in it infinitely many times. Then
$MN_{(d_n)}$ and $MN_{(e_n)}$ are homeomorphic.
\end{theorem}

We will need the following definition and remark, whose proof is straightforward.
\begin{definition}
Finite sequences of  numbers in $[0,1)$, $d_1,d_2,\ldots,d_n$ and $e_1,e_2,\ldots,e_n$ are said to be {\it order equivalent} if there exists a homeomorphism $h\colon [0,1] \to [0,1]$ such that $h(0)=0$ and $h(d_i)=e_i$ for all $i$, $1 \le i \le n$. Infinite sequences $(d_n)$ and $(e_n)$ are {\it order equivalent} if for any $N$, sequences 
$(d_n)_{n\leq N}$ and $(e_n)_{n\leq N}$ are order equivalent.
\end{definition}

\begin{remark}\label{rearr}
\begin{itemize}
\item[(i)] Let  $D=(d_1,d_2,\ldots,d_n)$ and $E=(e_1,e_2,\ldots,e_n)$ be sequences in $[0,1)$, which are order equivalent.
 Then the trees $T_D$ and $T_E$  are homeomorphic.
 Moreover, there is a canonical homeomorphism 
 $H\colon T_D\to T_E$ induced from  a homeomorphism $h\colon [0,1]\to [0,1]$ satisfying
 $h(0)=0$ and $h(d_i)=e_i$,  $i\leq n$.

\item[(ii)] Let $D=(d_1,d_2,\ldots, d_n)$ be a sequence  in $[0,1)$. Let $\sigma$ be a permutation of $1,2\ldots, n$, and denote $\sigma(D)=(\sigma(d_1),\sigma(d_2),\ldots, \sigma(d_n))$.  Then $T_D=T_{\sigma(D)}$.

 \end{itemize}
\end{remark}

Let $m$ and  $(d_i)_{i\leq m}$ be given.
Suppose that $S$ is obtained by following the MN-construction for $(d_i)_{i\leq m}$, that is,
$S=T_{  (d_i)_{i\leq m} } $.
Given a tree $T$ and a map $f\colon T\to S$ we say that $f$ is a {\it MN-map} 
if there are $d_{m+1}, \ldots, d_n$ such that 
 $T=T_{  (d_i)_{i\leq n} }$ is obtained by following the MN-construction for $(d_i)_{i\leq n}$, and $f=g^n_m$.
In that case we will also say that $f$ is the {\it MN-map from $(d_i)_{i\leq m}$ to $(d_i)_{i\leq n}$}.

\begin{proof}[Proof of Theorem \ref{uniqueMN}]

Let $\{X_n,f_n\}$ be the inverse sequence obtained via the Mohler-Nikiel construction for $(d_n)$ and let $\{Y_n, g_n\}$  be the inverse sequence obtained via the Mohler-Nikiel construction for $(e_n)$. Since there is a permutation $\sigma$ of $\mathbb{N}$ such that $(d_n)$ and $(\sigma(e_n))_n$ are order equivalent, the proof will follow from the following two special cases.

{\bf{Case 1.}} Sequences $(d_n)$ and $(e_n)$ are order equivalent.

For every $N$ take a homeomorphism $h_N\colon [0,1]\to [0,1]$ satisfying
 $h_N(0)=0$ and $h_N(d_i)=e_i$, $i\leq N$.
Let $H_N\colon X_N \to Y_N$ be the homeomorphism induced from $h_N$.
We have for every $m<n$ that  $H_m\circ f^n_m=g^n_m\circ H_n$. Homeomorphisms $H_n$ induce a homeomorphism $H\colon MN_{(d_n)}\to MN_{(e_n)}$.

{\bf{Case 2.}} Let $(e_n)$ be a permutation of $(d_n)$. Let $R$ denote the set of  numbers that occur in $(d_n)$ (equivalently, in $(e_n)$).

We show that  we can construct strictly increasing subsequences $(k_i)$ and $(l_i)$ of natural numbers as well as maps $\alpha_i\colon X_{k_i}\to Y_{l_i}$ and $\beta_i\colon Y_{l_{i+1}}\to   X_{k_i}$ satisfying 
  $\beta_i\circ \alpha_{i+1}= f^{k_{i+1}}_{k_i}$ and $\alpha_i\circ \beta_i= g^{l_{i+1}}_{l_i}$.
  Moreover $\alpha_i$ will be the MN-map from
  $e_1,\ldots, e_{l_{i}}$
to  $e_1,\ldots, e_{l_{i}}, e'_1,\ldots, e'_{l(i)}$ for some $l(i)\in\mathbb{N}$ and  $e'_1,\ldots, e'_{l(i)} $ from $R$.
  Analogously, $\beta_i$ will be the MN-map 
  from $d_1,\ldots, d_{k_i}$ to 
$d_1,\ldots, d_{k_i}, d'_1,\ldots, d'_{k(i)}$ for some $k(i)\in\mathbb{N}$ and $d'_1,\ldots, d'_{k(i)} $ from $R$.

Start with $\alpha_0={\rm{Id}}\colon X_0\to Y_0$. As $X_1$ and $Y_1$ are triods and $X_0=Y_0=[0,1]$, we let $\beta_0\colon Y_1 \to X_0(=Y_0)$ be $g_0$ and $\alpha_1\colon X_1\to Y_1$ be an order-preserving homeomorphism determined by $g_0\circ \alpha_1=f_0$. Then $\beta_0\circ \alpha_{1}= f^{k_{1}}_{k_0}$ and $\alpha_0\circ \beta_0= g^{l_{1}}_{l_0}$,
where $l_0=k_0=0, l_1=k_1=1$.

Suppose that we have already constructed $\alpha_n$,  $ \beta_n$, $n< N$. We explain how to construct $\alpha_{N}$. We have that $\beta_{N-1}$ is the MN-map from $d_1,\ldots, d_{k_{N-1}}$
to  $d_1,\ldots, d_{k_{N-1}}, d'_1,\ldots, d'_k$ for some $k\in \mathbb{N}$ and  $d'_1,\ldots, d'_k $ from $R$. Since each element of $(d_i)$ appears infinitely many times, there are indices $t_1,\ldots, t_k$ with $k_{N-1}<t_1<t_2<\ldots<t_k$ such that 
$d_{t_j}= d'_j$, $j=1,\ldots,k$. Let $k_{N}=t_k$.
List the $n=k_{N}-k_{N-1}-k$ elements from 
$d_{k_{N-1}+1}, d_{k_{N-1}+2}\ldots, d_{k_{N}} $, except $d'_1,\ldots, d'_k $, as   $d''_1,d''_2,\ldots, d''_{n-1},d''_n$. Finally,
let $\alpha_N$ be the MN-map from 
 $d_1,\ldots, d_{k_{N-1}}, d'_1,\ldots, d'_k$ 
to
 $$d_1,\ldots, d_{k_{N-1}}, d'_1,\ldots, d'_k, d''_1,d''_2,\ldots, d''_{n-1},d''_n.$$ 
 
 Then, by Remark \ref{rearr} (ii), $\beta_{N-1}\circ \alpha_N= f^{k_{N}}_{k_{N-1}}$. 
 \end{proof}

\section{The $|\mathbb{T}_{\mathcal{ CE}}|$ is the Mohler-Nikiel universal dendroid}\label{M-N Dend}

The main goal of this section is the following theorem.
\begin{theorem}\label{toprealMN}
The topological realization  of the projective Fra\"{\i}ss\'e limit of $\treend$ is the Mohler-Nikiel universal dendroid.
\end{theorem}

   Simple-confluent epimorphisms are special, as we showed in Section~\ref{sec:simplecon}.  Recall that if $f\colon S\to T$ is simple-confluent, where $T$ and $S$ are rooted trees and $p$ is a ramification vertex in $T$, then $P_f(p)$ denotes the set of vertices in $S$ which are special for $p$ (wrt $f$). If $q$ is special for $p$, we let $\alpha_f(q,p)$ be the induced surjection from the set of components of $C_q$ to the set of components of $C_p$, that is, $\alpha_f(q,p)(E)=D$ if and only if $D=f(E)\setminus\{p\}$.
 Equivalently, $\alpha_f(q,p)$ maps the immediate successor of $q$ contained in $E$ to the immediate successor of $p$ contained in $D$, which is how we use the notation $\alpha_f(q,p)$ below. We observe that the following holds. 
   
 \begin{lemma}\label{indepe}
Let $f\colon S\to T$ be simple-confluent, where $T$ and $S$ are rooted trees. Suppose that all vertices of $S$ are end vertices or ramification vertices. Let $p_1,\ldots, p_m\in T$ be ramification vertices, 
which form a maximal set of pairwise incomparable vertices in $T$. Then $P_f(p_1)\cup\ldots\cup P_f(p_m)$ is a maximal set of pairwise incomparable vertices in $S$.

\end{lemma}

 \begin{proof}
Suppose there are vertices $a$ and $b$ in $P_f(p_1)\cup\ldots\cup P_f(p_m)$ which are comparable, say $a < b$. Since any two vertices in a $P_f(p_i)$ are incomparable there are $p_j \not = p_k$ such that  $a \in P_f(p_j)$ and $b \in P_f(p_k)$. If $E$ is a component of $C_a$ then $E$ contains a component of $C_b$, therefore $f(E)$ contains a component of $C_{p_k}$. On the other hand, $f(E)\setminus\{p_j\}$ is equal to a component of $C_{p_j}$. Since $p_j$ and $p_k$ are incomparable, this is impossible.

For maximality we show that every branch in $S$ contains a vertex from $P_f(p_1)\cup\ldots\cup P_f(p_m)$. Let $B$ be a branch in $S$ with end vertex $e$.  By maximality of $p_1,\ldots, p_m$ there is $i$ such that the end vertex $f(e)$ is comparable to $p_i$; hence $p_i< f(e)$. Then there is $q \in P_f(p_i)$ such that $q \le e$, i.e. $q$ is a vertex in $B$. 

 \end{proof}

\bigskip

{\bf{The double splitting edge operation.}} 
Let $T$ be a rooted tree. We let $d_T$ be a map that adds two vertices to each edge of $T$. Specifically, for each $e=\langle a,b\rangle\in E(T)$, we take two new vertices $e^a, e^b$, remove the edge $e$, and add the three edges $\langle a, e^a \rangle$, $\langle e^a, e^b \rangle$, $\langle e^b, b\rangle$. We let $d_T(e^a)=d_T(a)=a$ and $d_T(e^b)=d_T(b)=b$. 
In a situation when $T$ is understood from the context, we will sometimes abuse the notation and simply write $d$ instead of $d_T$. We will call $d$ the {\it double splitting edge operation}.

We will later need the following lemma.
\begin{lemma}\label{doubleamal}
     Let $f\colon B\to A$ be light confluent and let $g\colon C\to A$ be the double splitting edge operation. Then the standard amalgamation produces 
     a rooted tree $D$, the double splitting edge operation $f_0\colon D\to B$ and a light confluent $g_0\colon D\to C$ such that
 $f\circ f_0=g\circ g_0$. 
\end{lemma}
\begin{proof}

Given edges $\langle a,b\rangle$ in $A$ and $\langle a_B,b_B\rangle$ in $B$ such that 
   $f(a_B)=a$ and $f(b_B)=b$ there are edges $\langle a_C,a_C'\rangle$, $\langle a_C',b_C'\rangle$, and $\langle b_C',b_C\rangle$ in $C$ such that $g(a_C)=g(a_C')=a$ and $g(b_C)=g(b_C')=b$. Then $D$ contains the vertices $(a_B,a_C)$, $(a_B,a_C')$, $(b_B,b_C')$, and $(b_B,b_C)$ with edges $\langle (a_B,a_C),(a_B,a_C')\rangle$, $\langle (a_B,a_C'),(b_B,b_C')\rangle$, and $\langle (b_B,b_C'),(b_B,b_C)\rangle$. The epimorphism $f_0$ maps  $(a_B,a_C)$ and $(a_B,a_C')$ to $a_B$ and it maps $(b_B,b_C')$ and $(b_B,b_C)$ to $b_B$ so it is a double  splitting edge operation. The epimorphism $g_0$ is light confluent by Theorem  \ref{onelight}.
\end{proof}
\bigskip

Recall that for a rooted tree $T$ and $a\in T$, we write $C_a=\{x\in T\colon x> a\}$ and $\overline{C}_a=\{x\in T\colon x\geq  a\}$. We will often identify the set of components of $C_a$ with the set of immediate successors $\im(a)$ of $a$.

{\bf{The multiplying branches operation }}(with parameter $n$).
Let $T$ be a rooted tree and let $n$ be divisible by the least common multiple of \{$\sord(a)\colon a\in T\}$. 
We will recursively define a rooted tree $S$ such that $\sord(a)=n$ for every vertex $a\in S$ different from an end vertex, and a light confluent epimorphism $f\colon S\to T$.

First, let $a\in T$ be such that all immediate successors of $a$ are end vertices. Replace $\overline{C}_a$ by $\frac{n}{\sord(a)}$ many copies of $\overline{C}_a$. Specifically, take  $T\setminus C_a$ along with 
$\frac{n}{\sord(a)}$ many copies of $\overline{C}_a$. Identify all those trees at $a$ to obtain a rooted tree $S_a$. Let $f_a$ map each of the copies of $\overline{C}_a$ in $S_a$ to $\overline{C}_a$ in $T$ and otherwise be the identity. Then $f_a$ is a light confluent epimorphism. 

In the recursive step suppose that vertex $a$ is such that all immediate successors of $a$ already have the successor order $n$ or are end vertices. Replace $\overline{C}_a$ by $\frac{n}{\sord(a)}$ many copies, $\overline{C}_a^i$, $i=1,\ldots, \frac{n}{\sord(a)}$, of $\overline{C}_a$ to obtain a `new' tree and let the epimorphism map the copies of $\overline{C}_a^i$ to $\overline{C}_a$ and be the identity otherwise. 

After finitely many steps we will get a rooted tree $S$  such that the successor order of each vertex that is not an end vertex is $n$. We obtained $S$ from $T$ by a sequence of light confluent epimorphisms.

\bigskip

{\bf{The colored adding branches operation}} (with parameter $n$).
Let $T$ be a rooted tree. We now additionally have for each vertex a coloring of its immediate successors. If $a\in T$, let $c_a$ be a function from the set of immediate successors of $a$, $\im(a)$, onto an initial segment of natural numbers $\mathbb{N}$.
Let $$\{n_1,n_2,\ldots, n_l\}=
\{j\in\mathbb{N}\colon \exists a\in T \text{ such that } c_a(\im(a))=\{0,1,\ldots,j-1\}\}.$$ 
Let $$ N=\max\{|c_a^{-1}(i)| \colon a\in T, i\in\mathbb{N}\}$$
Let $n\geq N$ be divisible by the least common multiple of $n_1, n_2,\ldots, n_l$. 

We now proceed similarly (but not exactly) as in the multiplying branches operation. The resulting tree $S$ and the light confluent $f\colon S\to T$  will be such that for any $b\in S$, if $a=f(b)$ and $c_a$ is the coloring into $k$ many colors, then for all $i\in\{0,\ldots,k-1\}$,  there are exactly $\frac{n}{k}$ successors of $b$ mapped to $i$ by $c_a\circ f$.

For the recursive step suppose that we already have light confluent $f'\colon T'\to T$ and 
$b\in T' $  such that all non-end vertices $s\in C_b$ 
are such that if $t=f'(s)$ and $c_t$ is a coloring into $k(t)$ many colors, then for all $i\in\{0,\ldots,k(t)-1\}$,  there are exactly $\frac{n}{k(t)}$ successors of $s$ mapped to $i$ by $c_t\circ f'$.
Let $a=f'(b)$ and $k$ be such that $c_a\circ f'$ is a coloring into $k$ many colors, i.e. that $c_a\circ f'(\im(b))=
\{0,1,\ldots,k-1\}$. For every $i=0,1,\ldots, k-1$ fix $b_i\in \im(b)$  such that $c_a\circ f'(b_i)=i$
and let $m_i=|(c_a\circ f')^{-1}(i)|$.
For each such $i$ replace $\overline{C}_{b_i}$ by  
$\frac{n}{k}-m_i+1$ many copies of $\overline{C}_{b_i}$.

In the obtained rooted tree the successor order of each non-end vertex is equal to $n$.

We emphasize that in case for each vertex we color its immediate successors into exactly one color, then the colored adding branches operation is not the same as the multiplying branches operation.

\bigskip
 
{\bf{The Fra\"{\i}ss\'e sequence.}}
Let $A_1$ be the unique rooted tree on three vertices that has height 1. Let $f_i= d_i\circ u_i\colon A_{i+1}\to A_i$, where $d_i$ is the double splitting edge operation and  $u_i$ is the multiplying branches operation for $n=2^{i+1}$. 
This produces an inverse sequence $\{A_n, f_n\}$. Our goal now is to show that this is a Fra\"{\i}ss\'e sequence.

\begin{theorem}\label{fraisse}
The sequence $\{A_n, f_n\}$ is a Fra\"{\i}ss\'e sequence.
\end{theorem}

Let us first understand what maps $f_m^n$ look like.

\begin{proposition}\label{internchar}
If $1\leq m$, then $A_m$ is a regular rooted tree such that $\htt(A_m)=3^{m-1}$ and $\sord(A_m)=2^m$. 

 If $m<n$ and $k=\htt(A_m)$ then there are $l$ and $0= t_0\leq s_0<t_1\leq s_1<\ldots<t_k=s_k=l$ such that the following hold for $h=f^n_m$. 
  \begin{enumerate}
\item For any $a\in A_m$ with $\htt(a)=i$ and $b\in A_n$ with $h(b)=a$ we have that: 
$\htt(b)\in [0,s_0]$ iff $i=0$,  and $\htt(b)\in (s_{i-1}, s_{i}]$ iff $0<i\leq k$.
\item For $b\in A_n$, $a=h(b)$, and $0\leq i< k$, we have $\htt(b)=t_i$ iff $b\in P_{h}(a)$ and $\htt(a)=i$.
\item For any $a\in A_m$, $b\in P_{h}(a)$, and an immediate successor $x$ of $a$, there are exactly
$\frac{\sord(A_n)}{\sord(A_m)}$ many immediate successors of $b$ such that $\alpha_{h}(b,a)$ maps them to $x$.  

  \end{enumerate}
  
\end{proposition}

\begin{proof}
    It is not hard to calculate that $l=\htt(A_n)=3^{n-1}$, and 
    $s_0, t_1, s_1, \ldots, s_{k-1}$ such that 
    $$ t_i-s_{i-1}=  \frac{3^{n-m}+1}{2}, \ \  i=1, \dots, k,$$
    and
$$ s_i-t_i =  \frac{3^{n-m}-1}{2}, \ \  i=0,\ldots, k-1,$$
        are as needed.
\end{proof}

 \begin{proposition}\label{unique}
 Let $A_m$ and $A_n$ be  regular rooted trees with $\htt(A_m)=3^{m-1}$, $\htt(A_n)=3^{n-1}$, $\sord(A_m)=2^m$, and $\sord(A_n)=2^n$.
Suppose that $h\colon A_n\to A_m$ is a simple-confluent epimorphism, and let $0= t_0\leq s_0<t_1\leq s_1<\ldots<t_k= s_k=3^{n-1}$, where $k=3^{m-1}$, be determined by
$t_i-s_{i-1}= \frac{3^{n-m}+1}{2}$ for $  i=1, \dots, k,$
    and
$ s_i-t_i =  \frac{3^{n-m}-1}{2}$ for $i=0,\ldots, k-1$.
Suppose that $h$ satisfies 
(1)-(3) of Proposition \ref{internchar}.
 Then $h=f^n_m\circ \alpha$, where $\alpha$ is an isomorphism of $A_n$ (that is, a bijection that preserves the order and the edge relation).
  \end{proposition}
  \begin{proof}
  Clear.
  \end{proof}

\begin{proof}[Proof of Theorem \ref{fraisse}] 

Fix $m$ and a simple-confluent epimorphism $f\colon S\to A_m$, where $S$ is a  rooted tree, and let $k=\htt(A_m)$.
Without loss of generality, all vertices in $S$ are end vertices or ramification vertices. We will use Proposition \ref{Fraisse equiv} to conclude that $\{A_n,f_n\}$ is a Fra\"{\i}ss\'e sequence.

We first find $n$ large enough and then we 
 find $g\colon A_n\to S$, simple-confluent, such that $h= f\circ g$ satisfies 
(1)-(3) of Proposition \ref{internchar}. By Proposition \ref{unique} this will finish the proof of Theorem \ref{fraisse}.
In fact, we will find such a $g$ as a composition of a simple-monotone epimorphism $g_1$ and a light confluent epimorphism $g_2$. The constructions of $g_1$ and $g_2$ are shown in Claims 1 and 2, respectively.

We now explain the conditions that we require for the $n$. We point out that any $n'>n$ will work as well. 
Fix a branch $B$ of $S$ and note that $f(B)$ is a branch in $A_m$. Let $c_0=r_{A_m}<c_1<\ldots< c_k$ be all vertices of the branch $f(B)$. 
Let $s_i^B$ be the largest vertex in $B$ such that $f(s_i^B)=c_{i}$. For $i<k$, let $t_i^B$ be the unique vertex in $P_f(c_i)\cap B$, which exists by Lemma \ref{indepe}, and for $i=k$ let $t_k^B=s_k^B$.
Let 
$$\beta_B=\max(\{t^B_i-s^B_{i-1}\colon i=1,\ldots, k\}\cup \{s^B_i-t^B_i\colon i=0,\ldots, k-1 \}),$$
and let 
$$\beta= \max\{\beta_B\colon B \textrm{ is a branch in } S\}. $$
To each vertex $b\in S$, which is not an end vertex we assign a coloring $c_b$ of $\im(b)$. In case $b\notin \bigcup_{a\in A_m} P_f(a)$, let the coloring be trivial, i.e. let $c_b\colon \im(b)\to\{0\}$ be the coloring into just one color.
We then let $\gamma_b=\sord(b)$.
In case $b\in P_f(a)$, let $\rho_a$ enumerate $\im(a)$ into $0,1,\ldots, n_a-1$, where $n_a=|\im(a)|$, and take $c_b=\rho_a\circ\alpha_f(b,a)$.
In this case we let 
$$\gamma_b=n_a \times\max\{|c_b^{-1}(i)|\colon i\in\mathbb{N}\}$$ and 
take $\gamma=\max\{\gamma_b\colon b\in S\}$.

Finally, we take any $n$ such that 
$$ \beta\leq \frac{3^{n-m}-1}{2} 
\ \  \textrm{ and } \ \ 
\gamma \leq 2^n.
$$

For $n$ satisfying the above conditions, let $l=\htt(A_n)$ and $N=\sord(A_n)$.
Let $0= t_0\leq s_0<t_1\leq s_1<\ldots<t_k= s_k=3^{n-1}$ be determined by
$t_i-s_{i-1}= \frac{3^{n-m}+1}{2}$ for $  i=1, \dots, k,$
    and
$ s_i-t_i =  \frac{3^{n-m}-1}{2}$ for $i=0,\ldots, k-1$.

\bigskip
 
The following claim gives us the simple-monotone epimorphism $g_1$.

{\bf Claim 1.}
There is a regular rooted tree $R$ with  $\htt(R)=\htt(A_n)$ and simple-monotone epimorphism $g_1\colon R\to S$  such that for $h_1=f\circ g_1$ we have the following.
\begin{enumerate}
\item For any $a\in A_m$ with $\htt(a)=i$ and $b\in R$ with $h_1(b)=a$ we have that: 
$\htt(b)\in [0,s_0]$ iff $i=0$,  and $\htt(b)\in (s_{i-1}, s_{i}]$ iff $0<i\leq k$.
\item For $b\in R$, $a=h_1(b)$, and $0\leq i< k$, we have $\htt(b)=t_i$ iff $b\in P_{h_1}(a)$ and $\htt(a)=i$.
\end{enumerate}

\begin{proof}[Proof of Claim 1]
Fix a branch $B$ of $S$ and $i$. Recall the definitions of $s_i^B$ and $t_i^B$. 
If $i=0,1,\ldots, k-1$, let $\hat{s}_i^B$ be an immediate successors of $s_i^B$. Split the edge $\langle s_i^B, \hat{s}_i^B\rangle\in E(S)$ into
$$(s_i-t_i)-(s_i^B-t^B_i)+1$$
many edges. Repeat that for every immediate successors of $s_i^B$.
If $i=1,2\ldots, k$, let $\hat{t}^B_i$ be the immediate predecessor of $t_i^B$. Split the edge $\langle \hat{t}_i^B, t_i^B\rangle\in E(S)$ into $$(t_i-s_{i-1})-(t^B_i-s^B_{i-1})+1$$ many edges.

Note that those definitions are independent of the choice of $B$. Indeed, in case an edge $e=\langle s_i^B, \hat{s}_i^B\rangle=\langle s_i^{B'}, \hat{s}_i^{B'}\rangle$ is contained in branches $B$ and $B'$, then 
$s_i^B=s_i^{B'}$ and $t^B_i=t^{B'}_i$. Similarly,
in case $e=\langle \hat{t}_i^B, t_i^B\rangle=\langle \hat{t}_i^{B'}, t_i^{B'}\rangle$ is contained in branches $B$ and $B'$, then $t_i^B=t_i^{B'}$ and
$s_{i-1}^B=s_{i-1}^{B'}$.

This defines $R$. Note that for any branch $D$ in $R$ we have $s_i-t_i=s_i^D-t^D_i$
and $t_i-s_{i-1}=t^D_i-s^D_{i-1}$.
Let $g_1 $ take all new vertices in $\langle s_i^B, \hat{s}_i^B\rangle$ to $s_i^B$, all new vertices in $\langle \hat{t}_i^B, t_i^B\rangle$ to $t_i^B$, and be the identity otherwise.
Therefore (1) is satisfied. Furthermore, since $g_1$ is monotone, we have  $b\in P_{f\circ g_1}(a)$  iff $g_1(b)\in P_{f}(a)$ and $\ord(b)=\ord(g_1(b))$, 
whenever $a\in A_m$ is different from an end vertex and  
$b\in R$. Therefore we can also conclude (2).

\end{proof}

Finally, the following claim gives us the light-confluent epimorphism $g_2$.

{\bf Claim 2:}
There is a light confluent epimorphism $g_2\colon A_n\to R$ such that for $h=f\circ g_1\circ g_2$, (1)-(3)  of Proposition \ref{internchar} holds.

\begin{proof}[Proof of Claim 2]
Recall that for any $b\in S$, we have defined a coloring $c_b$ of immediate successors of $b$.
Recall that for the simple-monotone epimorphism $g_1$ we have the order-preserving injection $i_{g_1}$ (see the proof of Lemma \ref{2monot} for the definition) that takes root to the root and end vertices to end vertices.
For $d\in R$, in case $d=i_f(b)$ for some $b\in S$,
let $c_d= c_b\circ\alpha_{g_1}(d,b)$. Otherwise let $c_d$ be the trivial coloring into a single color.
Apply the colored adding branches operation (with parameter $2^n$); this produces $g_2$. Let $h=f\circ g_1\circ g_2$.
Then $g_2$ is light confluent, (1), (2) are still satisfied for $h$,
and we made sure that in addition (3) is satisfied.
\end{proof}

It follows from  Proposition~\ref{unique} that $f^n_m = f\circ (g\circ \alpha)$ where $\alpha$ is an isomorphism of $A_n$. Thus Proposition~\ref{Fraisse equiv} gives us that $\{A_i,f_i\}$ Fra\" iss\' e sequence.
\end{proof}

Since there is no known characterization of the Mohler-Nikiel universal dendroid we have to explicitly find a homeomorphism between it and $|\mathbb{T}_{\mathcal{ CE}}|$.

\begin{proof}[Proof of Theorem \ref{toprealMN}]

Since $\{A_n,f_n\}$ is a Fra\" iss\' e sequence for the family of  trees with simple-confluent epimorphisms we show that the topological realization of $\iLim\{A_n,f_n\}$ is homeomorphic to the Mohler-Nikiel universal dendroid.

Denote $A_{nn}=A_n$.
We now construct $A_{nk}$ for all $k,n$ with $k>n$ and
we construct $f^{mn}_{kl}\colon A_{mn}\to A_{kl}$, whenever $m\geq k$ and $n\geq l$.

Suppose  for $k \ge n$, $A_{nk}$ is known. We then take 
$f^{n(k+1)}_{nk}= d_{A_{nk}}$, where $d_{A_{nk}}$ is the double splitting
edge operation. This in particular defines $A_{n(k+1)}$. Therefore we have just defined  $A_{nk}$ and $f^{n(k+1)}_{nk}$ for all $k,n$ with $k\geq n$.

Recall that  $f_n= d_n\circ u_n\colon A_{n+1}\to A_n$, where $d_n$ is the double splitting edge operation and  $u_n$ is the multiplying branches operation with the parameter $2^n$. 
We let $f^{(n+1)(n+1)}_{n(n+1)}\colon A_{(n+1)(n+1)}\to A_{n(n+1)}$ to be $u_n$, $n=1,2,\ldots$.
 
Suppose for $k \ge n+1$ we have constructed a light confluent epimorphism $f^{(n+1)k}_{nk}\colon A_{(n+1)k}\to A_{nk}$, then apply the standard 
amalgamation procedure to $f^{(n+1)k}_{nk}$  and $d=f^{n(k+1)}_{nk}$, and let $f^{(n+1)(k+1)}_{n(k+1)}$ and $f^{(n+1)(k+1)}_{(n+1)k}$ be the resulting epimorphisms. In particular,
\begin{equation}\label{diag-amal}
    f^{(n+1)k}_{nk}\circ   f^{(n+1)(k+1)}_{(n+1)k}=
f_{nk}^{n(k+1)}\circ f^{(n+1)(k+1)}_{n(k+1)}.
\end{equation} 
By Lemma \ref{doubleamal}, we have $f^{(n+1)(k+1)}_{(n+1)k}=d$ and $f^{(n+1)(k+1)}_{n(k+1)}$ is light confluent, therefore this is consistent with what we have already defined. 
 Thus, we have all $f^{mn}_{kl}\colon A_{mn}\to A_{kl}$, whenever $0\leq m-k\leq 1$ and $0\leq n-l\leq~1$.
 By taking appropriate compositions of such maps (and there are many choices for that by commutativity of diagrams (\ref{diag-amal})), we have $f^{mn}_{kl}\colon A_{mn}\to A_{kl}$ for all $m\geq k$ and $n\geq l$.

 \begin{figure}[!ht]
\begin{tikzcd}
\tikzstyle{arrow} = [thick,<-,>=stealth]
 \node at (0,0) (start) {A_{nk}};
   \node at (0,-3) (stop) {A_{(n+1)k}};
       \node at (4,0) (A) {A_{n(k+1)}};
     \node at (4,-3) (B) {A_{(n+1)(k+1)}};
        \draw [arrow] (start) -- node[anchor=east] {f^{(n+1)k}_{nk}} (stop);
     \draw [arrow] (start) -- node[anchor=south] {f^{n(k+1)}_{nk}} (A);  
     \draw [arrow] (A) -- node[anchor=west] {f^{(n+1)(k+1)}_{n(k+1)}} (B); 
     \draw [arrow] (stop) -- node[anchor=south] {f^{(n+1)(k+1)}_{(n+1)k}} (B);  
\end{tikzcd}
\caption{Construction of $f^{mn}_{kl}.$}
\end{figure}

Let $\mathbb{L}_n=\iLim\{A_{nk}, f^{n(k+1)}_{nk}\}$.  Note that, because the epimorphisms $f^{n(k+1)}_{nk}$ are doubling edge maps, each edge in $A_{nn}$ corresponds to a line segment in the topological realization of $\mathbb{L}_n$. Thus the topological realization of $\mathbb{L}_n$ is the continuum obtained by replacing the vertices and edges in $A_n$ with points and line segments. Let $f^{nk}\colon \mathbb{L}_n\to A_{nk}$,
 $k\geq n$, be the projection maps, i.e. maps for which we have $f^{nk}= f^{n(k+1)}_{nk} \circ f^{n(k+1)}$. Let $g_n\colon \mathbb{L}_{n+1}\to \mathbb{L}_n$ be maps such that for every $n< k$, it holds  
 $f^{nk}\circ g_n=f^{(n+1)k}_{nk} \circ f^{(n+1)k}$.
Denote by $\mathbb{L}$ the inverse limit of $\iLim\{\mathbb{L}_n, g_n\}$.
Then $\mathbb{L}$ is also the inverse limit of the directed system 
$\{A_{kl}, f^{mn}_{kl}\colon A_{mn}\to A_{kl}\}$. The inverse sequence $\{A_n,f_n\}$ is cofinal in that directed system, so $\iLim\{A_n,f_n\}=\iLim\{\mathbb{L}_n, g_n\}$.

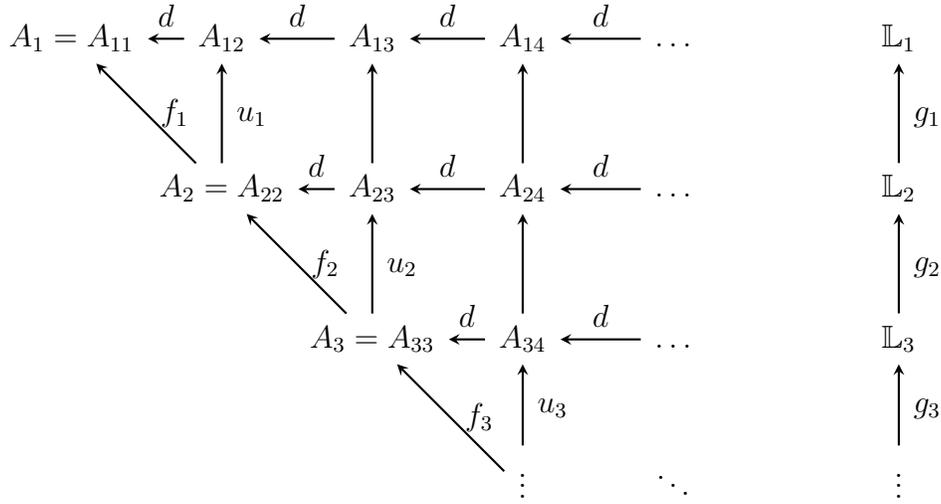
\begin{figure}[!ht]
\begin{tikzcd}
\tikzstyle{arrow} = [thick,<-,>=stealth]
 \node at (0,0) (11) {A_1=A_{11}};
 \node at (2,0) (12) {A_{12}};
  \node at (4,0) (13) {A_{13}};
   \node at (6,0) (14) {A_{14}};
 \node at (8,0) (d1) {\ldots};
  \node at (11,0) (l1) {\mathbb{L}_1};
  
 \node at (2,-2) (22) {A_2=A_{22}};
  \node at (4,-2) (23) {A_{23}};
   \node at (6,-2) (24) {A_{24}};
 \node at (8,-2) (d2) {\ldots};
  \node at (11,-2) (l2) {\mathbb{L}_2};

    \node at (4,-4) (33) {A_3=A_{33}};
   \node at (6,-4) (34) {A_{34}};
 \node at (8,-4) (d3) {\ldots};
  \node at (11,-4) (l3) {\mathbb{L}_3};

     \node at (6,-6) (A) {\vdots};
      \node at (8,-6) (B) {\ddots};
      \node at (11,-6) (C) {\vdots};

   \draw [arrow] (11) -- node[anchor=south] {d} (12);
    \draw [arrow] (12) -- node[anchor=south] {d} (13);
     \draw [arrow] (13) -- node[anchor=south] {d} (14);
      \draw [arrow] (14) -- node[anchor=south] {d} (d1);
       \draw [arrow] (22) -- node[anchor=south] {d} (23);
     \draw [arrow] (23) -- node[anchor=south] {d} (24);
      \draw [arrow] (24) -- node[anchor=south] {d} (d2);
        \draw [arrow] (33) -- node[anchor=south] {d} (34);
      \draw [arrow] (34) -- node[anchor=south] {d} (d3);

       \draw [arrow] (12) -- node[anchor=west] {u_1} (22); 
         \draw [arrow] (23) -- node[anchor=west] {u_2} (33); 
         \draw [arrow] (34) -- node[anchor=west] {u_3} (A); 
          \draw [arrow] (11) -- node[anchor=west] {f_1} (22);
          \draw [arrow] (22) -- node[anchor=west] {f_2} (33);
          \draw [arrow] (33) -- node[anchor=west] {f_3} (A);
        \draw [arrow] (13) --  (23); 
         \draw [arrow] (14) --  (24); 
          \draw [arrow] (24) -- (34); 
          \draw [arrow] (l1) -- node[anchor=west] {g_1} (l2);
          \draw [arrow] (l2) -- node[anchor=west] {g_2} (l3);
          \draw [arrow] (l3) -- node[anchor=west] {g_3} (C);
\end{tikzcd}
\caption{Constructing the Mohler-Nikiel universal dendroid.}
\end{figure}

Now pass to the topological realizations and get
$|\mathbb{L}|=\iLim \{|\mathbb{L}_n|, g^*_n\}$, where $g^*_n$ is the quotient map of $g_n$.

Note that if $u_n$ takes an edge $v\in E(A_{nn})$ to an edge $w\in E(A_{n(n+1)})$, then $g^*_n $ takes the line segment corresponding to $v$ to the line segment corresponding to $w$.

The tree $|\mathbb{L}_1|$ is obtained by taking two disjoint copies of [0,1] and identifying them at 0. 
The tree $|\mathbb{L}_2|$ is obtained from $|\mathbb{L}_1|$ by following the MN-construction for $(\frac{2}{3}, \frac{2}{3}, \frac{1}{3}, \frac{1}{3}, 0)$, and $g^*_1$ is the corresponding MN-map. Continuing this, for each $n$,
the $|\mathbb{L}_n|$ is the rooted tree
obtained from [0,1] by the MN-construction for the non-increasing sequence  $s_n$ in which the number 
$\frac{k}{3^{n-1}}$, for each $k=0,1,\ldots, 3^{n-1}-1$, occurs exactly $n$ many times.
The map $g^*_n\colon |\mathbb{L}_{n+1}|\to |\mathbb{L}_n|$ is the MN-map from $s_n$ to $s'_{n+1}$, where $s'_{n+1}$ is the reordering of $s_{n+1}$ in which $s_n$ is an initial segment followed by all numbers occurring in $s_{n+1}$ and not in $s_n$ ordered in the non-increasing way (taking into account the repetitions).
So $|\mathbb{L}|$ is the Mohler-Nikiel universal dendroid.
\end{proof}

\section{Open problems}
As far as the authors know the continuum $|\mathbb{T}_{\mathcal{ C}}|$ has not been studied.  Some open questions/problems are the following.
\begin{question}
Is $|\mathbb{T}_{\mathcal{ C}}|$ universal in the class of smooth dendroids?
\end{question}

The homeomorphism group $ H(|\mathbb{T}_{\mathcal{ C}}|)$ acts on $|\mathbb{T}_{\mathcal{ C}}|$ via $(h,x)
\mapsto h(x)$. Given $x \in |\mathbb{T}_{\mathcal{ C}}|$, let ${\rm Orb}_{|\mathbb{T}_{\mathcal{ C}}|}(x) =\{h(x) \colon h \in H(|\mathbb{T}_{\mathcal{ C}}|)\}$   be the orbit of $x$ under this action and define the {\it homogeneity degree} of $|\mathbb{T}_{\mathcal{ C}}|$ to be the number of orbits. A continuum having homogeneity degree equal to $n$ is known as a {\it $\frac{1}{n}$-homogeneous} continuum. If $\pi\colon \mathbb{T}_{\mathcal{ C}}\to |\mathbb{T}_{\mathcal{ C}}|$ is the quotient map, we let $r=\pi(r_{\mathbb{T}_{\mathcal{ C}}})$.  We would like to understand orbits of this action of $ H(|\mathbb{T}_{\mathcal{ C}}|)$  on $|\mathbb{T}_{\mathcal{ C}}|$.

\begin{question}
\begin{enumerate}
\item  Is the set of endpoints of $|\mathbb{T}_{\mathcal{ C}}|$ a single orbit?
\item Is ${\rm Orb}_{|\mathbb{T}_{\mathcal{ C}}|}(r)=\{r\}$?
\item What is the  homogeneity degree of $|\mathbb{T}_{\mathcal{ C}}|$?  
\item Are all orbits (possibly excluding the orbit of $r$) dense?

\end{enumerate}
\end{question}
In fact, at the moment we cannot exclude the possibility that  $|\mathbb{T}_{\mathcal{ C}}|$  is a $\frac{1}{3}$-homogeneous continuum.
Hoehn and Pacheco Ju\' arez \cite{HJ} explicitly ask whether there exists a smooth dendroid having homogeneity degree equal to 3, different from a dendrite or a fan. In Theorem~10 of their article they provide several properties that any such dendroid must have.

Since $|\mathbb{T}_{\mathcal{ C}}|$ is a smooth dendroid, the set $E(|\mathbb{T}_{\mathcal{ C}}|)$ of endpoints is a $G_\delta$ set. Indeed, the set of endpoints is the complement of the projection onto the first coordinate of the set $\{(x,y)\colon x<y\}=\bigcup_n \{(x,y)\colon x\leq y \text{ and } d(x,y)\geq \frac{1}{n}\}\subseteq |\mathbb{T}_{\mathcal{ C}}|^2$. The smoothness means that $\{(x,y)\colon x\leq y \}$ is closed. Therefore $\{(x,y)\colon x<y\}$ is $F_\sigma$, and its complement is a $G_\delta$ set. Hence $E(|\mathbb{T}_{\mathcal{ C}}|)$ is a Polish space and it is comeager in $|\mathbb{T}_{\mathcal{ C}}|$.

The set of endpoints of Lelek fan also forms a comeager set, is one-dimensional, almost zero-dimensional, totally disconnected, homogeneous, and in fact it is homeomorphic to the complete Erdös space, see \cite{KOT}. This motivates Question \ref{endpo}(1).

It is not hard to see that we can embed continuum many disjoint triods in the Mohler-Nikiel universal dendroid. Therefore it is not embeddable in the plane $\mathbb{R}^2$, see \cite{Moore}. We expect that the same is true for the continuum $|\mathbb{T}_{\mathcal{ C}}|$.  Since $|\mathbb{T}_{\mathcal{ C}}|$ is one-dimensional, it is embeddable in  $\mathbb{R}^3$. This suggest Question \ref{endpo}(2). 
\begin{question}\label{endpo}
\begin{enumerate}
\item What is the topological dimension of $E(|\mathbb{T}_{\mathcal{ C}}|)$? Is $E(|\mathbb{T}_{\mathcal{ C}}|)$ totally disconnected? Is it almost zero-dimensional?
\item Is $E(|\mathbb{T}_{\mathcal{ C}}|)$ embeddable in $\mathbb{R}^2$?
\end{enumerate}
\end{question}

\begin{problem}
Find a topological characterization for $|\mathbb{T}_{\mathcal{ C}}|$.
\end{problem}

A characterization to the Molher-Nikiel universal dendroid could allow a simpler proof that $|\mathbb{T}_{\mathcal{ CE}}|$ is homeomorphic to the Molher-Nikiel universal dendroid.
\begin{problem}
    Find a topological characterization of the Molher-Nikiel universal dendroid.
\end{problem}

\bigskip

{\bf{Acknowledgements.}}
We would like to thank Alessandro Codenotti for helpful comments on an earlier version of this preprint.


\begin{thebibliography}{10}





\bibitem{Bar-Kub} A.  Bartoš, W. Kubi\'s,
\textit{Hereditarily indecomposable continua as generic mathematical structures}, 	arXiv:2208.06886, 2022.

\bibitem{B-K} D. Bartošová and A. Kwiatkowska, \textit{Lelek fan from a projective Fraïssé limit}, Fund. Math. 231 (2015), no. 1, 57--79.

\bibitem{B-C} G. Basso and R. Camerlo, \textit{Fences, their endpoints, and projective Fra\"{\i}ss\'{e} limit}, Trans. Amer. Math. Soc.  374 (2021), 4501--4535.

\bibitem{CP} J. J. Charatonik,  J. R. Prajs,
{\it On lifting properties for confluent mappings},
Proc. Amer. Math. Soc. 133 (2005), no. 2, 577--585.

\bibitem{CCP} J. J. Charatonik, W. J. Charatonik, J.R. Prajs, 
{\it Hereditarily unicoherent continua and their absolute retracts},
Rocky Mountain J. Math. 34 (2004), no. 1, 83--110.

\bibitem{JJC-Wazewski} J. J. Charatonik, {\it Monotone mappings of universal dendrites},
Topology Appl. 38 (1991), no. 2, 163--187.

\bibitem{JJC-Confluent}  J. J. Charatonik, \textit{Confluent mappings and unicoherence of continua}, Fund. Math. 56 (1964), 213--220.


\bibitem{CKR} W. J. Charatonik, A. Kwiatkowska, R. P. Roe, {\it The projective Fra\"{\i}ss\'e limit of the class of all connected finite graphs with confluent epimorphisms},  arXiv:2206.12400. 

\bibitem{Lelek-fan} W. J. Charatonik, \textit{The Lelek fan is unique}, Houston J. Math. 15 (1989), 27--34.

\bibitem{Co-Kw} A. Codenotti and A. Kwiatkowska,
\textit{Projective Fraïssé limits and generalized Ważewski dendrites},  arXiv:2210.06899, accepted to Fundamenta Mathematicae. 

\bibitem{Czuba} S. T. Czuba, {\it On Dendroids with Kelley's Property},
Proc. Amer. Math. Soc., Vol. 102, No. 3 (Mar., 1988), pp. 728--730. 

\bibitem{Du} B. Duchesne, {\it Topological properties of Ważewski dendrite groups}, 
J. Éc. polytech. Math. 7 (2020), 431--477.

\bibitem{HJ} L. Hoehn, and Y. Pacheco Ju\' arez, \textit{Dendroids with a low homogeneity degree},
Topology Appl. 277 (2020), 107218 

\bibitem{Ingram} W. T. Ingram, \textit{Inverse limits and a Property of J. L. Kelley, II}, Bol. Soc. Mat. Mexicana, 9 (2003), 135--150.

\bibitem{Pseudo} T. Irwin and S. Solecki, \textit{Projective Fra\"{\i}ss\'e Limits and the Pseudo-arc}, Tran. Amer. Math. Soc.  358, (2006), 3077--3096.


\bibitem{SI} S. Iyer, \textit{The homeomorphism group of the universal Knaster continuum},  arXiv:2208.02461, 2022.

\bibitem{KOT} K. Kawamura, L. G. Oversteegen, E. Tymchatyn, \textit{On homogeneous totally disconnected 1-dimensional spaces},
Fund. Math.150(1996), no.2, 97--112.

\bibitem{Knaster} B. Knaster, \textit{Un continue don't tout sous-continu est ind\' ecomposable},
Fund. Math. 3 (1922), 247--286.

\bibitem{K-D3} A. Kwiatkowska, \textit{Universal minimal flows of generalized Wa\. zewski dendrites}, J. Symb. Log.  83, (2018), 1618--1632.

\bibitem{kubis} W.  Kubiś and A. Kwiatkowska, \textit{The Lelek fan and the Poulsen simplex as Fraïssé limits}, Rev. R. Acad. Cienc. Exactas Fís. Nat. Ser. A Mat. RACSAM 111 (2017), no. 4, 967--981.

\bibitem{Moore} R. L. Moore, \textit{Concerning Triods in the Plane and the Junction Points of Plane Continua}, Proceedings of the National Academy of Sciences of the United States of
America, 1928, Vol. 14, No. 1, pp. 85--88.

\bibitem{Nadler} S. Nadler, \textit{Continuum Theory: An Introduction}, Chapman \& Hall/CRC Pure and Applied Mathematics, CRC Press, 2017.

\bibitem{Dendroid} L. Mohler and J. Nikiel, \textit{A universal smooth dendroid answering a question of J. Krasinkiewicz},
Houston J. Math. 14 (1988), no. 4, 535--541.

\bibitem{OP}  L. Oversteegen, J. R. Prajs, 
{\it On confluently graph-like compacta}, Fund. Math.178(2003), no.2, 109--127.

\bibitem{Menger} A. Panagiotopoulos and S. Solecki, \textit{A combinatorial model for the Menger curve}, Journal of Topology and Analysis, 14(01), (2022) 203--229. 

\bibitem{LW} L. Wickman, \textit{Projective Fra\"{\i}ss\'e theory and Knaster continua}, Ph.D. Thesis, University of Florida, 2022.

\end{thebibliography}
\end{document}